\documentclass[a4paper, 11pt, english]{article}

\usepackage{lscape}
\usepackage{pdflscape}

\usepackage{stmaryrd}

\usepackage{cmll}

\usepackage{comment}

\usepackage{amscd}
\usepackage{amssymb,amsmath,amsthm}
\usepackage{mathptmx}
\usepackage{mathrsfs}
\usepackage{color}
\usepackage{xspace}
\usepackage{bussproofs}
\EnableBpAbbreviations

\usepackage{tikz}

  \usepackage{txfonts}

\usepackage{bpextra}

\usepackage{enumerate}

\usepackage{centernot}

\usepackage{mathtools}

\usepackage{hyperref}

\usepackage{cleveref}

\usepackage{relsize}

\usepackage{caption}

\usepackage{multirow}

\usepackage{multicol}

\usepackage{array}
\newcolumntype{C}[1]{>{\centering\arraybackslash}p{#1}}
\newcolumntype{L}[1]{>{\arraybackslash}p{#1}}

\usepackage[left=3cm,top=3cm,right=3cm,bottom=2cm]{geometry}
\newtheorem{theorem}{Theorem}[section]
\newtheorem{corollary}[theorem]{Corollary}
\newtheorem{lemma}[theorem]{Lemma}
\newtheorem{proposition}[theorem]{Proposition}

\newtheorem{fact}[theorem]{Fact}

\newtheorem{definition}[theorem]{Definition}

\def\fCenter{{\mbox{$\ \vdash\ $}}}

\newcommand{\fns}{\footnotesize}

\newcommand{\marginnote}[1]{\marginpar{\raggedright\tiny{#1}}}

\def\mc{\multicolumn}

\newcommand{\mneg}{\ensuremath{^{\bot}}\xspace}
\newcommand{\mtop}{\ensuremath{1}\xspace}
\newcommand{\mand}{\ensuremath{\otimes}\xspace}

\newcommand{\mrarr}{\ensuremath{\multimap}\xspace}
\newcommand{\mbot}{\ensuremath{\bot}\xspace}
\newcommand{\mor}{\ensuremath{\invamp}\xspace}

\newcommand{\mdrarr}{\ensuremath{\multimapdot}\xspace}
\newcommand{\MNEG}{\ensuremath{\ast}\xspace}
\newcommand{\MTOPBOT}{\ensuremath{\Phi}\xspace}
\newcommand{\MANDOR}{\ensuremath{\,,}\xspace}
\newcommand{\MARR}{\ensuremath{\gg}\xspace}

\newcommand{\aatop}{\ensuremath{\top}\xspace}
\newcommand{\aand}{\ensuremath{\with}\xspace}
\newcommand{\bigaand}{\ensuremath{\bigwith}\xspace}
\newcommand{\ararr}{\ensuremath{\,{\raisebox{-0.065ex}{\rule[0.5865ex]{1.38ex}{0.1ex}}\mkern-1mu\gtrdot}\,}\xspace}

\newcommand{\abot}{\ensuremath{0}\xspace}
\newcommand{\aor}{\ensuremath{\oplus}\xspace}
\newcommand{\bigaor}{\ensuremath{\bigoplus}\xspace}

\newcommand{\adrarr}{\ensuremath{\,{\gtrdot\mkern-3mu\raisebox{-0.065ex}{\rule[0.5865ex]{1.38ex}{0.1ex}}}\,}\xspace}

\newcommand{\AANDOR}{\ensuremath{\cdot}\xspace}
\newcommand{\ATOPBOT}{\ensuremath{\textrm{I}}\xspace}
\newcommand{\AARR}{\ensuremath{\gtrdot}\xspace}

\newcommand{\ktopk}{\ensuremath{\texttt{t}}\xspace}
\newcommand{\kbotk}{\ensuremath{\texttt{f}}\xspace}
\newcommand{\kandk}{\ensuremath{\wedge}\xspace}
\newcommand{\bigkandk}{\ensuremath{\bigwedge}\xspace}
\newcommand{\kork}{\ensuremath{\vee}\xspace}
\newcommand{\bigkork}{\ensuremath{\bigvee}\xspace}
\newcommand{\krarrk}{\ensuremath{\rightarrow}\xspace}
\newcommand{\kdrarrk}{\ensuremath{\,{>\mkern-7mu\raisebox{-0.065ex}{\rule[0.5865ex]{1.38ex}{0.1ex}}}\,}\xspace}

\newcommand{\KTOPBOTK}{\ensuremath{\copyright}\xspace}
\newcommand{\KANDORK}{\ensuremath{\,;\,}\xspace}
\newcommand{\KARRK}{\ensuremath{>}\xspace}

\newcommand{\kneg}{\ensuremath{\neg_\oc}\xspace}
\newcommand{\ktop}{\ensuremath{\texttt{t}_\oc}\xspace}
\newcommand{\kbot}{\ensuremath{\texttt{f}_\oc}\xspace}
\newcommand{\kand}{\ensuremath{\wedge_\oc}\xspace}
\newcommand{\kor}{\ensuremath{\vee_\oc}\xspace}
\newcommand{\krarr}{\ensuremath{\rightarrow_\oc}\xspace}
\newcommand{\kdrarr}{\ensuremath{\,{>\mkern-7mu\raisebox{-0.065ex}{\rule[0.5865ex]{1.38ex}{0.1ex}}_\oc}\,}\xspace}

\newcommand{\KTOPBOT}{\ensuremath{\copyright_\oc}\xspace}
\newcommand{\KANDOR}{\ensuremath{\,;_\oc\,}\xspace}
\newcommand{\KARR}{\ensuremath{>_\oc}\xspace}

\newcommand{\topk}{\ensuremath{\texttt{t}_\wn}\xspace}
\newcommand{\botk}{\ensuremath{\texttt{f}_\wn}\xspace}
\newcommand{\andk}{\ensuremath{\wedge_\wn}\xspace}
\newcommand{\ork}{\ensuremath{\vee_\wn}\xspace}
\newcommand{\rarrk}{\ensuremath{\rightarrow_\wn}\xspace}
\newcommand{\drarrk}{\ensuremath{\,{>\mkern-7mu\raisebox{-0.065ex}{\rule[0.5865ex]{1.38ex}{0.1ex}}_\wn}\,}\xspace}

\newcommand{\TOPBOTK}{\ensuremath{\copyright_\wn}\xspace}
\newcommand{\ANDORK}{\ensuremath{\,;_\wn\,}\xspace}
\newcommand{\ARRK}{\ensuremath{>_\wn}\xspace}

\newcommand{\KKBAND}{\ensuremath{\ \rotatebox[origin=c]{90}{$\blacktriangleright$}\ \xspace}}

\newcommand{\KKRARR}{\ensuremath{\vartriangleright}\xspace}
\newcommand{\KKOR}{\ensuremath{\,\triangledown\,}\xspace}

\newcommand{\KKLARR}{\ensuremath{\vartriangleleft}\xspace}
\newcommand{\KKBARR}{\ensuremath{\blacktriangleright}\xspace}
\newcommand{\KKBLARR}{\ensuremath{\blacktriangleleft}\xspace}

\newcommand{\KKBand}{\ensuremath{\,\rotatebox[origin=c]{90}{$\triangleright$}\,\xspace}}

\newcommand{\KKor}{\ensuremath{\,\rotatebox[origin=c]{90}{$\triangleleft$}\,\xspace}}

\newcommand{\KKrarr}{\ensuremath{\,{\raisebox{-0.065ex}{\rule[0.5865ex]{1.38ex}{0.1ex}}\mkern-1mu\rhd}\,}\xspace}

\newcommand{\KKlarr}{\ensuremath{\,{\lhd\mkern-1.2mu\raisebox{-0.065ex}{\rule[0.5865ex]{1.38ex}{0.1ex}}}\,}\xspace}

\newcommand{\KKBrarr}{\ensuremath{\,{\raisebox{-0.065ex}{\rule[0.5865ex]{1.38ex}{0.1ex}}\mkern-1mu{\blacktriangleright}}\,}\xspace}

\newcommand{\KKBlarr}{\ensuremath{\,{{\blacktriangleleft}\mkern-1.2mu\raisebox{-0.065ex}{\rule[0.5865ex]{1.38ex}{0.1ex}}}\,}\xspace}

\newcommand{\WCIRCW}{\ensuremath{\circ}\xspace}
\newcommand{\BCIRCB}{\ensuremath{\bullet}\xspace}
\newcommand{\wboxw}{\ensuremath{\Box}\xspace}
\newcommand{\wdiaw}{\ensuremath{\Diamond}\xspace}
\newcommand{\bboxb}{\ensuremath{\blacksquare}\xspace}
\newcommand{\bdiab}{\ensuremath{\Diamondblack}\xspace}

\newcommand{\WCIRC}{\ensuremath{\circ_\oc}\xspace}
\newcommand{\BCIRC}{\ensuremath{\bullet_\oc}\xspace}

\newcommand{\wdia}{\ensuremath{\Diamond_\oc}\xspace}
\newcommand{\bbox}{\ensuremath{\blacksquare_\oc}\xspace}

\newcommand{\CIRCW}{\ensuremath{\circ_\wn}\xspace}
\newcommand{\CIRCB}{\ensuremath{\bullet_\wn}\xspace}
\newcommand{\boxw}{\ensuremath{\Box_\wn}\xspace}

\newcommand{\diab}{\ensuremath{\Diamondblack_\wn}\xspace}

\newcommand{\pandp}{\ensuremath{\cap}\xspace}

\newcommand{\PANDORP}{\ensuremath{\centerdot}\xspace}

\newcommand{\pwdiawp}{\ensuremath{\Diamond}\xspace}
\newcommand{\pbdiabp}{\ensuremath{\Diamondblack}\xspace}
\newcommand{\pbboxbp}{\ensuremath{\blacksquare}\xspace}
\newcommand{\pwboxwp}{\ensuremath{\Box}\xspace}
\newcommand{\PWCIRCWP}{\ensuremath{\circ}\xspace}
\newcommand{\PBCIRCBP}{\ensuremath{\bullet}\xspace}

\newcommand{\pand}{\ensuremath{\cap}\xspace}

\newcommand{\PANDOR}{\ensuremath{\centerdot}\xspace}

\newcommand{\pwbox}{\ensuremath{\Box}\xspace}

\newcommand{\pbdia}{\ensuremath{\Diamondblack}\xspace}
\newcommand{\PWCIRC}{\ensuremath{\circ}\xspace}
\newcommand{\PBCIRC}{\ensuremath{\bullet}\xspace}

\usetikzlibrary{arrows}
\usetikzlibrary{matrix}
\usetikzlibrary{patterns}
\usetikzlibrary{shapes}

\makeindex

\title{Linear Logic Properly Displayed}

\author{Giuseppe Greco \and
Alessandra Palmigiano\thanks{This research is supported by the NWO Vidi grant 016.138.314, the NWO Aspasia grant 015.008.054, and a Delft Technology Fellowship awarded to the second author in 2013.}
}
\date{}

\begin{document}

\maketitle

\begin{abstract}
We introduce proper display calculi for intuitionistic, bi-intuitionistic and classical linear logics with exponentials, which are sound, complete, conservative, and enjoy cut-elimination and subformula property. Based on the same design, we introduce a variant of Lambek calculus with exponentials, aimed at capturing the controlled application of exchange and associativity. {\em Properness} (i.e.~closure under uniform substitution of all parametric parts in rules) is the main interest and added value of the present proposal, and allows for the smoothest proof of cut-elimination.   Our proposal builds on an algebraic and order-theoretic analysis  of linear logic, and applies the guidelines of the multi-type methodology in the design of display calculi.
{\em 2010 Math.~Subj.~Class.} 03F52, 03F05, 03G10.
\end{abstract}

\tableofcontents

\section{Introduction}

Linear Logic \cite{Girard1987} is one of the best known substructural logics, and the best known example
of a resource-sensitive logic. The fact that  formulas are treated as resources implies that e.g.~two copies of a given assumption $A$ guarantee a different proof-power than three copies, or one copy.  From an algebraic perspective, this fact translates in the stipulation that linear conjunction and disjunction, denoted $\mand$ and $\mor$ respectively, are {\em not} idempotent, in the sense that none of the inequalities composing the identities $a\mand a = a = a\mor a$ are valid in linear algebras.  From a proof-theoretic perspective, this fact translates into the well known stipulation that unrestricted (left- and right-) weakening and contraction rules cannot be included in Gentzen-style presentations of linear logic.

However,  resources might exist which are available {\em unlimitedly}; that is, having one or more copies of these special resources  guarantees the same proof-power. To account for the co-existence of general and unlimited resources, the language of linear logic includes, along with the connectives $\mand$ and $\mor$ (sometimes referred to as the {\em multiplicative} conjunction and disjunction), also {\em additive} conjunction and disjunction, respectively denoted $\aand$ and $\aor$, which are idempotent (i.e.~$A\aand A = A = A\aor A$ for every formula $A$). Moreover, the language of linear logic includes the modal operators $\oc$ and $\wn$, called {\em exponentials}, which respectively govern the controlled application of left- and right-weakening and contraction rules for formulas under their scope,
 algebraically encoded by the following identities: 
\[ \oc (A\aand B) = \oc A \mand \oc B\quad \quad \wn (A\aor B) = \wn A\mor \wn B.\]
The interplay between additive and multiplicative connectives, mediated by exponentials, is the main hurdle to a smooth proof-theoretic treatment of linear logic. Indeed, in the 
extant formalisms, this interplay is  encoded by means of rules   the parametric parts (or contexts) of which are not arbitrary, and hence closed under arbitrary substitution, but are {\em restricted} in some way. These restricted contexts create additional complications in the definition of smooth and general reduction strategies for syntactic cut-elimination.

In the present paper, proof calculi for intuitionistic and classical linear logics are introduced in which all parameters in rules occur {\em unrestricted}. This is possible thanks to the introduction of a {\em richer} language in which general and unlimited resources are assigned different {\em types}, each of which  interpreted by a different type of algebra (linear algebras for general resource-type terms, and (bi-)Heyting algebras or Boolean algebras for unlimited resource-type terms), and the interaction between these types is mediated by  pairs of adjoint connectives, the composition of which captures Girard's exponentials $\oc$ and $\wn$ as defined connectives. The proof-theoretic behaviour of the adjoint connectives is that of standard normal modal operators. Moreover, the information capturing the essential properties of the exponentials can be  expressed in the new language  by means of identities of a syntactic shape called {\em analytic inductive} (cf.~\cite{GMPTZ}), which guarantees that they can be equivalently encoded into analytic rules. The metatheory of these calculi is smooth and encompassed in a general theory (cf.~\cite{GMPTZ,CiabattoniRamanayake2016,TrendsXIII}), so that one obtains soundness, completeness, conservativity, cut-elimination, and subformula property as easy corollaries of general facts. Moreover, the same general theory guarantees that these meta-properties transfer to all analytic  variants of the basic framework, such as non-commutative, affine, or relevant linear logic.

The calculi introduced in the present paper are designed according to the {\em multi-type} methodology, introduced in \cite{Multitype,PDL} to provide DEL and PDL with analytic calculi based on motivations discussed in \cite{GKPLori,GAV}, and further developed in \cite{TrendsXIII,BGPTW,Inquisitive,GrecoPalmigianoMultiTypeAlgebraicProofTheory}. The multi-type  methodology  refines and generalizes {\em proper display calculi} (cf.~\cite[Section 4.1]{Wansing1998}) so as to extend the benefits of their meta-theory to logics, such as linear logic, inquisitive logic, DEL and PDL, which are not properly displayable  in their single-type presentation, according to the characterization results given in \cite{Kracht1996,GMPTZ,CiabattoniRamanayake2016}. The possibility of appealing to the general multi-type methodology in the specific case of linear logic is justified by an  analysis of the algebraic  semantics  of linear logic, aimed at  identifying the different types and  their key interactions. 

\paragraph{Structure of the paper.} In Section \ref{sec:survey}, we highlight the ideas on which the approach of the present paper is based as they have occurred in the literature in linear logic and in the proof theory of neighbouring logics. Section \ref{sec: semantic environment} outlines  the algebraic and order-theoretic environment  which motivates our proposed treatment for the various linear logics which are also presented algebraically in this section. In Section \ref{sec: multi-type language}, we introduce the syntactic counterpart of the multi-type semantic environment introduced in Section \ref{sec: semantic environment}, in the form of a multi-type language,  define translations between the single-type and the multi-type languages for linear logics, and discuss how equivalent analytic reformulations can be given of non-analytic axioms in linear logic. In Section \ref{sec:ProperDisplayCalculiForLinearLogics}, we introduce  calculi for the various  linear logics. In Section \ref{sec:properties}, we prove their soundness, syntactic completeness, conservativity, cut-elimination and subformula property. In Section \ref{sec:structural control}, we apply the same techniques developed in the previous sections to the problem, pertaining to the field of categorial grammar, of developing flexible enough formal frameworks  to  account for exceptions in rules for the generation of grammatical sentences in natural language. Specifically, we introduce exponentials which control exchange and associativity in the same way in which exponentials in linear logic control weakening and contraction. We discuss conclusions and further directions in Section \ref{sec:conclusions}. Given that the present paper draws from many areas, we believe it might be convenient for the reader to include several appendices, each of which contains the definitions and facts from a given area which are used in the paper. Specifically, Appendix \ref{appendix: cut elim} reports on the definition of proper multi-type calculi and their meta-theorem; Appendix \ref{sec:analytic inductive ineq} on the definition of analytic inductive inequalities; Appendix \ref{sec: background can ext} on basic definitions and facts on canonical extensions. Appendix \ref{sec: inversion lemmas} collects the proof of the inversion lemmas which are needed for verifying the completeness. Various derivations are collected in Appendixes \ref{sec: deriving axioms}, \ref{sec:some derivations} and \ref{sec:sneak}.

\section{Precursors of the multi-type approach to linear logic}\label{sec:survey}
The multi-type methodology is aimed at developing  analytic calculi with a uniform design and excellent properties for general classes of logics, but as we will argue in the present paper, it is particularly natural for linear logic, given that precisely the proof-theoretic, algebraic and category-theoretic methods developed for linear logic have been the growth-bed for many key insights and ideas the multi-type methodology builds on. In the present section, we review some of these solutions with the specific aim of highlighting those aspects that  anticipate the multi-type approach. Our presentation is certainly not an exhaustive survey of the literature, for which we refer to \cite{Troelstra1992,Mellies2009,Benton1994}.

Right from the first paper \cite{Girard1987}, Girard describes the connectives of linear logic as arising from the {\em decomposition} of classical connectives; this decomposition makes it possible to isolate e.g.~the constructive versus the non-constructive behaviour of connectives, but also the linear versus the non-linear behaviour, using structural rules to capture each type of behaviour. In \cite{Girard1995}, Girard further expands on a conceptual framework of reference for these ideas, and in particular describes the linear behaviour as the behaviour of general actions, which can be performed at a cost, and the non-linear behaviour in terms of the behaviour of situations, or actions which can be performed at no cost (or the cost of which is negligible). The purpose of the exponential connectives is then to bridge the linear and the non linear behaviour, thus making the language of linear logic expressive enough to support a G\"odel-type translation of intuitionistic/classical logic into linear logic.

While exponentials are essential precisely for embedding intuitionistic and classical logic into linear logic, they pose significant complications when it comes to cut-elimination and normalization, as is witnessed in the so-called Girard-Tait \cite{girard87,tait} approach to cut-elimination for Girard's calculus, discussed also by Melli\`es \cite{Mellies2009}. Indeed, due to the fact that exponentials have non-standard introduction rules (the so-called promotion and dereliction rules), ad-hoc transformation steps need to be devised to account for the permutation of cuts on $\oc$- or $\wn$-formulas with restricted weakening and contraction and with promotion, giving rise to cut elimination proofs which, besides being lengthier and more cumbersome, are not structured in such a way that the same uniform reduction strategy applies.

Motivated both by these technical issues and by more general questions about ways of making different logics coexist and interact with one other, in \cite{girard:unity} Girard introduces a calculus for sequents $\Gamma; \Gamma'\vdash \Delta'; \Delta$ with different {\em maintenance zones}: the  classical zone $\Gamma'$ and $\Delta'$ where all structural rules are applicable, and the linear zones $\Gamma$ and $\Delta$, where the application of structural rules is limited according to the stipulations of the linear behaviour. The intended reading of the sequent above in terms of the linear logic language is then $\Gamma, \oc\Gamma'\vdash \wn\Delta', \Delta$. A very similar environment is Andreoli's calculus of dyadic sequents \cite{Andreoli1992}, introduced to address the issues of cut elimination and proof search.
The strategy underlying these solutions aims at enriching the language of the calculus in a way that accounts at the {\em structural} level for a neat separation of the linear behaviour from the non linear behaviour. 

Although developed independently, Girard's and Andreoli's approach has strong similarities with Belnap's general proof-theoretic paradigm of {\em display logic} \cite{Belnap1982}. What Belnap refers to as display logic is not in fact one logical system presented as a sequent calculus, but rather a {\em methodology} refining Gentzen sequent calculi through the systematic enrichment of the logical language with an additional layer of structural connectives besides the comma. This richer environment makes it possible to  enforce a neat separation of roles between {\em introduction rules} (i.e.~rules introducing principal formulas) and {\em structural rules} (i.e.~those rules expressed purely in terms of structural variables and structural connectives). Indeed, while introduction rules are defined along very uniform and rigid lines and only capture the most basic information on  each logical connective (the polarity of each coordinate),
structural rules encode most information on the behaviour of each connective and on the interaction between different connectives. Thanks to this neat separation of roles, a suitable environment is created in which the most important technical contribution of Belnap's paradigm can be stated and proved; namely, a general, smooth and robust proof strategy for cut-elimination, hinging  on the design principles we summarize as follows: (a) uniform  shape of introduction rules for all connectives, and (b) information on the specific behaviour of each connective encoded purely at the structural level. These ideas are very much aligned with Girard's considerations about the crucial role of structural rules (e.g.\ in telling the linear and non-linear behaviour apart)  in the genesis of linear logic, and succeed in creating an explicit mathematical environment in which they can be developed further.

The  alignment of Belnap's and Girard's programs is also reflected in Belnap's major motivation in introducing display calculi: creating a proof-theoretic environment capable of simultaneously accounting for ``an indefinite number of logics all mixed together including boolean [...], intuitionistic, relevance and (various) modal logics.'' This motivation drives Belnap's bookkeeping mechanism  ``permitting control in the presence of multiple logics'' which is remarkable in its elegance and power, and is based on the fact that (structural) connectives in all these logics are {\em residuated}. The engine of this bookkeeping mechanism is the property that gives display logic its name, namely the {\em display property}, requiring that for any (derivable) sequent $X\vdash Y$ and any substructure $Z$ of $X$ or $Y$, the sequent $X\vdash Y$ is provably equivalent to either a sequent of the form $Z\vdash Y'$ or of the form $X'\vdash Z$ (and exactly one of these cases occurs). Informally, a sequent calculus enjoys the display property if it is always possible to display any substructure $Z$ of a given (derivable) sequent $X\vdash Y$ either in {\em precedent position} ($Z\vdash Y'$) or in {\em succedent position} ($X'\vdash Z$) in a way that preserves logical equivalence.

However, when designing his display calculus for linear logic \cite{Belnap1990}, Belnap derogates from  his own design principle (a) and does not introduce structural counterparts for the exponentials. Without their structural counterparts, it is of course not possible to express the key properties of exponentials purely at the structural level, which also results in a violation of (b). Hence, in Belnap's calculus, these properties are still encoded in the introduction rules for $\oc$ and $\wn$, which are hence non-standard. This latter fact is the key reason why the cut elimination theorem for Belnap's display calculus for linear logic does not make as significant an improvement in smoothness over those for previous calculi as it could have made.

Belnap's derogation has a technical reason, which stems from the fact that, when seen as operations on linear algebras (cf.~Definition \ref{def:linear algebras}), the exponentials $\oc$ and $\wn$ are not residuated. Hence, if their structural counterparts were allowed in the language of Belnap's display calculus for linear logic, the resulting calculus would lose the display property, due to the fact that this property  critically hinges on the presence of certain structural rules (the so-called {\em display rules}) for all structural connectives, which would be unsound in the case of exponentials,  precisely due to their not being residuated.
Hence, to preserve the display property for his calculus for linear logic (which is the ultimate drive for his cut-elimination metatheorem), Belnap is forced to give up on the full enforcement of design principles (a) and (b).\footnote{Similar considerations apply for Belnap's treatment of the additive connectives, which also lack their structural counterparts. We address this issue in the companion paper \cite{GrecoPalmigianoLatticeLogic}.}

The differences in Belnap's and Girard-Andreoli's approaches reflect the tradeoff we are facing: unlike Belnap, Girard-Andreoli are not bound to a general strategy for cut-elimination (such as the one hinging on Belnap's display property), which leaves them the freedom to include exponentials at the structural level, at the price of not being able to develop a cut-elimination strategy via a meta-theorem. On the other hand, it is precisely the attempt to achieve cut-elimination via a general meta-theorem which prevents Belnap from using the full power of his design principles (a) and (b), and hence forces him to settle for a less than completely smooth cut-elimination result. Both solutions make improvements, which leave open space for further improvement.

Our way out of this empasse takes its move from Girard's initial idea about linear logic arising from the {\em decomposition} of classical and intuitionistic connectives, and is based on viewing exponentials as compositions of adjoint maps. This is indeed possible, and has been already pursued in the category-theoretic setting \cite{BBHdP,Benton1994,Mellies2009}. Benton \cite{Benton1994}, in particular, takes the environment of {\em adjoint models} (each consisting of a symmetric monoidal adjunction between a symmetric monoidal closed category and a cartesian closed category) as the semantic justification for the introduction of the  logical system LNL, in which the linear and non-linear behaviour ``exist on an equal footing''. Benton explores several  sequent calculi for LNL, some of which fail to be cut-free. His final proposal is perhaps the closest precursor of the multi-type calculus which is introduced in the present paper, in that it features both  linear and  nonlinear sequents, and rules indicating how to move from one to the other. However, the proof of cut elimination does not appeal to a meta-theorem, and hence does not straightforwardly transfer to axiomatic extensions of linear logic such as classical, affine or relevant linear logic \cite{Mellies2009,dosen,jacobs}.

In a nutshell, the approach pursued in the present paper is based semantically on some well known facts, which yield  the definition of the algebraic counterparts of adjoint models. Firstly, the algebraic interpretation of $\oc$ (resp.~$\wn$) is an {\em interior operator} (resp.~a {\em closure operator}) on any linear algebra $\mathbb{L}$. Hence, by general order theory, the operation $\oc$ (resp.~$\wn$) can be identified with the composition of two adjoint maps: the  map $\iota: \mathbb{L}\twoheadrightarrow K_{\oc}$ (resp.\ $\gamma: \mathbb{L}\twoheadrightarrow K_{\wn}$) onto the range of $\oc$ (resp.~$\wn$) and the natural order-embedding $e_\oc: K_{\oc}\hookrightarrow \mathbb{L}$ (resp.~$e_\wn: K_{\wn}\hookrightarrow \mathbb{L}$). 
Secondly, the interaction between $\oc$, $\mand$ and $\aand$ (resp.\ between $\wn$, $\mor$ and $\aor$) can be equivalently rephrased by saying that the poset $K_\oc$ (resp.~$K_\wn$) has a natural algebraic structure (cf.~Proposition \ref{prop: algebraic structure on kernels}). Thirdly, the  algebraic structures of $\mathbb{L}$ and $K_{\oc}$ (resp.\  $\mathbb{L}$ and $K_{\wn}$) are in fact {\em compatible} with the adjunction $e_{\oc}\dashv \iota$ (resp.~$\gamma\vdash e_{\wn}$) in a sense which is both mathematically precise and general (cf.~proof of Proposition \ref{prop: algebraic structure on kernels}), and which  provides the  algebraic underpinning of the translation from classical and intuitionistic logic to linear logic. Interestingly, a very similar compatibility condition also accounts for the  G\"odel-Tarski translation from intuitionistic logic to the modal logic S4 (cf.~algebraic analysis in \cite[Section 3.1]{CPZ:Trans}).  

The composite mathematical environment consisting of the algebras $\mathbb{L}, K_\oc$ and $K_\wn$ together with the adjunction situations between them naturally provides the interpretation for a logical language which is {\em polychromatic} in Melli\`es' terminology \cite{Mellies2009}, in the sense that  admits terms of as  many types as there are algebras in the environment. The mathematical properties of this environment provide the semantic justification for the design of the calculi introduced in Section \ref{sec:ProperDisplayCalculiForLinearLogics}, in which both display property and design principles (a) and (b) are satisfied. The change of perspective we pursue, along with Benton \cite{Benton1994}, Melli\`es \cite{Mellies2009}, and Jacobs \cite{jacobs}, is that this environment can also be taken as {\em primary} rather than derived. Accordingly, at least for proof-theoretic purposes, linear logic can be more naturally accommodated by such a composite environment than the standard linear algebras.


\section{Multi-type semantic environment for linear logic}
\label{sec: semantic environment}
In the present section, we introduce the algebraic environment which justifies semantically the multi-type approach to linear logic which we develop in Section \ref{sec:ProperDisplayCalculiForLinearLogics}. In the next subsection, we take the algebras introduced in \cite{Ono1993,Troelstra1992} as starting point. We  enrich their definition, make it more modular, and expand on the properties of the images of the algebraic interpretation of exponentials which are briefly mentioned in \cite{Troelstra1992}, leading to our notion of (right and left) `kernels'. In Section \ref{ssec:from ono axioms to eterogeneous arrows}, we expand on the algebraic significance of the interaction axioms between exponentials in terms of the existence of certain maps between kernels. In Subsection \ref{ssec: reverse engineering}, we show that linear algebras with exponentials can be equivalently presented in terms of composite environments consisting of linear algebras without exponentials and other algebraic structures, connected via suitable adjoint maps. In Subsection \ref{ssec: canonical extensions}, we report on results pertaining to the theory of canonical extensions applied to the composite environments introduced in the previous subsections. These results will be used in the development of the next sections.

\subsection{Linear algebras and their kernels}\label{ssec: linear algebras}
\begin{definition}
\label{def:linear algebras}
$\mathbb{L} = (L, \aand, \aor, \aatop, \abot, \mand,\mor, \mtop,\mbot, \mrarr)$ is an {\em intuitionistic linear algebra} (IL-algebra) if:

\begin{enumerate}
\item[IL1.] $(L, \aand, \aor, \aatop, \abot)$ is a lattice with join $\aor$, meet $\aand$, top $\aatop$ and bottom $\abot$;
\item[IL2.] $(L, \mand, \mtop)$ is a commutative monoid;
\item[IL3.] $(L, \mor, \mbot)$ is a commutative monoid;
\item[IL4.] $a \mand b \leq c$ iff $b \leq a \mrarr c$ for all $a, b, c \in L$;
\item[IL5.] $\mor$ preserves all finite meets, hence also the empty meet $\aatop$, in each coordinate.
\end{enumerate}
An IL-algebra is a {\em classical linear algebra} (CL-algebra) if
\begin{enumerate}
\item[C.] $(a\mrarr \abot)\mrarr \abot = a$ for every $a\in L$.
\end{enumerate}
We will sometimes abbreviate $a\mrarr \abot$ as $a^{\mbot}$, and write C above as $a^{\mbot\mbot} = a$.

$\mathbb{L} = (L, \aand, \aor, \aatop, \abot, \mand,\mor, \mtop,\mbot, \mrarr, \mdrarr)$ is a {\em bi-intuitionistic linear algebra} (BiL-algebra) if:
\begin{enumerate}
\item[B1.] $(L, \aand, \aor, \aatop, \abot, \mand,\mor, \mtop,\mbot, \mrarr)$ is an IL-algebra;
\item[B2.] $a \mdrarr b \leq c$ iff $b \leq a \mor c$ for all $a, b, c \in L$.
\end{enumerate}

$\mathbb{L} = (L, \aand, \aor, \aatop, \abot, \mand,\mor, \mtop,\mbot, \mrarr, \oc )$ is an {\em intuitionistic linear algebra with storage} (ILS-algebra) (resp.\ a {\em classical linear algebra with storage} (CLS-algebra)) if:

\begin{enumerate}
\item[S1.] $(L, \aand, \aor, \aatop, \abot, \mand,\mor, \mtop,\mbot, \mrarr)$ is an IL-algebra (resp.\ a CL-algebra);
\item[S2.] $\oc : L\to L$ is monotone, i.e.\ $a \leq b$ implies $\oc  a\leq \oc  b$ for all $a, b\in L$;
\item[S3.] $\oc  a\leq a$ and $\oc  a\leq \oc  \oc  a$ for every $a\in L$;
\item[S4.] $\oc  a \mand \oc  b = \oc  (a\aand b)$ and $\mtop = \oc \aatop$ for all $a, b\in L$.
\end{enumerate}

$\mathbb{L} = (L, \aand, \aor, \aatop, \abot, \mand,\mor, \mtop,\mbot, \mrarr, \oc , \wn)$ is an {\em intuitionistic linear algebra with storage and co-storage} (ILSC-algebra) (resp.\ a {\em classical linear algebra with storage and co-storage} (CLSC-algebra)) if:

\begin{enumerate}
\item[SC1.] $(L, \aand, \aor, \aatop, \abot, \mand,\mor, \mtop,\mbot, \mrarr, \oc)$ is an ILS-algebra (resp.\ a CL-algebra);
\item[SC2.] $\wn : L\to L$ is monotone, i.e.\ $a \leq b$ implies $\wn  a\leq \wn  b$ for all $a, b\in L$;
\item[SC3.] $  a\leq \wn a$ and $\wn \wn  a\leq \wn  a$ for every $a\in L$;
\item[SC4.] $\wn  a \mor \wn  b = \wn  (a\aor b)$ and $\mbot = \wn \abot$ for all $a, b\in L$;
\end{enumerate}
An ILSC-algebra $\mathbb{L}$ is {\em paired} (ILP-algebra) if:
\begin{enumerate}
\item[P1.] $\oc (a\mrarr b)\leq \wn a\mrarr \wn b$ for all $a, b\in L$.
\end{enumerate}
$\mathbb{L} = (L, \aand, \aor, \aatop, \abot, \mand,\mor, \mtop,\mbot, \mrarr, \mdrarr, \oc , \wn)$ is a {\em bi-intuitionistic linear algebra with storage and co-storage} (BLSC-algebra) if:
\begin{enumerate}
\item[BLSC1.] $(L, \aand, \aor, \aatop, \abot, \mand,\mor, \mtop,\mbot, \mrarr, \oc, \wn)$ is an ILSC-algebra;
\item[BLSC2.] $(L, \aand, \aor, \aatop, \abot, \mand,\mor, \mtop,\mbot, \mrarr, \mdrarr)$ is a BiL-algebra.
\end{enumerate}
A BLSC-algebra $\mathbb{L}$ is {\em paired} (BLP-algebra) if:
\begin{enumerate}
\item[BLP1.] $(L, \aand, \aor, \aatop, \abot, \mand,\mor, \mtop,\mbot, \mrarr, \oc, \wn)$ is an ILP-algebra;
\item[BLP2.] $\oc a\mdrarr \oc b\leq \wn (a\mdrarr  b)$ for all $a, b\in L$.
\end{enumerate}
\end{definition}
CLSC-algebras reflect the logical signature of (classical) linear logic as originally defined by Girard, although no link between the exponentials is assumed, and ILP-algebras are a slightly modified version of Ono's modal FL-algebras \cite[Definition 6.1]{Ono1993}. The notion of BLSC-algebra is a symmetrization of that of ILSC-algebra, where the linear {\em co-implication} $\mdrarr$ (aka {\em subtraction}, or {\em exclusion}) has been added to the signature as the left residual of $\mor$, as was done in \cite{Gore1998,Troelstra1992,SambinBattilottiFaggian2014}. The proof-theoretic framework introduced in Section \ref{sec:ProperDisplayCalculiForLinearLogics} will account for each of these environments modularly and conservatively.

Condition IL4 implies that $\mand$ preserves all existing joins (hence all finite joins, and the empty join $\abot$ in particular) in each coordinate, and $\mrarr$ preserves all existing meets in its second coordinate and reverses all existing joins in its first coordinate. Hence,  the following de Morgan law holds in IL-algebras:
\begin{equation}
\label{eq:intuitionistic:de:Morgan}
(a\aor b)^{\mbot} = (a)^{\mbot}\aand(b)^{\mbot}.
\end{equation}
the main difference between IL-algebras and CL-algebras is captured by the following
\begin{lemma} The following de Morgan law also holds in CL-algebras:
\begin{equation}
\label{eq:intuitionistic:de:Morgan}
(a\aand b)^{\mbot} = (a)^{\mbot}\aor(b)^{\mbot}.
\end{equation}
\end{lemma}
\begin{proof}
Since by definition $(\cdot)^{\mbot}$ is antitone, $(a\aand b)^{\mbot}$ is a common upper bound of $a^{\mbot}$ and $b^{\mbot}$. Condition C  implies that $(\cdot)^{\mbot}: L\to L$ is surjective. Hence, to prove the converse inequality, it is enough to show that if $c\in L$ and $a^{\mbot}\leq c^{\mbot}$ and $b^{\mbot}\leq c^{\mbot}$, then $(a\aand b) ^{\mbot}\leq c^{\mbot}$. The assumptions imply that $c = c^{\mbot\mbot}\leq a^{\mbot\mbot} = a$, and likewise $c\leq b$. Hence $c\leq a\aand b$, which implies that $(a\aand b) ^{\mbot}\leq c^{\mbot}$ as required.
\end{proof}
By conditions S2 and S3, the operation  $\oc : L\to L$ is an interior operator on $L$ seen as a poset. Dually, by SC2 and SC3, the operation $\wn : L\to L$ is a closure operator on $L$ seen as a poset.
By general order-theoretic facts (cf.~\cite[Chapter 7]{DaveyPriestley2002}) this means that   \[\oc  = e_{\oc}\circ \iota\quad \mbox{ and }\quad \wn  =  e_{\wn}\circ \gamma,\] where $\iota: L\twoheadrightarrow \mathsf{Range}(\oc )$ and $\gamma: L\twoheadrightarrow \mathsf{Range}(\wn )$, defined by $\iota(a) = \oc  a$ and $\gamma(a) = \wn  a$ for every $a\in L$, are adjoints of the natural embeddings $e_{\oc}: \mathsf{Range}(\oc )\hookrightarrow L$ and $e_{\wn}: \mathsf{Range}(\wn )\hookrightarrow L$ as follows: \[e_{\oc}\dashv \iota \quad \mbox{ and }\quad \gamma\dashv e_{\wn},\] i.e.\ for every $a\in L$, $o\in \mathsf{Range}(\oc )$, and $c\in \mathsf{Range}(\wn)$,
 \[e_{\oc}(o)\leq a\quad \mbox{ iff }\quad o\leq \iota(a)\quad \mbox{ and }\quad \gamma(a)\leq c\quad \mbox{ iff }\quad a\leq e_{\oc}(c). \]
In what follows, we let  $K_{\oc}$ and $K_{\wn}$ be the subposets of $L$ identified by $\mathsf{Range}(\oc ) = \mathsf{Range}(\iota)$ and $\mathsf{Range}(\wn ) = \mathsf{Range}(\gamma)$ respectively. Sometimes we will refer to elements in $K_{\oc}$ as `open', and elements in $K_{\wn}$ as `closed'.
\begin{lemma} For every ILSC-algebra $\mathbb{L}$, every $\alpha\in K_{\oc}$ and $\xi\in K_{\wn}$,
\begin{equation}
\label{eq:retractions}
\iota(e_{\oc}(\alpha)) = \alpha\quad\quad \mbox{ and }\quad \quad \gamma(e_{\wn}(\xi)) = \xi.
\end{equation}
\end{lemma}
\begin{proof}
We only prove the first identity, the proof of the remaining one being dual. By adjunction, $\alpha \leq \iota(e_{\oc}(\alpha))$ iff  $e_{\oc}(\alpha)\leq e_{\oc}(\alpha)$, which always holds. As to the converse inequality $\iota(e_{\oc}(\alpha))\leq \alpha$, since $e_{\oc}$ is an order-embedding, it is enough to show that $ e_{\oc}(\iota(e_{\oc}(\alpha)))\leq e_{\oc}(\alpha)$, which by adjunction is equivalent to $\iota(e_{\oc}(\alpha))\leq \iota(e_{\oc}(\alpha))$, which always holds.
\end{proof}
 In what follows,
\begin{enumerate}
\item $\kork$ and $\bigkork$ denote joins in $K_{\oc}$;
 \item $\kandk$ and $\bigkandk$ denote meets in $K_{\wn}$;
\item $\aor$ and $\bigaor$ denote joins in $L$;
\item $\aand$ and $\bigaand$ denote meets in $L$.
\end{enumerate}
 The following fact shows that $L$-joins of open elements are open, and $L$-meets of closed elements are closed.

\begin{fact}
\label{fact:join and bang}
For all (finite) set $I$,
\begin{enumerate}
\item if $\bigkork_{i\in I} \iota  (a_i)$ exists, then 
$\bigaor_{i\in I}\oc a_i = \oc (\bigaor_{i \in I} \oc  a_i)$;
\item if $\bigkandk_{i\in I} \gamma  (a_i)$ exists, then 
$\bigaand_{i\in I}\wn a_i = \wn (\bigaand_{i \in I} \wn  a_i)$.
\end{enumerate}
\end{fact}
\begin{proof}
We only prove the first item, the proof of the second one being dual. The following chain of identities holds:
\begin{center}
\begin{tabular}{r c l l}
$\oc (\bigaor_{i\in I}\oc a_i)$& $ =$& $ \oc (\bigaor_{i\in I}e_{\oc} (\iota(a_i)))$& ($\oc  = e_{\oc}\circ \iota$)\\
& $ =$& $ \oc \circ e_{\oc}(\bigkork_{i\in I}\iota(a_i))$& ($e_{\oc}$ preserves existing joins)\\
& $ =$& $ e_{\oc}\circ \iota\circ e_{\oc} (\bigkork_{i\in I}\iota(a_i))$ & ($\oc  = e_{\oc}\circ \iota$) \\
& $ =$& $e_{\oc} (\bigkork_{i\in I}\iota(a_i))$ & ($e_{\oc}\circ\iota\circ e_{\oc} = e_{\oc}$)\\
& $ =$& $\bigaor_{i\in I}e_{\oc}(\iota(a_i))$ & ($e_{\oc}$ preserves existing joins) \\
& $ =$& $\bigaor_{i\in I}\oc a_i$. & ($\oc  = e_{\oc}\circ \iota$) \\
\end{tabular}
\end{center}
\end{proof}

Hence,  the subposets $K_{\oc}$ and $K_{\wn}$ are, respectively, a $\aor$-subsemilattice and a $\aand$-subsemilattice of $(L, \aor, \aand, \abot,\aatop)$: indeed,

\[e_{\oc}(\bigkork_{i\in I} \iota (a_i)) = \bigaor_{i\in I}e_{\oc}(\iota(a_i)) = \bigaor_{i\in I}\oc  a_i\quad \mbox{ and }\quad e_{\wn}(\bigkandk_{i\in I} \gamma (a_i)) = \bigaand_{i\in I}e_{\wn}(\gamma(a_i)) = \bigaand_{i\in I}\wn  a_i.\]

\begin{definition}
\label{def:kernel}
For any ILS-algebra $\mathbb{L} = (L, \aand, \aor, \aatop, \abot, \mand, \mtop, \mrarr, \oc )$, let the {\em left-kernel} of $\mathbb{L}$ be the algebra $\mathbb{K}_{\oc} = (K_{\oc}, \kand, \kor, \ktop, \kbot, \krarrk)$ defined as follows:
\begin{itemize}
\item[LK1.] $K_{\oc}: = \mathsf{Range}(\oc ) = \mathsf{Range}(\iota)$, where $\iota: L\twoheadrightarrow K_{\oc}$ is defined by letting $\iota(a) = \oc a$ for any $a\in L$;
\item[LK2.] $\alpha \kor \beta: = \iota(e_{\oc}(\alpha) \aor e_{\oc} (\beta))$ for all $\alpha, \beta\in K_{\oc}$;
\item[LK3.] $\alpha \kand \beta: = \iota (e_{\oc}(\alpha) \aand e_{\oc} (\beta))$ for all $\alpha, \beta\in K_{\oc}$;
\item[LK4.]  $\ktop:  = \iota (\aatop)$;
\item[LK5.]  $\kbot: = \iota(\abot)$;
\item[LK6.] $\alpha\rightarrow \beta: = \iota (e_{\oc} (\alpha)\mrarr e_{\oc} (\beta))$.
\end{itemize}
We also let
\begin{itemize}
\item[LK7.] $\kneg \alpha: = \alpha\rightarrow \kbot = \iota (e_{\oc}(\alpha)\mrarr e_{\oc}( \kbot)) = \iota (e_{\oc}(\alpha)\mrarr \abot) = \iota (e_{\oc}(\alpha)^{\mbot})$.
\end{itemize}
For any ILSC-algebra $\mathbb{L} = (L, \aand, \aor, \aatop, \abot, \mand, \mtop, \mrarr, \oc , \wn)$, let the {\em right-kernel} of $\mathbb{L}$ be the algebra $\mathbb{K}_{\wn} = (K_{\wn}, \kand, \kor, \ktop, \kbot)$ defined as follows: 
\begin{itemize}
\item[RK1.] $K_{\wn}: = \mathsf{Range}(\wn ) = \mathsf{Range}(\gamma)$, where $\gamma: L\twoheadrightarrow K_{\wn}$ is defined by letting $\gamma(a) = \wn a$ for any $a\in L$;
\item[RK2.] $\xi \kand \chi: = \gamma(e_{\wn} (\xi) \aand e_{\wn} (\chi))$;
\item[RK3.] $\xi \kor \chi: = \gamma (e_{\wn} (\xi) \aor e_{\wn} (\chi))$;
\item[RK4.]  $\kbot: = \gamma (\abot)$;
\item[RK5.]  $\ktop: = \gamma(\aatop)$.
\end{itemize}
For any BLSC-algebra $\mathbb{L} = (L, \aand, \aor, \aatop, \abot, \mand, \mtop, \mrarr, \mdrarr, \oc , \wn)$, let the {\em right-kernel} of $\mathbb{L}$ be the algebra $\mathbb{K}_{\wn} = (K_{\wn}, \kand, \kor, \ktop, \kbot, \kdrarr)$ such that $\mathbb{K}_{r} = (K_{\wn}, \kand, \kor, \ktop, \kbot)$ is defined  as above, and moreover:
\begin{itemize}
\item[RK6.] $\xi\kdrarrk \chi: = \gamma (e_{\wn} (\xi)\mdrarr e_{\wn} (\chi))$.
\end{itemize}
\end{definition}
It is interesting to notice the fit between  the definition of the algebraic structure of the kernels given above and the translation by which intuitionistic formulas embed into linear formulas given in \cite[Chapter 5]{Girard1987}.
The following proposition develops and expands on an observation made by Troelstra (cf.\ \cite[Exercise after Lemma 8.17]{Troelstra1992}).
\begin{proposition}\label{prop: algebraic structure on kernels}
For any  $\mathbb{L}$,
\begin{enumerate}
\item if $\mathbb{L}$ is an ILS-algebra, then   $\mathbb{K}_{\oc}$ is a Heyting algebra;
\item if $\mathbb{L}$ is an CLS-algebra, then   $\mathbb{K}_{\oc}$ is a Boolean algebra;
\item if $\mathbb{L}$ is an ILSC-algebra, then   $\mathbb{K}_{\oc}$ and $\mathbb{K}_{\wn}$ are  a Heyting algebra and a distributive lattice respectively;
\item if $\mathbb{L}$ is an BLSC-algebra, then   $\mathbb{K}_{\oc}$ and $\mathbb{K}_{\wn}$ are  a Heyting algebra and a co-Heyting algebra respectively.
\end{enumerate}
\end{proposition}
\begin{proof}
We only prove the first and second item, the proofs of the remaining items being dual to the proof of item 1. Let us show that for all $a, b\in L$, the greatest lower bound of $\iota(a)$ and $\iota(b)$ exists and coincides with $\iota (a\aand b)$.
From $b\leq \aatop$ and S2 we get $\oc b\leq \oc \aatop$. Hence, by S4 and the monotonicity of $\mand$ in both coordinates, \[e_{\oc}(\iota (a\aand b)) = \oc (a\aand b) = \oc a\mand \oc b\leq \oc a\mand \oc \aatop  = \oc a\mand 1 = \oc a = e_{\oc}(\iota(a)),\]
 which implies that $\iota (a\aand b)\leq \iota (a)$, since $e_{\oc}: K_{\oc}\hookrightarrow L$ is an order-embedding.
Likewise, one shows that $\iota (a\aand b)\leq \iota (b)$, which finishes the proof that  $\iota (a\aand b)$ is a common lower bound  for $\iota(a)$ and $\iota(b)$.
To finish the proof of the claim, one needs to show that, if $c\in L$ and $\iota (c)\leq \iota (a)$ and    $\iota (c)\leq \iota (b)$, then $\iota (c)\leq \iota (a\aand b)$. Indeed, the assumptions imply that $\oc c = e_{\oc}(\iota(c))\leq e_{\oc}(\iota(a)) = \oc a\leq a$, and likewise, $\oc c\leq b$, hence $\oc c\leq a\aand b$, which implies, by S2 and S3, that $e_{\oc}(\iota(c)) = \oc c = \oc \oc c\leq \oc (a\aand b) = e_{\oc}(\iota(a\aand b))$. This implies that $\iota (c)\leq \iota (a\aand b)$, as required.

Let us now show that $\kand$ preserves coordinatewise all existing joins in $K_{\oc}$. That is, if $\bigkork_{i\in I}\iota(a_i)$ exists, then for every $b\in L$, \[\left(\bigkork_{i\in I}\iota(a_i)\right)\kand \iota (b) = \bigvee_{i\in I} (\iota(a_i)\kand \iota(b)).\]
The following chain of identities holds:
\begin{center}
\begin{tabular}{r c l l}
$\left(\bigkork_{i\in I}\iota(a_i)\right)\wedge \iota (b)$ &$ = $& $\iota \left(\bigaor_{i\in I}\oc a_i\right)\kand \iota (b)$& (Definition \ref{def:kernel} (LK2))\\
&$ = $& $\iota (\left(\bigaor_{i\in I}\oc a_i\right)\aand b)$ & (Definition \ref{def:kernel} (LK3))\\
&& \\
$\bigkork_{i\in I}(\iota(a_i)\kand \iota (b))$  &$ = $& $\bigkork_{i\in I}\iota(a_i\aand b)$ & (Definition \ref{def:kernel} (LK3))\\
 &$ = $& $\iota\left(\bigaor_{i\in I}\oc (a_i\aand b)\right)$ & (Definition \ref{def:kernel} (LK2))\\
  &$ = $& $\iota\left(\bigaor_{i\in I}(\oc a_i\mand \oc b)\right)$ & (S4)\\
    &$ = $& $\iota (\left(\bigaor_{i\in I}\oc a_i\right)\mand \oc b)$ & ($\mand$ preserves existing joins)\\
\end{tabular}
\end{center}
Hence, to finish the proof of the claim it is enough to show that \[\iota (\left(\bigaor_{i\in I}\oc a_i\right)\aand b) = \iota (\left(\bigaor_{i\in I}\oc a_i\right)\mand \oc b).\]
Since $e_{\oc}: K_{\oc}\hookrightarrow L$ is injective and $\oc = e_{\oc}\circ \iota$, it is enough to show that
\[\oc (\left(\bigaor_{i\in I}\oc a_i\right)\aand b) = \oc  (\left(\bigaor_{i\in I}\oc a_i\right)\mand \oc b).\]
Indeed, by S3, S4 and Fact \ref{fact:join and bang},
\[\oc (\left(\bigaor_{i\in I}\oc a_i\right)\aand b) = \oc \oc (\left(\bigaor_{i\in I}\oc a_i\right)\aand b) = \oc  \left(\oc \left(\bigaor_{i\in I}\oc a_i\right)\mand \oc b\right) = \oc  (\left(\bigaor_{i\in I}\oc a_i\right)\mand \oc b),\]
as required. By S2 and S4, $\iota(a)\leq \iota(\aatop) = 1$  for every $a\in L$, which motivates Definition \ref{def:kernel} (LK4), and moreover,  S3 implies that $\iota(0)\leq 0$, which motivates Definition \ref{def:kernel} (KL5).   Let us show that for all $a, b, c\in L$,
\[\iota(a)\kand\iota(b)\leq \iota (c)\quad \mbox{ iff }\quad \iota(b)\leq \iota(a)\krarrk \iota(c).\]
\begin{center}
\begin{tabular}{r c l l}
$\iota(a)\kand\iota(b)\leq \iota (c)$ & iff & $\iota(a\aand b)\leq \iota(c)$ & (Definition \ref{def:kernel} (LK3))\\
& iff & $\oc (a\aand b)\leq \oc c$ & ($e_{\oc}: L\hookrightarrow K$ order-embedding)\\
& iff & $\oc a\mand \oc b\leq \oc c$ & (S4)\\
& iff & $\oc b\leq \oc a\mrarr \oc c$ & (IL4)\\
& iff & $e_{\oc}(\iota(b))\leq \oc a\mrarr \oc c$ & ($\oc  = e_{\oc}\circ \iota$)\\
& iff & $\iota(b)\leq \iota(\oc a\mrarr \oc c)$ & ($e_{\oc}\dashv \iota$)\\
& iff & $\iota(b)\leq \iota(a)\krarrk \iota(c)$. & (Definition \ref{def:kernel} (LK6))\\
\end{tabular}
\end{center}
This concludes the proof of item 1. As to item 2, let us assume that $a^{\mbot\mbot} = a$ for any $a\in L$. Then, by Definition \ref{def:kernel} (LK7),
\[\kneg \kneg \iota(a) = \iota(a^{\mbot\mbot}) = \iota(a),\]
which is enough to establish that the Heyting algebra $\mathbb{K}_{\oc}$ is a Boolean algebra.
\end{proof}
Summing up, the axiomatization of the exponentials as interior and closure operators respectively generates
 a {\em composite} algebraic environment of linear logic, in which linear algebras come together with one or two other algebras, namely the kernels, which, depending on the original linear signature, can be endowed with a structure of Heyting algebras, distributive lattices or co-Heyting algebras. Moreover, the adjunctions relating the linear algebra and its kernel(s) also guarantee  the natural embedding maps to enjoy the following additional properties: 
\begin{proposition} For every ILS(C)-algebra $\mathbb{L}$,   all $\alpha, \beta\in \mathbb{K}_{\oc}$, and all $\xi, \chi\in \mathbb{K}_{\wn}$,
\[ e_{\oc}(\alpha)\mand e_{\oc}(\beta) =  e_{\oc}(\alpha\kand \beta)\quad\mbox{ and }\quad e_{\wn}(\xi)\mor e_{\wn}(\chi) =  e_{\wn}(\xi\ork \chi)\]
\[e_{\oc}(\ktop) = \mtop\quad \mbox{ and }\quad e_{\wn}(\botk) = \mbot.\]
\end{proposition}
\begin{proof} We only prove the identities involving $e_\oc$, the proof of the remaining ones being dual.
 Since $\iota : \mathbb{L}\twoheadrightarrow\mathbb{K}_{\oc}$ is surjective and order-preserving, proving the required identity is equivalent to showing that, for all $a, b\in \mathbb{L}$,
 \[ e_{\oc}(\iota(a))\mand e_{\oc}(\iota(b)) =  e_{\oc}(\iota(a)\kand \iota(b))\quad\mbox{ and }\quad e_{\oc}(\iota(\aatop)) = \mtop.\]
Since $\iota$ preserves meets, $e_{\oc}(\iota(a)\kand \iota(b)) = e_{\oc}(\iota(a\aand b)) = \oc (a\aand b)$. Hence the displayed identities above are equivalent to
  \[\oc a\mand \oc b =  \oc (a\aand b)\quad\mbox{ and }\quad \oc \aatop = \mtop,\]
  which are true by S4.
\end{proof}
\subsection{The interaction of exponentials in paired linear algebras}\label{ssec:from ono axioms to eterogeneous arrows}
\begin{proposition}
\label{prop:equiv promotion}
\begin{enumerate}
\item The following are equivalent in any ILSC-algebra $\mathbb{L}$:
\begin{enumerate}
\item for all $a, b, c\in \mathbb{L}$, if $\oc a\mand b\leq \wn c$ then $\oc a\mand \wn b\leq \wn c$;
\item for all $a, b, c\in \mathbb{L}$, if $b\leq \oc a\mrarr\wn c$ then $\wn b\leq \oc a\mrarr\wn c$;
\item $\wn (\oc a\mrarr \wn b) = \oc a\mrarr \wn b$ for all $a, b\in \mathbb{L}$.
\end{enumerate}
\item If $\mathbb{L}$ is also a BLSC-algebra, then the following are equivalent:
\begin{enumerate}
\item for all $a, b, c\in \mathbb{L}$, if $\oc a\leq  b\mor \wn c$ then $\oc a\leq \oc b\mor \wn c$;
\item for all $a, b, c\in \mathbb{L}$, if $\wn c\mdrarr\oc a\leq b$ then $\wn c\mdrarr\oc a\leq \oc b$;
\item $\oc (\wn b\mdrarr\oc a) = \wn b\mdrarr\oc a$ for all $a, b\in \mathbb{L}$.
\end{enumerate}
\end{enumerate}
\end{proposition}
\begin{proof}
We only prove item 1, the proof of item 2 being order-dual. Clearly, (a) and (b) are equivalent by IL4. Let us assume (c), and let $a, b, c\in \mathbb{L}$ such that $b\leq \oc a\mrarr\wn c$. Then by SC2,
\[\wn b \leq \wn (\oc a\mrarr\wn c) = \oc a\mrarr\wn c,\]
which proves (b). Let us assume (b) and let $a, b\in \mathbb{L}$. Then $\oc a\mrarr\wn b\leq \oc a\mrarr\wn b$ implies that $\wn (\oc a\mrarr\wn b)\leq \oc a\mrarr\wn b$, which, together with $\oc a\mrarr\wn b \leq \wn (\oc a\mrarr\wn b)$ which holds by SC3, proves (c).
\end{proof}
The following proposition develops and expands on \cite[Lemma 6.1(6)]{Ono1993}
\begin{proposition}
\label{prop:existence of multi-type implication}
\begin{enumerate}
\item For any ILP-algebra $\mathbb{L}$, any (hence all) of the conditions (a)-(c) in Proposition \ref{prop:equiv promotion}.1 holds.
\item For any BLP-algebra $\mathbb{L}$, any (hence all) of the conditions (a)-(c) in Proposition \ref{prop:equiv promotion}.1 and 2 hold.
\end{enumerate}
\end{proposition}
\begin{proof}
We only prove item 1, the proof of (the second part of) item 2 being order-dual. Let $a, b, c\in \mathbb{L}$. By IL2 and IL4, it is enough to show that if $\oc a\leq b\mrarr\wn c$ then $\oc a\leq \wn b\mrarr\wn c$.  Indeed, by the assumption, S3, S2, P,  and SC3,
\[\oc a = \oc\oc a\leq \oc(b\mrarr\wn c)\leq \wn b\mrarr\wn \wn c = \wn b\mrarr\wn  c,\]
as required.
\end{proof}
Proposition \ref{prop:existence of multi-type implication} and the two items (c) of Proposition 
\ref{prop:equiv promotion}   imply that in any ILP-algebra (resp.\ BLP-algebra)
$\mathbb{L}$, a binary map  (resp.\ binary maps) \[\KKrarr: \mathbb{K}_{\oc}\times \mathbb{K}_{\wn}\longrightarrow \mathbb{K}_{\wn}\quad \mbox{ and }\quad \KKBrarr: \mathbb{K}_{\wn} \times \mathbb{K}_{\oc} \longrightarrow \mathbb{K}_{\oc}\] can be defined such that, for any $\alpha\in \mathbb{K}_{\oc}$ and $\xi \in \mathbb{K}_{\wn}$, the element $\alpha\KKrarr \xi\in \mathbb{K}_{\wn}$ (resp.\ $\alpha\KKBrarr \xi\in \mathbb{K}_{\wn}$) is the unique solution to the following equation(s):
\begin{equation}\label{eq: main properties andm and rasm}
e_{\oc}(\alpha)\mrarr e_{\wn}(\xi) = e_{\wn} (\alpha\KKrarr \xi) \quad \mbox{ and }\quad  e_{\wn} (\xi)\mdrarr e_{\oc}(\alpha) = e_{\oc} (\xi \KKBrarr \alpha)\end{equation}

\begin{proposition}
\label{prop:order theoret properties eterogeneous arrows}
\begin{enumerate}
\item For any ILP-algebra $\mathbb{L}$,
\begin{enumerate}
\item the map $\KKrarr$ preserves finite meets in its second coordinate and reverses finite joins in its first coordinate.
\item $\gamma( a)\leq  \iota (a\mrarr b)\KKrarr \gamma( b)$ for all $a, b\in \mathbb{L}$.
\end{enumerate}
\item For any BLP-algebra $\mathbb{L}$,
\begin{enumerate}
\item the map $\KKBrarr$ preserves finite joins in its second coordinate and reverses finite meets in its first coordinate.
\item $\gamma( a\mdrarr b)\KKBrarr \iota(b)\leq  \iota (a)$ for all $a, b\in \mathbb{L}$.
\end{enumerate}
\end{enumerate}
\end{proposition}
\begin{proof}
We only prove item 1, the proof of item 2 being order-dual. 
As to (a), to show that $\alpha\KKrarr (\xi\kand \chi) = (\alpha\KKrarr \xi)\kand  (\alpha\KKrarr  \chi)$, it is enough to show that
\[e_{\wn}(\alpha\KKrarr(\xi\kand \chi)) = e_{\wn}((\alpha\KKrarr \xi)\kand  (\alpha\KKrarr  \chi)).\]
Indeed,
\begin{center}
\begin{tabular}{r c l r}
&&$e_{\wn}((\alpha\KKrarr \xi)\kand  (\alpha\KKrarr  \chi))$ \\
&$=$& $e_{\wn}(\alpha\KKrarr \xi)\aand  e_{\wn}(\alpha\KKrarr  \chi)$ & ($e_{\wn}$ preserves existing meets) \\
&$=$& $(e_{\oc}(\alpha)\mrarr e_{\wn}(\xi))\aand (e_{\oc}(\alpha)\mrarr e_{\wn}(\chi))$ & \eqref{eq: main properties andm and rasm}\\
&$=$& $e_{\oc}(\alpha)\mrarr (e_{\wn}(\xi)\aand e_{\wn}(\chi))$ & ($\mrarr$ preserves meets in 2nd coord.)\\
&$=$& $e_{\oc}(\alpha)\mrarr e_{\wn}(\xi\kand \chi)$& ($e_{\wn}$ preserves existing meets)\\
&$=$& $e_{\wn}(\alpha\KKrarr(\xi\kand \chi))$. & \eqref{eq: main properties andm and rasm}\\
\end{tabular}
\end{center}
To show that $\alpha\KKrarr \ktop = \ktop$, it is enough to show that
\[e_{\wn}(\alpha\KKrarr \ktop) = e_{\wn}(\ktop) = \aatop.\]
Indeed,
\begin{center}
\begin{tabular}{r c l r}
&&$e_{\wn}(\alpha\KKrarr \ktop)$ \\
&$=$& $e_{\oc}(\alpha)\mrarr e_{\wn}(\ktop)$ & \eqref{eq: main properties andm and rasm}\\
&$=$& $e_{\oc}(\alpha)\mrarr \aatop$& ($e_{\wn}$ preserves existing meets)\\
&$=$& $\aatop$. & ($\mrarr$ preserves meets in 2nd coord.)\\
\end{tabular}
\end{center}
The verification that $\KKrarr$  reverses finite joins in its first coordinate is similar, and makes use of $e_{\oc}$ preserving existing joins and $\mrarr$ reversing joins in its first coordinate. As to (b), let $a, b\in \mathbb{L}$. Since $\mathbb{L}$ is an ILP-algebra, axiom P holds (cf.\ Definition \ref{def:linear algebras}):
\begin{center}
\begin{tabular}{c l l}
& $\oc (a\mrarr b)\leq \wn a\mrarr \wn b$ & (P)\\
iff & $\wn a\leq \oc (a\mrarr b)\mrarr \wn b$ & (IL2 and IL4)\\
iff & $e_{\wn}(\gamma( a))\leq e_{\oc}(\iota (a\mrarr b))\mrarr e_{\wn}(\gamma( b))$ & ($\oc = e_{\oc}\circ\iota$ and $\wn = e_{\wn}\circ\gamma$)\\
iff & $e_{\wn}(\gamma( a))\leq  e_{\wn}(\iota (a\mrarr b)\KKrarr \gamma( b))$ & \eqref{eq: main properties andm and rasm}\\
iff & $\gamma( a)\leq  \iota (a\mrarr b)\KKrarr \gamma( b)$. & ($e_{\wn}$ order-embedding)\\
\end{tabular}
\end{center}
\end{proof}

\subsection{Reverse-engineering exponentials}\label{ssec: reverse engineering}
In the previous subsections, we have seen that certain mathematical structures (namely, the kernels and their adjunction situations, and the maps $\KKrarr$ and $\KKBrarr$) arise from the axiomatization of linear algebras with exponentials. In the present subsection, we take these mathematical structures as primary (we call them {\em heterogeneous algebras}, adopting the terminology of \cite{birkhoff1970heterogeneous}), and capture linear algebras with exponentials in terms of these. The results of the present subsection  establish the equivalence between the multi-type  and the standard (single-type) algebraic semantics for linear logic, and hence provide the mathematical justification for taking the multi-type semantics as primary.

Heterogeneous algebras for the various linear logics can be understood as the algebraic counterparts of Benton's LNL-models \cite[Definition 9]{Benton1994}, and the identities and inequalities defining them 
can be understood as the multi-type counterparts of the identities defining the algebraic behaviour of exponentials. The latter understanding will be developed further in Section \ref{sec: multi-type language}.

\begin{definition}
\label{def:heterogeneous algebras}
\begin{enumerate}
\item An {\em heterogeneous ILS-algebra} (resp.~{\em heterogeneous CLS-algebra}) is a structure $(\mathbb{L}, \mathbb{A}, e_{\oc}, \iota)$ such that $\mathbb{L}$ is an IL-algebra (resp.\ CL-algebra), $\mathbb{A}$ 
    is a Heyting algebra,  $e_{\oc}: \mathbb{A}\hookrightarrow \mathbb{L}$ and $\iota: \mathbb{L}\twoheadrightarrow \mathbb{A}$ such that $e_{\oc}\dashv \iota$, and $\iota(e_{\oc}(\alpha)) = \alpha$ for every $\alpha\in \mathbb{A}$, and moreover, for all $\alpha, \beta\in \mathbb{A}$,
    \[ e_{\oc}(\alpha)\mand e_{\oc}(\beta) =  e_{\oc}(\alpha\kandk \beta) \quad\mbox{ and }\quad e_{\oc}(\ktopk) = \mtop.\]
An heterogeneous ILS-algebra (resp.~{\em CLS-algebra}) is {\em perfect} if both $\mathbb{L}$ and $\mathbb{A}$ are perfect (cf.~Definition \ref{def:perfect LE}).

\item An {\em heterogeneous ILSC-algebra} (resp.~{\em heterogeneous BLSC-algebra}) is a structure $(\mathbb{L}, \mathbb{A}, \mathbb{B}, e_{\oc}, \iota, e_{\wn}, \gamma)$ such that $\mathbb{L}$ is an IL-algebra (resp.\ BiL-algebra), $(\mathbb{L}, \mathbb{A}, e_{\oc}, \iota)$ is an heterogeneous ILS-algebra, $\mathbb{B}$ 
    is a distributive lattice (resp.\ co-Heyting algebra), 
    $e_{\wn}: \mathbb{B}\hookrightarrow \mathbb{L}$ and $\gamma: \mathbb{L}\twoheadrightarrow \mathbb{B}$ such that $\gamma \dashv e_{\wn}$, and $\gamma(e_{\wn}(\xi)) = \xi$ for every $\xi\in \mathbb{B}$, and moreover, for all $\xi, \chi\in \mathbb{B}$,
    \[ e_{\wn}(\xi)\mor e_{\wn}(\chi) =  e_{\wn}(\xi\kork \chi) \quad\mbox{ and }\quad e_{\wn}(\kbotk) = \mbot.\]
An heterogeneous ILSC-algebra (resp.~{\em BLSC-algebra}) is {\em perfect} if $\mathbb{L}$, $\mathbb{A}$ and $\mathbb{B}$ are perfect (cf.\ Definition \ref{def:perfect LE}).

\item An {\em heterogeneous ILP-algebra} is a structure $(\mathbb{L}, \mathbb{A}, \mathbb{B}, e_{\oc}, \iota, e_{\wn}, \gamma, \KKrarr)$ such that
  $(\mathbb{L}, \mathbb{A}, \mathbb{B}, e_{\oc}, \iota, e_{\wn}, \gamma)$ is an heterogeneous ILSC-algebra, and 
  \[\KKrarr: \mathbb{A}\times \mathbb{B}\longrightarrow \mathbb{B}\] is such that, for any $\alpha\in \mathbb{A}$, $\xi \in \mathbb{B}$ and $a, b\in \mathbb{L}$, 
\[e_{\oc}(\alpha)\mrarr e_{\wn}(\xi) = e_{\wn} (\alpha\KKrarr \xi)\quad \mbox{ and }\quad \gamma( a)\leq  \iota (a\mrarr b)\KKrarr \gamma( b).\]
An heterogeneous ILP-algebra  is {\em perfect} if $\mathbb{L}$, $\mathbb{A}$ and $\mathbb{B}$ are perfect (cf.\ Definition \ref{def:perfect LE}) and $\KKrarr$ is completely meet-preserving in its second coordinate and completely join-reversing in its first coordinate.

\item An {\em heterogeneous BLP-algebra} is a structure $(\mathbb{L}, \mathbb{A}, \mathbb{B}, e_{\oc}, \iota, e_{\wn}, \gamma, \KKrarr, \KKBrarr)$ such that
  $(\mathbb{L}, \mathbb{A}, \mathbb{B}, e_{\oc}, \iota, e_{\wn}, \gamma, \KKrarr)$ is an heterogeneous ILP-algebra, and 
  \[\KKBrarr: \mathbb{B}\times \mathbb{A}\longrightarrow \mathbb{A}\] is such that, for any $\alpha\in \mathbb{A}$, $\xi \in \mathbb{B}$ and $a, b\in \mathbb{L}$, 
\[e_{\wn}(\xi)\mdrarr e_{\oc}(\alpha) = e_{\wn} (\xi\KKBrarr \alpha)\quad \mbox{ and }\quad \gamma( a\mdrarr b)\KKBrarr \iota(b)\leq  \iota (a).\]
An heterogeneous BLP-algebra  is {\em perfect} if $\mathbb{L}$, $\mathbb{A}$ and $\mathbb{B}$ are perfect (cf.\ Definition \ref{def:perfect LE}), $\KKrarr$ is completely meet-preserving in its second coordinate and completely join-reversing in its first coordinate, and $\KKBrarr$ is completely join-preserving in its second coordinate and completely meet-reversing in its first coordinate.

\end{enumerate}
\end{definition}
In the light of the definitions above, the results of Sections \ref{ssec: linear algebras} and \ref{ssec:from ono axioms to eterogeneous arrows} can be summarized by the following:

\begin{proposition}
\label{prop:from single to multi}
For any algebra $\mathbb{L}$,
\begin{enumerate}
\item If $\mathbb{L}$ is an ILS-algebra (resp.~CLS-algebra), then $(\mathbb{L}, \mathbb{K}_\oc, e_{\oc}, \iota)$ is an heterogeneous ILS-algebra (resp.~heterogeneous CLS-algebra), where $\mathbb{K}_\oc$ is  as in Definition \ref{def:kernel}.

\item If $\mathbb{L}$ is an ILSC-algebra (resp.~BLSC-algebra), then  $(\mathbb{L}, \mathbb{K}_\oc, \mathbb{K}_\wn, e_{\oc}, \iota, e_{\wn}, \gamma)$ is an heterogeneous ILSC-algebra (resp.~heterogeneous BLSC-algebra), where $\mathbb{K}_\oc$ and $\mathbb{K}_\wn$ are  as  in Definition \ref{def:kernel}.

\item If $\mathbb{L}$ is an ILP-algebra, then $(\mathbb{L}, \mathbb{K}_\oc, \mathbb{K}_\wn, e_{\oc}, \iota, e_{\wn}, \gamma, \KKrarr)$ is an heterogeneous ILP-algebra,  where $\mathbb{K}_\oc$ and $\mathbb{K}_\wn$ are  as  in Definition \ref{def:kernel}, and $\KKrarr$ as indicated right after the proof of Proposition \ref{prop:existence of multi-type implication}.

\item If $\mathbb{L}$ is an BLP-algebra, then $(\mathbb{L}, \mathbb{K}_\oc, \mathbb{K}_\wn, e_{\oc}, \iota, e_{\wn}, \gamma, \KKrarr, \KKBrarr)$  is an heterogeneous BLP-algebra, where $\mathbb{K}_\oc$ and $\mathbb{K}_\wn$ are  as  in Definition \ref{def:kernel}, and $\KKrarr$ and $\KKBrarr$ as indicated right after the proof of Proposition \ref{prop:existence of multi-type implication}.
\end{enumerate}
\end{proposition}
Together with the proposition above, the following proposition  shows that heterogeneous algebras are an equivalent presentation of linear algebras with exponentials:

\begin{proposition}
\label{prop:reverse engineering}
For any algebra $\mathbb{L}$,
\begin{enumerate}
\item If $(\mathbb{L}, \mathbb{A}, e_{\oc}, \iota)$ is an heterogeneous ILS-algebra (resp.~heterogeneous CLS-algebra),
then $\mathbb{L}$ can be endowed with the structure of ILS-algebra (resp.\ CLS-algebra)  defining  $\oc: \mathbb{L}\rightarrow \mathbb{L}$ by $\oc a: = e_{\oc}(\iota(a))$ for every $a\in \mathbb{L}$. Moreover, $\mathbb{A}\cong \mathbb{K}_{\oc}$.

\item If $(\mathbb{L}, \mathbb{A}, \mathbb{B}, e_{\oc}, \iota, e_{\wn}, \gamma)$ is an heterogeneous ILSC-algebra (resp.~heterogeneous BLSC-algebra),
then $\mathbb{L}$ can be endowed with the structure of ILSC-algebra (resp.\ BLSC-algebra) by defining $\oc$ as in the item above, and  $\wn: \mathbb{L}\rightarrow \mathbb{L}$ by $\wn a: = e_{\wn}(\gamma(a))$ for every $a\in \mathbb{L}$. Moreover, $\mathbb{A}\cong \mathbb{K}_{\oc}$ and $\mathbb{B}\cong \mathbb{K}_{\wn}$.

\item If $(\mathbb{L}, \mathbb{A}, \mathbb{B}, e_{\oc}, \iota, e_{\wn}, \gamma, \KKrarr)$ is an heterogeneous ILP-algebra,
then $\mathbb{L}$ can be endowed with the structure of ILP-algebra by defining $\oc$ and $\wn$ as in the items above.

\item If $(\mathbb{L}, \mathbb{A}, \mathbb{B}, e_{\oc}, \iota, e_{\wn}, \gamma, \KKrarr, \KKBrarr)$  is an heterogeneous BLP-algebra,
then $\mathbb{L}$ can be endowed with the structure of BLP-algebra by defining $\oc$ and $\wn$ as in the items above.
\end{enumerate}
\end{proposition}

\begin{proof}
Let us prove item 1. By assumption, $\mathbb{L}$ is an IL-algebra (resp.\ CL-algebra), which verifies S1. The assumption that $e_{\oc}\dashv \iota$ implies that both $e_{\oc}$ and $\iota$ are monotone, and hence so is their composition $\oc$, which verifies S2. Also from $e_{\oc}\dashv \iota$ and $\iota (a)\leq \iota (a)$ it immediately follows that $\oc a = e_{\oc}(\iota(a))\leq a$ for every $a\in \mathbb{L}$. To finish the proof of S3, we need to show that $\oc a\leq \oc\oc a$ for every $a\in \mathbb{L}$. By definition of $\oc$, this is equivalent to showing that  $e_{\oc}(\iota(a))\leq e_{\oc}(\iota(e_{\oc}(\iota(a))))$. By the monotonicity of  $e_{\oc}$, it is enough to show that $\iota(a)\leq \iota(e_{\oc}(\iota(a)))$, which by adjunction is equivalent to $e_{\oc}(\iota(a))\leq e_{\oc}(\iota(a))$, which is always true.  This finishes the proof of S3. As to S4, let us observe preliminarily that, since $\iota$ is a right-adjoint, it preserves existing meets, and hence $\iota(\top) = \ktopk$ and $\iota(a\aand b) = \iota(a)\kandk \iota(b)$ for all $a, b\in \mathbb{L}$. By definition, showing that $\oc (a\aand b) = \oc a\mand \oc b$ and $\oc \aatop = \mtop$ is equivalent to showing that \[e_{\oc}(\iota(a\aand b)) = e_{\oc}(\iota(a))\mand e_{\oc}(\iota(b))\quad \mbox{ and } \quad e_{\oc}(\iota(\aatop)) = \mtop,\]
which, thanks to the preliminary observation,  can be equivalently rewritten as follows:
\[e_{\oc}(\iota(a)\kandk \iota(b)) = e_{\oc}(\iota(a))\mand e_{\oc}(\iota(b)) \quad \mbox{ and } \quad e_{\oc}(\ktopk) = \mtop\]
which are true, by the assumptions on $e_{\oc}$. This completes the proof that $(\mathbb{L}, \oc)$ is an ILS-algebra. As to the second part of the statement, let us show preliminarily that the following identities hold:
\begin{itemize}
\item[LK2$_{\mathbb{A}}$.] $\alpha \kork \beta = \iota(e_{\oc}(\alpha) \aor e_{\oc} (\beta))$ for all $\alpha, \beta\in \mathbb{A}$;
\item[LK3$_{\mathbb{A}}$.] $\alpha \kandk \beta = \iota (e_{\oc}(\alpha) \aand e_{\oc} (\beta))$ for all $\alpha, \beta\in \mathbb{A}$;
\item[LK4$_{\mathbb{A}}$.]  $\ktopk  = \iota (\aatop)$;
\item[LK5$_{\mathbb{A}}$.]  $\kbotk = \iota(\abot)$;
\item[LK6$_{\mathbb{A}}$.] $\alpha\krarrk \beta = \iota (e_{\oc} (\alpha)\mrarr e_{\oc} (\beta))$ for all $\alpha, \beta\in \mathbb{A}$.
\end{itemize}
Let us observe  that, since $e_{\oc}$ is a left-adjoint, it preserves existing joins, and hence $e_{\oc}(\kbotk) = \abot$ and $e_{\oc}(\alpha\kork \beta) = e_{\oc}(\alpha)\aor e_{\oc}(\beta)$. Together with $\iota\circ e_{\oc}  = Id_{\mathbb{A}}$, these identities imply that \[\kbot = \iota(e_{\oc}(\kbotk)) = \iota(\abot)\quad \mbox{ and }\quad \alpha\kork \beta = \iota(e_{\oc}(\alpha\kork \beta)) = \iota( e_{\oc}(\alpha)\aor e_{\oc}(\beta)),\]
which proves LK2$_{\mathbb{A}}$ and LK5$_{\mathbb{A}}$. As to LK3$_{\mathbb{A}}$ and LK4$_{\mathbb{A}}$, since $\iota$ preserves existing meets, and $\iota\circ e_{\oc}  = Id_{\mathbb{A}}$,
\[\iota (e_{\oc}(\alpha) \aand e_{\oc} (\beta)) = \iota (e_{\oc}(\alpha)) \kand \iota(e_{\oc} (\beta)) =\alpha \kandk \beta\quad \mbox{ and }\quad  \ktopk  = \iota (\aatop),\]
as required. As to LK6$_{\mathbb{A}}$, let $\alpha, \beta\in \mathbb{A}$. Since $\mathbb{A}$ is a Heyting algebra, the inequality $\alpha\kandk (\alpha\krarrk \beta)\leq \beta$ holds. Also:
\begin{center}
\begin{tabular}{c ll}
& $\alpha\kandk (\alpha\krarrk \beta)\leq \beta$ &\\
iff & $e_{\oc}(\alpha\kandk (\alpha\krarrk \beta))\leq e_{\oc}(\beta)$ & ($e_{\oc}$ order-embedding)\\
iff & $e_{\oc}(\alpha)\mand e_{\oc}(\alpha\krarrk \beta)\leq e_{\oc}(\beta)$ & (assumption on $e_{\oc}$)\\
iff & $e_{\oc}(\alpha\krarrk \beta)\leq e_{\oc}(\alpha)\mrarr e_{\oc}(\beta).$ & (IL4)\\
\end{tabular}
\end{center}
Since $\iota$ is  monotone and $\iota\circ e_{\oc} = Id_{\mathbb{A}}$, this implies that:
\[\alpha \krarrk \beta = \iota (e_{\oc}(\alpha\krarrk \beta))\leq \iota(e_{\oc}(\alpha)\mrarr e_{\oc}(\beta)).\]
Conversely, for all $\alpha, \beta\in \mathbb{A}$,
\begin{center}
\begin{tabular}{c ll}
& $\iota(e_{\oc}(\alpha)\mrarr e_{\oc}(\beta)) \leq \alpha \krarrk \beta$ &\\
iff & $\iota(e_{\oc}(\alpha)\mrarr e_{\oc}(\beta)) \leq \alpha \krarrk \iota(e_{\oc}(\beta))$ & ($\iota\circ e_{\oc} = Id_{\mathbb{A}}$)\\
iff & $\alpha \kandk \iota(e_{\oc}(\alpha)\mrarr e_{\oc}(\beta)) \leq \iota(e_{\oc}(\beta))$ & (residuation in $\mathbb{A}$)\\
iff & $e_{\oc}(\alpha \kandk \iota(e_{\oc}(\alpha)\mrarr e_{\oc}(\beta))) \leq e_{\oc}(\beta)$ & ($e_{\oc}\dashv \iota$)\\
iff & $e_{\oc}(\alpha) \mand e_{\oc}(\iota(e_{\oc}(\alpha)\mrarr e_{\oc}(\beta))) \leq e_{\oc}(\beta)$ & (assumption on $e_{\oc}$)\\
iff & $e_{\oc}(\iota(e_{\oc}(\alpha)\mrarr e_{\oc}(\beta))) \leq e_{\oc}(\alpha) \mrarr e_{\oc}(\beta)$ & (IL4)\\
iff & $\iota(e_{\oc}(\alpha)\mrarr e_{\oc}(\beta)) \leq \iota(e_{\oc}(\alpha) \mrarr e_{\oc}(\beta))$, & ($e_{\oc}\dashv \iota$)\\
\end{tabular}
\end{center}
and the last inequality  is clearly true. This finishes the proof of LK6$_{\mathbb{A}}$. To show that $\mathbb{A}$ and $\mathbb{K}_{\oc}$ are isomorphic as Heyting algebras, notice that the domain of $\mathbb{K}_{\oc}$ is defined  as  $K_{\oc}: =\mathsf{Range}(\oc) = \mathsf{Range}(e_{\oc}\circ\iota)$. Since by assumption $\iota$ is surjective,   $K_{\oc} = \mathsf{Range}(e_{\oc})$, and since $e_{\oc}$ is an order-embedding,  $K_{\oc}$, regarded as a sub-poset of $\mathbb{L}$, is order-isomorphic (hence lattice-isomorphic) to the domain of $\mathbb{A}$ with its lattice order. Let $e'_{\oc}: \mathbb{K}_{\oc}\hookrightarrow\mathbb{L}$ and $\iota': \mathbb{L}\twoheadrightarrow\mathbb{K}_{\oc}$ the adjoint maps arising from $\oc$.   Let $e: \mathbb{A}\to \mathbb{K}_{\oc}$ denote the order-isomorphism between $\mathbb{A}$ and $\mathbb{K}_{\oc}$. Thus, $e_{\oc} = e'_{\oc}\circ e$ and $\iota' = e\circ \iota$. To finish the proof of item 1, we need to show that for all $\alpha, \beta\in \mathbb{A}$,
\[e(\alpha\krarrk_{\mathbb{A}}\beta) = e(\alpha)\krarrk_{\mathbb{K}_{\oc}}e(\beta).\]
\begin{center}
\begin{tabular}{c ll}
& $e(\alpha)\krarrk_{\mathbb{K}_{\oc}}e(\beta)$ & \\
$=$ & $\iota' (e'_{\oc}(e(\alpha))\mrarr e'_{\oc}(e(\beta)))$ & (definition of $\krarrk_{\mathbb{K}_{\oc}}$)\\
$=$ & $\iota' (e_{\oc}(\alpha)\mrarr e_{\oc}(\beta))$ & ($e_{\oc} = e'_{\oc}\circ e$)\\
$=$ & $e(\iota (e_{\oc}(\alpha)\mrarr e_{\oc}(\beta)))$ & ($\iota' = e\circ \iota$)\\
$=$ & $e(\alpha\krarrk_{\mathbb{A}}\beta)$. & (LK6$_{\mathbb{A}}$)\\
\end{tabular}
\end{center}
The proof of item 2 is similar to that of item 1 and is omitted. As to item 3, to finish the proof it is enough to show that, for all $a, b\in \mathbb{L}$,
\[\oc (a\mrarr b)\leq \wn a\mrarr \wn b. \]
\begin{center}
\begin{tabular}{c l l}
& $\oc (a\mrarr b)\leq \wn a\mrarr \wn b$ & \\
iff & $\wn a\leq \oc (a\mrarr b)\mrarr \wn b$ & (IL2 and IL4)\\
iff & $e_{\wn}(\gamma( a))\leq e_{\oc}(\iota (a\mrarr b))\mrarr e_{\wn}(\gamma( b))$ & ($\oc := e_{\oc}\circ\iota$ and $\wn: = e_{\wn}\circ\gamma$)\\
iff & $e_{\wn}(\gamma( a))\leq  e_{\wn}(\iota (a\mrarr b)\KKrarr \gamma( b))$ & (first assumption on $\KKrarr$)\\
iff & $\gamma( a)\leq  \iota (a\mrarr b)\KKrarr \gamma( b)$, & ($e_{\wn}$ order-embedding)\\
\end{tabular}
\end{center}
and the last inequality  is true by assumption. The proof of item 4 is order-dual to one of item 3 and is omitted.
\end{proof}

\subsection{Canonical extensions of linear algebras and their kernels}\label{ssec: canonical extensions}
In the previous subsection, we showed that  heterogeneous algebras for the various linear logics (cf.~Definition \ref{def:heterogeneous algebras}) are equivalent  presentations of linear algebras with exponentials, and hence can serve as  equivalent semantic structures for each linear logic, which can also be taken as the primary semantics. This change in perspective is particularly advantageous when it comes to defining the canonical extension of a linear algebra with exponential(s) in a way which uniformly applies general criteria. Indeed, the canonical extension  of a {\em normal} (distributive) lattice expansion $\mathbb{A} = (L, \mathcal{F}, \mathcal{G})$
(cf.~Definition \ref{def: normal (d)le}) is defined in a uniform way for any signature as the normal (distributive) lattice expansion $\mathbb{A}^\delta: = (L^\delta, \mathcal{F}^\sigma, \mathcal{G}^\pi)$, where $L^\delta$  is the canonical extension of  $L$ (cf.~Definition \ref{def:can:ext}), and $\mathcal{F}^\sigma: = \{f^\sigma\mid f\in \mathcal{F}\}$ and  $\mathcal{G}^\pi: = \{g^\pi\mid g\in \mathcal{G}\}$ (cf.~Definition \ref{def:f sigma and f pi}).

However, since the exponentials are {\em not normal} when regarded as operations on linear algebras, when taking them as primary we do not have general guidelines in choosing whether to take the $\sigma$- or the $\pi$-extension of each (cf.~Definition \ref{def:f sigma and f pi}),  given that  the $\sigma$-extensions and $\pi$-extensions of exponentials do not coincide in general, and   different settings or purposes might provide different reasons to choose one extension over the other. So we would need to motivate our choice on the basis of considerations which might not easily be portable to other settings  (as done in \cite{coumans2014relational}).  

In contrast, when {\em defining} exponentials   as compositions of pairs of adjoint maps, thanks to the fact that adjoint maps are {\em normal} (in the sense of Definition \ref{def: normal (d)le}) in a lattice-based environment,  the general criterion for defining the canonical extensions of normal lattice expansions can be straightforwardly exported to the heterogeneous algebras of Definition \ref{def:heterogeneous algebras}. Following the  general guidelines, we take the $\sigma$-extensions of left adjoint maps  and the $\pi$-extensions of the right adjoint maps (which are themselves adjoints, by the general theory), and then {\em define} the canonical extension of exponentials as   the composition of these. Since the adjoint maps are normal and unary, they have the extra benefit of being smooth (that is, their $\sigma$- and $\pi$-extensions coincide cf.~Section \ref{sec: background can ext}), but this is not essential. The essential aspect is that this definition is not taken on a case-by-case basis, but rather instantiates a general criterion.

As a practical benefit of this defining strategy,  we are now in a position  to obtain two key properties of  the identities and inequalities defining  the heterogeneous algebras of Definition \ref{def:heterogeneous algebras} as instances of general results in the theory of  unified correspondence  \cite{CoGhPa14,ConPalSou,CoPa12,PaSoZh15,CoPa:non-dist,CP:constructive,GMPTZ} for (multi-type) normal (distributive) lattice expansions.   Specifically, these inequalities  are all of a certain syntactic shape called {\em analytic inductive} (cf.~Definition \ref{def:analytic inductive ineq}). 
By unified correspondence theory,  (analytic) inductive inequalities (1) are  canonical (cf.~Theorem \ref{theor: analytic inductive is canonical}, \cite[Theorem 19]{GMPTZ}), and (2) can be equivalently encoded into analytic rules of a proper display calculus (cf.~\cite[Proposition 59]{GMPTZ}). Property (1) guarantees that the validity of all the identities and inequalities defining the heterogeneous algebras  of the lower rows of the diagrams in the statement of Proposition \ref{prop:canonical extensions} below transfers to the heterogeneous algebras in the upper rows of the same diagrams, and hence, the algebraic completeness of each original logical system transfers to the corresponding proper subclass of  {\em perfect} heterogeneous algebras. Moreover, the heterogeneous algebras  in the upper rows are such that all the various maps involved (including $\KKrarr^\pi$ and $\KKBrarr^\sigma$) are residuated in each coordinate, which implies that the display postulates relative to these connectives are sound. Property (2) guarantees that
the identities and inequalities defining the heterogeneous algebras of Definition \ref{def:heterogeneous algebras} can be equivalently encoded into (multi-type) analytic rules, which will form part of the calculi introduced in Section \ref{sec:ProperDisplayCalculiForLinearLogics}.

In what follows, we let $\mathbb{L}^\delta$, $\mathbb{A}^\delta$ and $\mathbb{B}^\delta$ denote the canonical extensions of $\mathbb{L}$, $\mathbb{A}$ and $\mathbb{B}$ respectively.
\begin{proposition}
\label{prop:canonical extensions}
For any algebra $\mathbb{L}$,
\begin{enumerate}
\item If $(\mathbb{L}, \mathbb{A}, e_{\oc}, \iota)$ is an heterogeneous ILS-algebra (resp.~heterogeneous CLS-algebra),
then $(\mathbb{L}^\delta, \mathbb{A}^\delta, e_{\oc}^\sigma, \iota^\pi)$ is a perfect heterogeneous ILS-algebra (resp.\ CLS-algebra).
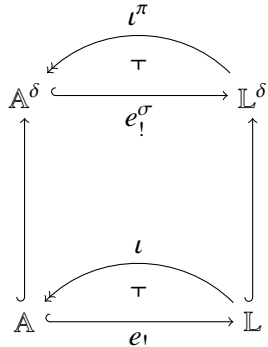
\begin{center}
\begin{tikzpicture}[node/.style={circle, draw, fill=black}, scale=1]

\node (A) at (-1.5,-1.5) {$\mathbb{A}$};
\node (A delta) at (-1.5,1.5) {$\mathbb{A}^{\delta}$};
\node (L) at (1.5,-1.5) {$\mathbb{L}$};
\node (L delta) at (1.5,1.5) {$\mathbb{L}^{\delta}$};
\node (adj) at (0,-1.1) {{\rotatebox[origin=c]{270}{$\vdash$}}};
\node (adju) at (0,1.9) {{\rotatebox[origin=c]{270}{$\vdash$}}};
\draw [right hook->] (A) to  (A delta);
\draw [right hook->] (L)  to (L delta);
\draw [right hook->] (A)  to node[below] {$e_{\oc}$}  (L);
\draw [right hook->] (A delta)  to node[below] {$e_{\oc}^{\sigma}$}  (L delta);

\draw [->>] (L delta) to [out=135,in=45, looseness=1]   node[above] {$\iota^{\pi}$}  (A delta);
\draw [->>] (L) to [out=135,in=45, looseness=1]   node[above] {$\iota$}  (A);
\end{tikzpicture}
\end{center}

\item If $(\mathbb{L}, \mathbb{A}, \mathbb{B}, e_{\oc}, \iota, e_{\wn}, \gamma)$ is an heterogeneous ILSC-algebra (resp.~heterogeneous BLSC-algebra),
then $(\mathbb{L}^\delta, \mathbb{A}^\delta, \mathbb{B}^\delta, e_{\oc}^\sigma, \iota^\pi, e_{\wn}^\pi, \gamma^\sigma)$ is a perfect heterogeneous ILSC-algebra (resp.\ BLSC-algebra).
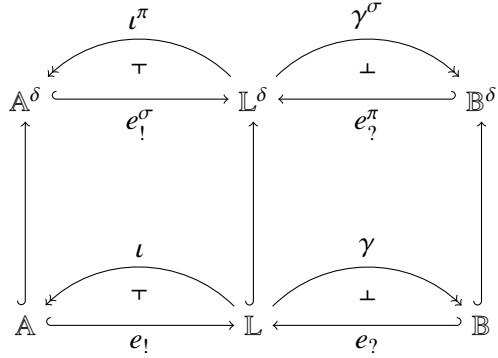
\begin{center}
\begin{tikzpicture}[node/.style={circle, draw, fill=black}, scale=1]

\node (A) at (-1.5,-1.5) {$\mathbb{A}$};
\node (A delta) at (-1.5,1.5) {$\mathbb{A}^{\delta}$};
\node (L) at (1.5,-1.5) {$\mathbb{L}$};
\node (L delta) at (1.5,1.5) {$\mathbb{L}^{\delta}$};
\node (B) at (4.5,-1.5) {$\mathbb{B}$};
\node (B delta) at (4.5,1.5) {$\mathbb{B}^{\delta}$};
\node (adj) at (0,-1.1) {{\rotatebox[origin=c]{270}{$\vdash$}}};
\node (adju) at (0,1.9) {{\rotatebox[origin=c]{270}{$\vdash$}}};
\node (adj) at (3,-1.1) {{\rotatebox[origin=c]{90}{$\vdash$}}};
\node (adju) at (3,1.9) {{\rotatebox[origin=c]{90}{$\vdash$}}};
\draw [right hook->] (A) to  (A delta);
\draw [right hook->] (L)  to (L delta);
\draw [right hook->] (B) to  (B delta);
\draw [right hook->] (A)  to node[below] {$e_{\oc}$}  (L);
\draw [right hook->] (A delta)  to node[below] {$e_{\oc}^{\sigma}$}  (L delta);
\draw [left hook->] (B)  to node[below] {$e_{\wn}$}  (L);
\draw [left hook->] (B delta)  to node[below] {$e_{\wn}^{\pi}$}  (L delta);

\draw [->>] (L delta) to [out=135,in=45, looseness=1]   node[above] {$\iota^{\pi}$}  (A delta);
\draw [->>] (L) to [out=135,in=45, looseness=1]   node[above] {$\iota$}  (A);
\draw [<<-] (B delta) to [out=135,in=45, looseness=1]   node[above] {$\gamma^{\sigma}$}  (L delta);
\draw [<<-] (B) to [out=135,in=45, looseness=1]   node[above] {$\gamma$}  (L);
\end{tikzpicture}
\end{center}

\item If $(\mathbb{L}, \mathbb{A}, \mathbb{B}, e_{\oc}, \iota, e_{\wn}, \gamma, \KKrarr)$ is an heterogeneous ILP-algebra,
then $$(\mathbb{L}^\delta, \mathbb{A}^\delta, \mathbb{B}^\delta, e_{\oc}^\sigma, \iota^\pi, e_{\wn}^\pi, \gamma^\sigma, \KKrarr^\pi)$$ is a perfect heterogeneous ILP-algebra. Hence,  \[\KKrarr^\pi: \mathbb{A}^\delta\times \mathbb{B}^\delta\longrightarrow \mathbb{B}^\delta\] has residuals $\KKBand: \mathbb{A}^\delta\times \mathbb{B}^\delta\longrightarrow \mathbb{B}^\delta$ and $\KKlarr : \mathbb{B}^\delta\times \mathbb{B}^\delta\longrightarrow \mathbb{A}^\delta$ in each coordinate.

\item If $(\mathbb{L}, \mathbb{A}, \mathbb{B}, e_{\oc}, \iota, e_{\wn}, \gamma, \KKrarr, \KKBrarr)$  is an heterogeneous BLP-algebra,
then $$(\mathbb{L}^\delta, \mathbb{A}^\delta, \mathbb{B}^\delta, e_{\oc}^\sigma, \iota^\pi, e_{\wn}^\pi, \gamma^\sigma, \KKrarr^\pi, \KKBrarr^\sigma)$$ is a perfect heterogeneous ILP-algebra. Hence, not only $\KKrarr^\pi$ has residuals as in the item above, but also \[\KKBrarr^\sigma: \mathbb{B}^\delta\times \mathbb{A}^\delta\longrightarrow \mathbb{A}^\delta\] has residuals $\KKor: \mathbb{B}^\delta\times \mathbb{A}^\delta\longrightarrow \mathbb{A}^\delta$ and $\KKBlarr: \mathbb{A}^\delta\times \mathbb{A}^\delta\longrightarrow \mathbb{B}^\delta$ in each coordinate.
\end{enumerate}
\end{proposition}
\begin{proof}
As to item 1, it is a basic fact (cf.\ Section \ref{sec: background can ext}) that  $e_{\oc}\dashv \iota$ implies that  $e_{\oc}^\sigma\dashv \iota^\pi$. This in turn implies that $\iota^{\pi}(e_{\oc}^\sigma (\upsilon))\leq \upsilon$ for every $\upsilon\in \mathbb{A}^\delta$. The converse inequality $ \upsilon\leq \iota^{\pi}(e_{\oc}^\sigma (\upsilon))$ also holds, since the original inequality $\alpha\leq \iota(e_{\oc} (\alpha))$ is valid in $\mathbb{A}$ and  is analytic inductive (cf.\ Definition \ref{def:analytic inductive ineq}), and hence  canonical (cf.~Theorem \ref{theor: analytic inductive is canonical}). The identity $e_{\oc}^\sigma(\ktop) = \mtop$ clearly holds, since $\mathbb{A}$ and $\mathbb{L}$ are subalgebras of $\mathbb{A}^\delta$ and $\mathbb{L}^\delta$ respectively, $e_{\oc}^\sigma$ coincides with $e_{\oc}$ on $\mathbb{A}$ and $e_{\oc}(\ktop) = \mtop$. Finally, the two inequalities $e_{\oc}(\alpha)\mand e_{\oc}(\beta) \leq  e_{\oc}(\alpha\kandk \beta)$ and  $e_{\oc}(\alpha\kandk \beta)\leq e_{\oc}(\alpha)\mand e_{\oc}(\beta)$ are also analytic inductive, hence canonical, which completes the proof that $(\mathbb{L}^\delta, \mathbb{A}^\delta, e_{\oc}^\sigma, \iota^\pi)$ is an heterogeneous ILS-algebra (resp.\ CLS-algebra), which is also perfect, since $\mathbb{L}^\delta$ and $\mathbb{A}^\delta$ are perfect (cf.\ Section \ref{sec: background can ext}). The remaining items are proved in a similar way, observing that all the inequalities mentioned in these statements are analytic inductive (cf.~Definition \ref{def:analytic inductive ineq}), hence canonical. The existence of the residuals of $\KKrarr$ and $\KKBrarr$, as well as the claim that the heterogeneous algebras are perfect, can be argued as follows: by a proof analogous to the proof of the first (resp.\ second) item (b) of Proposition \ref{prop:order theoret properties eterogeneous arrows}, one shows that  $\KKrarr$ (resp.\ $\KKBrarr$) is finitely meet-preserving (resp.\ join-preserving) in its second coordinate and finitely join-reversing (resp.\ meet-reversing) in its first coordinate. As discussed in Section \ref{sec: background can ext}, this implies that $\KKrarr^\pi$ (resp.\ $\KKBrarr^\sigma$) is completely meet-preserving (resp.\ join-preserving) in its second coordinate and completely join-reversing (resp.\ meet-reversing) in its first coordinate. Since these maps are defined between complete lattices, this is sufficient to infer the existence of the required residuals.
\end{proof}
The following is an immediate consequence of Propositions \ref{prop:from single to multi}, \ref{prop:reverse engineering} and \ref{prop:canonical extensions}.
\begin{corollary}
\label{cor:canonical ext}
For any algebra $\mathbb{L}$,
\begin{enumerate}
\item if $(\mathbb{L}, \oc)$ is an ILS-algebra (resp.\ CLS-algebra),
then $\mathbb{L}^\delta$ can be endowed with the structure of ILS-algebra (resp.\ CLS-algebra) by defining  $\oc^\delta: \mathbb{L}^\delta\rightarrow \mathbb{L}^\delta$ by $\oc^\delta : = e_{\oc}^\sigma\circ \iota^\pi$. Moreover, $\mathbb{K}_{\oc}^\delta\cong \mathbb{K}_{\oc^\delta}$.

\item If $(\mathbb{L}, \oc, \wn)$ is an ILSC-algebra (resp.\ BLSC-algebra),
then $\mathbb{L}^\delta$ can be endowed with the structure of ILSC-algebra (resp.\ BLSC-algebra) by defining $\oc^\delta$ as in the item above and $\wn^\delta: \mathbb{L}^\delta\rightarrow \mathbb{L}^\delta$ by $\wn^\delta : = e_{\wn}^\pi\circ \gamma^\sigma$. Moreover, $\mathbb{K}_{\oc}^\delta\cong \mathbb{K}_{\oc^\delta}$ and $\mathbb{K}_{\wn}^\delta\cong \mathbb{K}_{\wn^\delta}$.
\item If $(\mathbb{L}, \oc, \wn)$ is an ILP-algebra  (resp.\ BLP-algebra),
then $\mathbb{L}^\delta$ can be endowed with the structure of ILP-algebra (resp.\ BLP-algebra) by defining $\oc^\delta$ and $\wn^\delta$ as in the items above.
\end{enumerate}
\end{corollary}

\section{Multi-type Hilbert-style presentation for linear logic}\label{sec: multi-type language}
In Section \ref{ssec: reverse engineering}, the heterogeneous algebras associated with the various linear logics have been introduced (cf.~Definition  \ref{def:heterogeneous algebras})  and shown to be  equivalent presentations of  linear algebras with exponentials. These constructions  motivate from a semantic perspective the syntactic shift we take in the present section, from the original single-type language to a multi-type language. Indeed, the heterogeneous algebras of Definition  \ref{def:heterogeneous algebras} provide a natural interpretation for the following multi-type language $\mathcal{L}_{\mathrm{MT}}$,\footnote{\label{footnote:comparison Benton} There are clear similarities between $\mathcal{L}_{\mathrm{MT}}$ and the language of Linear Non Linear logic LNL \cite{Benton1994}, given that they both aim at capturing the interplay between the linear and the non linear behaviour. However, there are also differences: for instance, in $\mathcal{L}_{\mathrm{MT}}$, only the $\mathsf{Linear}$ type has atomic propositions, whereas in LNL each type has its own atomic propositions. This difference  reflects a difference in the aims of \cite{Benton1994} and of the present paper:  while \cite{Benton1994} aims at studying the environment of adjunction models in their own right,  the present paper aims at studying Girard's linear logic and its variants through the lenses of the multi-type environment, and hence focuses on the specific multi-type language adequate for this task. As we will discuss in the following section, we will present a slightly different version of this language, which accounts for 
the residuals of $\KKrarr$ and $\KKBrarr$ in each coordinate, and  the residuals of $\kor$ and $\andk$. Finally, in the next section we will use a different notation for the heterogeneous unary connectives, which is aimed at emphasizing their standard proof-theoretic behaviour rather than their intended algebraic interpretation.} defined by simultaneous induction  from a given set $\mathsf{AtProp}$ of atomic propositions (the elements of which are denoted by letters $p, q$):\footnote{We specify the language corresponding to BLP-algebras, which is the richest signature. The multi-type languages corresponding to the other linear algebras are defined analogously,  suitably omitting the defining clauses which are not applicable.}
\begin{align*}
\mathsf{\oc\text{-}Kernel}\ni \alpha ::=&\, \iota(A) \mid \ktop \mid \kbot \mid \alpha \kor \alpha \mid  \alpha \kand \alpha \mid \alpha \krarr \alpha \mid \xi\KKBrarr \alpha \\
\mathsf{\wn\text{-}Kernel}\ni \xi ::=&\, \gamma(A) \mid \topk \mid \botk \mid \xi \andk \xi \mid \xi \ork \xi \mid \xi \drarrk \xi \mid \alpha\KKrarr \xi \\
\mathsf{Linear}\ni  A ::= & \,p \mid \, e_{\oc}(\alpha) \mid e_{\wn}(\xi) \mid \mtop \mid \mbot \mid A\mneg \mid A \mand A \mid A \mor A \mid A \mrarr A \mid A \mdrarr A\mid \aatop \mid \abot \mid A \aand A \mid A \aor A
\end{align*}
The interpretation of $\mathcal{L}_{\mathrm{MT}}$-terms  into heterogeneous algebras of compatible signature is defined as the straightforward generalization of the interpretation of propositional languages in algebras of compatible signature.

The toggle between linear algebras with exponentials and heterogeneous algebras (cf.\ Sections \ref{ssec: linear algebras}, \ref{ssec:from ono axioms to eterogeneous arrows} and  \ref{ssec: reverse engineering}) is reflected syntactically by the following translation $(\cdot)^t: \mathcal{L}\to \mathcal{L}_{\mathrm{MT}}$ 
between the original language(s) $\mathcal{L}$ of linear logic(s)  and (their corresponding multi-type languages)  $\mathcal{L}_{\mathrm{MT}}$:

\begin{center}
\begin{tabular}{r c l c r c l}
$p^t$ &$ = $& $p$ & $\quad\quad$ &$(A\mneg)^t$ &$ = $& $(A^t)\mneg$\\
$\aatop^t$ &$ = $& $\aatop$ && $\abot^t$ &$ = $& $\abot$\\
$\mtop^t$ &$ = $& $\mtop$ && $\mbot^t$ &$ = $& $\mbot$\\
$(A\aand B)^t$ &$ = $& $A^t \aand B^t$ && $(A\aor B)^t$ &$ = $& $A^t \aor B^t$\\
$(A\mand B)^t$ &$ = $& $A^t \mand B^t$ && $(A\mor B)^t$ &$ = $& $A^t \mor B^t$\\
$(A\mrarr B)^t$ &$ = $& $A^t \mrarr B^t$ && $(A\mdrarr B)^t$ &$ = $& $A^t \mdrarr B^t$\\
$(\oc A)^t$ &$ = $& $ e_{\oc}\iota(A^t)$ && $(\wn A)^t$ &$ = $& $e_{\wn}\gamma (A^t)$
\end{tabular}
\end{center}
Not only does the translation  $(\cdot)^t: \mathcal{L}\to \mathcal{L}_{\mathrm{MT}}$ 
elicit the switch from the single-type language to the multi-type language, but is also compatible with the underlying toggle between linear algebras with exponentials and their associated heterogeneous algebras. Indeed, for  every $\mathcal{L}$-algebra $\mathbb{L}$, let $\mathbb{L}^*$ denote its associated heterogeneous algebra (cf.~Proposition \ref{prop:from single to multi}). The following proposition is proved by a routine induction on  $\mathcal{L}$-formulas, using the deduction-detachment theorem of linear logic.
\begin{proposition}
\label{prop:consequence preserved and reflected}
For all $\mathcal{L}$-formulas $A$ and $B$ and every $\mathcal{L}$-algebra $\mathbb{L}$,
\[\mathbb{L}\models A\leq B \quad \mbox{ iff }\quad \mathbb{L}^*\models A^t\leq B^t.\]
\end{proposition}

The main technical difference between the single-type and the multi-type settings is that, while $\oc$ and $\wn$ are not {\em normal} (i.e.\ their algebraic interpretations are not finitely join-preserving or meet-preserving), all  connectives in $\mathcal{L}_{\mathrm{MT}}$ are normal, which allows to apply the {\em standard} proof-theoretic treatment for normal connectives to them (e.g.\ to associate each connective to its structural counterpart, have sound display rules etc.), according to the general definitions and results of multi-type algebraic proof theory \cite{GrecoPalmigianoMultiTypeAlgebraicProofTheory}. In particular, the general definition of {\em analytic inductive} inequalities can be instantiated to inequalities in the $\mathcal{L}_{\mathrm{MT}}$-signature (cf.\ Definition \ref{def:analytic inductive ineq}). Hence, we are now in a position to translate the identities and inequalities for the interpretations of the exponentials in linear algebras into $\mathcal{L}_{\mathrm{MT}}$ using $(\cdot)^t$, and verify whether the resulting translations are analytic inductive.
%
%

\begin{center}
\begin{tabular}{r l}
$\oc(A\aand B) = \oc A \mand \oc B\ \rightsquigarrow $ &
  $\begin{cases}
   e_{\oc}\iota(A\aand B) \leq e_{\oc}\iota A \mand e_{\oc}\iota B & (i)\\
   e_{\oc}\iota A \mand e_{\oc}\iota B\leq e_{\oc}\iota(A\aand B) & (ii)
  \end{cases}
  $
  \\
$\oc\aatop = \mtop\ \rightsquigarrow$ &
$ \begin{cases}
   e_{\oc}\iota\aatop \leq \mtop & (iii)\\
   \mtop \leq e_{\oc}\iota\aatop  & (iv)
  \end{cases}
  $
  \\
  $
\wn(A\aor B) = \wn A \mor \wn B\ \rightsquigarrow $&
  $\begin{cases}
   e_{\wn}\gamma(A\aor B) \leq e_{\wn}\gamma A \mor e_{\wn}\gamma B & (v)\\
   e_{\wn}\gamma A \mor e_{\wn}\gamma B\leq e_{\wn}\gamma(A\aor B) & (vi)
  \end{cases}
  $
  \\

$
\wn\abot = \mbot\ \rightsquigarrow $&
$  \begin{cases}
   e_{\wn}\gamma\abot \leq \mbot & (vii)\\
   \mbot \leq e_{\wn}\gamma\abot  & (viii)
  \end{cases}
  $
  \\

$\oc (A\mrarr B)\leq \wn A\mrarr \wn B\rightsquigarrow$ &
$\; e_{\oc}\iota(A\mrarr B) \leq e_{\wn}\gamma A\mrarr e_{\wn}\gamma B \quad (ix)$
   \\

$\oc A\mdrarr \oc B\leq \wn (A\mdrarr  B) \rightsquigarrow$ &
$\; e_{\oc}\iota A\mdrarr e_{\oc}\iota B\leq e_{\wn}\gamma(A\mdrarr B)  \quad (x)$
\\
  \end{tabular}
  \end{center}
It is easy to see that $(iii)$, $(iv)$, $(vii)$ and $(viii)$ are the only analytic inductive inequalities of the list above.  Indeed, recall (cf.\ Definition \ref{def:analytic inductive ineq}) that
%
$e_{\oc}$, $\gamma$, $\mand$, $\aand$ and $\mdrarr$ (resp.\ $e_{\wn}$, $\iota$, $\mor$, $\aor$ and $\mrarr$) are $\mathcal{F}$-connectives (resp.\ $\mathcal{G}$-connectives), since their interpretations preserve finite joins (resp.\ meets) in each positive coordinate and reverse finite meets (resp.\ joins) in each negative coordinate.\footnote{Recall that, for the sake of the present paper, we have confined ourselves to distributive linear logic, but the failure of analyticity transfers of course also to the non-distributive setting.} Then $(i)$ and $(ii)$ violate analyticity because of $\iota$ occurring in the scope of $e_{\oc}$ in the right-hand side,
 $(v)$ and $(vi)$  because of $\gamma$ occurring in the scope of $e_{\wn}$ in the left-hand side, and $(ix)$ and $(x)$ because of the subterms $e_{\wn}\gamma A$ and $e_{\oc}\iota A$ respectively.

In the light of the general result  characterizing  analytic inductive inequalities as exactly those  equivalently captured by analytic rules of proper display calculi (cf.~\cite[Propositions 59 and 61]{GMPTZ}), the failure of the inequalities above to be analytic inductive  gives a clear identification of the main hurdle towards the definition of a {\em proper} display calculus for linear logic. 
%

%
However, the order-theoretic analysis developed in Section \ref{sec: semantic environment} also provides a pathway to a solution:

\begin{proposition}\label{prop: from non analytic to analytic}
Each  (in)equality in the left column of the following table is semantically  equivalent  on heterogeneous algebras of the appropriate signature to the corresponding (in)equality in the right column: 
\begin{center}
\begin{tabular}{r c|c l}
$e_{\oc}\iota(A\aand B)  =  e_{\oc}\iota A \mand e_{\oc}\iota B $  &&& $e_{\oc}(\alpha\kand \beta)  =   e_{\oc}\alpha \mand e_{\oc}\beta$\\
$e_{\oc}\iota\aatop  =  \mtop$ &&& $e_{\oc}\ktop = \mtop$ \\
$e_{\wn}\gamma(A\aor B)  =  e_{\wn}\gamma A \mor e_{\wn}\gamma B$ &&& $e_{\wn}(\xi\ork \chi)  =  e_{\wn}\xi \mor e_{\wn}\chi$\\
$e_{\wn}\gamma\abot  =  \mbot$ &&& $e_{\wn}\botk = \mbot$\\
$e_{\oc}\iota(A\mrarr B) \leq e_{\wn}\gamma A\mrarr e_{\wn}\gamma B$ &&& $\gamma A\leq  \iota (A\mrarr B)\KKrarr \gamma B$\\
$e_{\oc}\iota A\mdrarr e_{\oc}\iota B\leq e_{\wn}\gamma(A\mdrarr B)$  &&& $\gamma( A\mdrarr B)\KKBrarr \iota B\leq  \iota A$\\
\end{tabular}
\end{center}
 \end{proposition}
 \begin{proof}
The proof of the equivalence of the identities in the rows from the first to the fourth immediately follows from the fact that the map $\iota$ (resp.\ $\gamma$) is surjective and preserves finite meets (resp.\ joins). The equivalences of the last two rows are shown in the proof of Proposition \ref{prop:reverse engineering}.  
\end{proof}
The identities and inequalities in the right column of the statement of Proposition \ref{prop: from non analytic to analytic} can be then taken as an alternative multi-type Hilbert-style presentation of the ($\{\oc, \wn\}$-fragment of) linear logic.
Finally, it is easy to verify that these identities and inequalities  are all analytic inductive (cf.~Definition \ref{def:analytic inductive ineq}), and hence  can be equivalently encoded into analytic rules of a proper (multi-type) display calculus. In the next section, we introduce the calculi resulting from this procedure. 

\section{Proper display calculi for linear logics}
\label{sec:ProperDisplayCalculiForLinearLogics}

In the present section, we introduce  display calculi for the various linear logics captured by the algebras of Definition \ref{def:linear algebras}. As is typical of similar existing calculi, the language manipulated by each of these calculi is built up from structural and operational (aka logical) term constructors. In the tables of Section \ref{ssec:language of DLL}, each structural symbol in the upper rows corresponds to one or two logical symbols in the lower rows. The idea, which will be made precise in Section \ref{ssec:soundness}, is that the interpretation of each structural connective coincides with that of the corresponding logical connective on the left-hand (resp.~right-hand) side (if it exists) when occurring in precedent (resp.~succedent) position.

The language $\mathcal{L}_{\mathrm{MT}}$ introduced in the previous section and the language introduced in the following subsection are clearly related. However, there are  differences. Besides the fact that the language below has an extra layer of structural connectives, the main difference is that the pure kernel-type connectives are represented only at the {\em structural} level (as are the heterogeneous binary connectives and their residuals).\footnote{In the synoptic tables of the next subsection, the operational symbols which are represented only at the structural level will appear between round brackets.} This choice is in line with the main aim of the present paper, which revolves around the original system of linear logic defined by Girard, and its intuitionistic and bi-intuitionistic variants. Accordingly, we include at the operational level only the connectives that are directly involved in capturing original linear formulas. Nonetheless,  calculi for the logics of the various heterogeneous algebras would be easily obtainable as  variants of the calculi introduced below, just by adding their corresponding standard introduction rules.

\subsection{Language}
\label{ssec:language of DLL}

 In the present subsection, we introduce the language of the display calculi for the various linear logics (we will use D.LL to refer to them collectively). Below, we introduce the richest signature, i.e.~the one intended to capture  the linear logic of BLP-algebras. This signature includes the types $\mathsf{Linear}$, $\mathsf{\oc\textrm{-}Kernel}$, and $\mathsf{\wn\textrm{-}Kernel}$, sometimes abbreviated as $\mathsf{L}$, $\mathsf{K_\wn}$, and $\mathsf{K_\wn}$ respectively.




\begin{center}
\begin{tabular}{l}
$\mathsf{\phantom{_\wn}L} \left\{\begin{array}{l}
A  ::= \,p \mid \mtop \mid \mbot \mid A \mand A \mid A \mor A \mid A \mrarr A \mid A \mdrarr A \mid \aatop \mid \abot \mid A \aand A \mid A \aor A \mid \wdia \alpha \mid \boxw \xi \\
 \\
X ::= \,A \mid \MTOPBOT \mid X \MANDOR X \mid X \MARR X\mid \ATOPBOT  \mid X\AANDOR X\mid X\AARR X \mid \WCIRC \Gamma \mid \CIRCW \Pi \\
\end{array} \right.$
 \\

 \\

$\mathsf{K_\oc} \left\{\begin{array}{l}
\alpha     ::= \, \bbox A \\
 \\
\Gamma ::= \, \BCIRC X \mid \KTOPBOT \mid \Gamma \KANDOR \Gamma \mid \Gamma \KARR \Gamma \mid \Pi \KKOR\Gamma \mid \Pi \KKBARR \Gamma \mid \Pi \KKLARR \Pi \\
\end{array} \right.$
 \\

 \\

$\mathsf{K_\wn} \left\{\begin{array}{l}
\xi     ::= \, \diab A \\
 \\
\Pi ::= \, \CIRCB X \mid \TOPBOTK \mid \Pi \ANDORK \Pi \mid \Pi \ARRK \Pi \mid \Gamma \KKBAND \Pi \mid  \Gamma \KKRARR\Pi \mid \Gamma \KKBLARR  \Gamma \\
\end{array} \right.$
 \\
\end{tabular}
\end{center}

Our notational conventions assign different variables to different types, and hence allow us to drop the subscripts of the pure $\mathsf{Kernel}$ connectives and of the unary multi-type connectives, given that the parsing of expressions such as $\Pi \KARRK \Pi$ and $\wdiaw\Gamma$ is unambiguous.
\begin{itemize}
\item Structural and operational pure $\mathsf{K}$-type connectives:
\end{itemize}
\begin{center}
\begin{tabular}{|c|c|c|c|c|c|c|c|c|c|c|c|}
\hline
\mc{6}{|c|}{\fns $\mathsf{K}_{\oc}$-type connectives} & \mc{6}{|c|}{\fns $\mathsf{K}_{\wn}$-type connectives} \\
\hline
\mc{2}{|c|}{$\KTOPBOT$} & \mc{2}{c|}{$\KANDOR$} & \mc{2}{c|}{$\KARR$} & \mc{2}{|c|}{$\TOPBOTK$} & \mc{2}{c|}{$\ANDORK$} & \mc{2}{c|}{$\ARRK$}   \\
\hline
$(\ktop)$ & $(\kbot)$        & $(\kand)$ & $(\kor)$       & $(\kdrarr)$ & \mc{1}{c|}{$(\krarr)$} & \mc{1}{|c|}{$(\topk)$} & $(\botk)$        & $(\andk)$ & $(\ork)$       & $(\drarrk)$ & $(\rarrk)$ \\
\hline
\end{tabular}
\end{center}

\begin{itemize}
\item Structural and operational pure $\mathsf{L}$-type connectives:
\end{itemize}
\begin{center}
\begin{tabular}{|c|c|c|c|c|c|c|c|c|c|c|c|}
\hline
\mc{6}{|c|}{\fns Multiplicative connectives} & \mc{6}{|c|}{\fns Additive connectives} \\
\hline
\mc{2}{|c|}{$\MTOPBOT$} & \mc{2}{c|}{$\MANDOR$} & \mc{2}{c|}{$\MARR$} & \mc{2}{|c|}{$\ATOPBOT$} & \mc{2}{c|}{\,$\AANDOR$} & \mc{2}{c|}{$\AARR$} \\

\hline
$\mtop$ & $\mbot$ & $\mand$ & $\mor$ & $\mdrarr$ & \mc{1}{c|}{$\mrarr$} & \mc{1}{|c}{$\aatop$} & $\abot$ & $\aand$ & $\aor$ & $(\adrarr)$ & $(\ararr)$ \\
\hline
\end{tabular}
\end{center}

\begin{itemize}
\item Structural and operational unary multi-type connectives:
\end{itemize}
\begin{center}
\begin{tabular}{|c|c|c|c|c|c|c|c|}
\hline
\mc{2}{|c|}{\fns $\mathsf{L} \to \mathsf{K_\oc}$} & \mc{2}{c|}{\fns $\mathsf{L} \to \mathsf{K_\wn}$} & \mc{2}{|c|}{\fns $\mathsf{K_\oc} \to \mathsf{L}$} & \mc{2}{c|}{\fns $\mathsf{K_\wn} \to \mathsf{L}$} \\
\hline
\mc{2}{|c|}{$\BCIRC$}               & \mc{2}{c|}{$\CIRCB$}               & \mc{2}{|c|}{$\WCIRC$} & \mc{2}{c|}{$\CIRCW$}           \\
\hline
$\phantom{\bbox}$ & $\bbox$ & $\diab$ & \mc{1}{c|}{$\phantom{\diab}$} & \mc{1}{|c}{$\wdia$} & $\phantom{\wdia}$ & $\phantom{\boxw}$ & $\boxw$ \\
\hline
\end{tabular}
\end{center}
The connectives $\bbox$,  $\diab$, $\wdia$ and $\boxw$ are interpreted  in heterogeneous algebras of appropriate signature as the maps $\iota$, $\gamma$, $e_\oc$ and $e_\wn$ respectively. Exponentials in the language D.LL are {\em defined}  as follows:
\begin{center}
\begin{tabular}{rcl}
$\oc A$         & $::=$ & $\wdia \bbox A$                          \\
$\wn A$        & $::=$ & $\boxw \diab A$                          \\
\end{tabular}
\end{center}
In what follows, we will omit the subscripts of the unary modalities.
\begin{itemize}
\item Structural and operational binary multi-type connectives:
\end{itemize}
\begin{center}
\begin{tabular}{|c|c|c|c|c|c|c|c|c|c|c|c|}
\hline
\mc{2}{|c|}{\fns $\mathsf{K_\oc} \times \mathsf{K_\wn} \to \mathsf{K_\wn}$} & \mc{2}{c|}{\fns $\mathsf{K_\oc} \times \mathsf{K_\wn} \to \mathsf{K_\wn}$} & \mc{2}{c|}{\fns $\mathsf{K_\wn} \times \mathsf{K_\wn} \to \mathsf{K_\oc}$} & \mc{2}{c|}{\fns $\mathsf{K_\wn} \times \mathsf{K_\oc} \to \mathsf{K_\oc}$} & \mc{2}{c|}{\fns $\mathsf{K_\wn} \times \mathsf{K_\oc} \to \mathsf{K_\oc}$} & \mc{2}{c|}{\fns $\mathsf{K_\oc} \times \mathsf{K_\oc} \to \mathsf{K_\wn}$}        \\
\hline
\mc{2}{|c|}{\ \,$\KKBAND$}                  & \mc{2}{c|}{$\KKRARR$}                         & \mc{2}{c|}{$\KKLARR$} & \mc{2}{c|}{\ \,$\KKOR$}                  & \mc{2}{c|}{$\KKBARR$}                           & \mc{2}{c|}{$\KKBLARR$}\\
\hline
$\ (\KKBand)\ $ & $\phantom{(\KKBand)}$ & $\phantom{(\KKrarr)}$ & $(\KKrarr)$ & $\phantom{(\KKlarr)}$ & $(\KKlarr)$ & $\ \phantom{(\KKor)\ }$ & $(\KKor)$ & $(\KKBrarr)$ & $\phantom{(\KKBrarr)}$ & $(\KKBlarr)$ & $\phantom{(\KKBlarr)}$ \\
\hline
\end{tabular}
\end{center}

\subsection{Rules}
\label{ssec:Rules}

In what follows, structures of type $\mathsf{Linear}$ are denoted by the variables $X, Y, Z$, and $W$; structures of type $\mathsf{\oc\text{-}Kernel}$ are denoted by the variables $\Gamma, \Delta, \Theta$, and $\Lambda$; structures of type $\mathsf{\wn\text{-}Kernel}$ are denoted by the variables $\Pi, \Sigma, \Psi$, and $\Omega$. With these stipulations, in the present subsection we  omit the subscripts of  pure  $\mathsf{Kernel}$-type structural connectives and  unary multi-type structural connectives.


\subsubsection*{Basic intuitionistic linear environment}
\label{sssec:DILL}

\begin{itemize}
\item Identity and cut rules
\end{itemize}
\begin{center}
\begin{tabular}{rl}
\AXC{\phantom{$p \fCenter p$}}
\LeftLabel{\fns $Id_{\mathsf{\,L}}$}
\UI$p \fCenter p$
\DisplayProof
 &
\AX$X \fCenter A$
\AX$A \fCenter Y$
\RightLabel{\fns $Cut_{\mathsf{\,L}}$}
\BI$X \fCenter Y$
\DisplayProof
 \\
 & \\
\AX$\Gamma \fCenter \alpha$
\AX$\alpha \fCenter \Delta$
\LeftLabel{\fns $Cut_\oc$}
\BI$\Gamma \fCenter \Delta$
\DisplayProof
 &
\AX$\Pi \fCenter \xi$
\AX$\xi \fCenter \Sigma$
\RightLabel{\fns $Cut_\wn$}
\BI$\Pi \fCenter \Sigma$
\DisplayProof
 \\
\end{tabular}
\end{center}

\begin{itemize}
\item Pure $\mathsf{Linear}$-type display  rules
\end{itemize}
\begin{center}
\begin{tabular}{rl}
\AX$X \AANDOR Y \fCenter Z$
\LeftLabel{\fns res$_a$}
\doubleLine
\UI$Y \fCenter X \AARR Z$
\DisplayProof
 &
\AX$X \fCenter Y \AANDOR Z$
\RightLabel{\fns res$_a$}
\doubleLine
\UI$Y \AARR X \fCenter Z$
\DisplayProof
 \\
 & \\
\AX$X \MANDOR Y \fCenter Z$
\LeftLabel{\fns res$_m$}
\doubleLine
\UI$Y \fCenter X \MARR Z$
\DisplayProof
 &
\AX$X \fCenter Y \MANDOR Z$
\RightLabel{\fns res$_m$}
\doubleLine
\UI$Y \MARR X \fCenter Z$
\DisplayProof
 \\
\end{tabular}
\end{center}

\begin{itemize}
\item Pure $\mathsf{Kernel}$-type display rules
\end{itemize}
\begin{center}
\begin{tabular}{rl}
\AX$\Gamma \KANDORK \Delta \fCenter \Theta$
\doubleLine
\LeftLabel{\fns res$_\oc$}
\UI$\Delta \fCenter \Gamma \KARRK \Theta$
\DisplayProof
 &
\AX$\Theta \fCenter\Gamma \KANDORK \Delta$
\doubleLine
\RightLabel{\fns res$_\oc$}
\UI$\Gamma \KARRK \Theta \fCenter  \Delta$
\DisplayProof
 \\

 &\\

\AX$\Pi \KANDORK \Sigma \fCenter \Psi$
\doubleLine
\LeftLabel{\fns res$_\wn$}
\UI$\Sigma \fCenter \Pi \KARRK \Psi$
\DisplayProof
 &
\AX$\Psi\fCenter \Pi \KANDORK \Sigma $
\doubleLine
\RightLabel{\fns res$_\wn$}
\UI$\Pi \KARRK \Psi \fCenter \Sigma$
\DisplayProof
 \\
 \end{tabular}
 \end{center}

\begin{itemize}
\item Multi-type display  rules
\end{itemize}
\begin{center}
\begin{tabular}{rl}
\AX$\Gamma \fCenter \BCIRCB X$
\LeftLabel{\fns adj$_{\oc \mathsf{L}}$}
\doubleLine
\UI$\WCIRCW \Gamma \fCenter X$
\DisplayProof
 &
\AX$X \fCenter \WCIRCW \Pi$
\RightLabel{\fns adj$_{\wn \mathsf{L}}$}
\doubleLine
\UI$\BCIRCB X \fCenter \Pi$
\DisplayProof
 \\
 \end{tabular}
 \end{center}

\begin{itemize}
\item Pure $\mathsf{Linear}$-type structural rules
\end{itemize}
\begin{center}
\begin{tabular}{@{}r@{}l@{}}


\!\!\!\!\!\!
\begin{tabular}{rl}
\mc{2}{c}{additive} \\
 & \\
\AX$X \fCenter Y$
\doubleLine
\LeftLabel{\fns $\ATOPBOT$}
\UI$X \AANDOR \ATOPBOT \fCenter Y$
\DisplayProof
 &
\AX$X \fCenter Y$
\doubleLine
\RightLabel{\fns $\ATOPBOT$}
\UI$X \fCenter Y \AANDOR \ATOPBOT$
\DisplayProof
 \\

 & \\

\AX$X \AANDOR Y \fCenter Z$
\LeftLabel{\fns $E_a$}
\UI$Y \AANDOR X \fCenter Z $
\DisplayProof
 &
\AX$X \fCenter Y \AANDOR Z$
\RightLabel{\fns $E_a$}
\UI$X \fCenter Z \AANDOR Y$
\DisplayProof \\

 & \\

\AX$X \AANDOR (Y \AANDOR Z) \fCenter W$
\doubleLine
\LeftLabel{\fns $A_a$}
\UI$(X \AANDOR Y) \AANDOR Z \fCenter W $
\DisplayProof
 &
\AX$X \fCenter (Y \AANDOR Z) \AANDOR W$
\RightLabel{\fns $A_a$}
\doubleLine
\UI$X \fCenter Y \AANDOR (Z \AANDOR W)$
\DisplayProof
 \\

 & \\

\AX$X \fCenter Y$
\LeftLabel{\fns $W_a$}
\UI$X \AANDOR Z \fCenter Y$
\DisplayProof
 &
\AX$X \fCenter Y$
\RightLabel{\fns $W_a$}
\UI$X \fCenter Y \AANDOR Z$
\DisplayProof
 \\

 & \\

\AX$X \AANDOR X \fCenter Y$
\LeftLabel{\fns $C_a$}
\UI$X \fCenter Y$
\DisplayProof
 &
\AX$Y \fCenter X \AANDOR X$
\RightLabel{\fns $C_a$}
\UI$Y \fCenter X$
\DisplayProof
 \\

 & \\

\AX$(X \AARR Y) \AANDOR Z \fCenter W$
\LeftLabel{\fns Gri$_a$}
\UI$X \AARR (Y \AANDOR Z) \fCenter W$
\DisplayProof
 &
\AX$X \fCenter (Y \AARR Z) \AANDOR W$
\RightLabel{\fns Gri$_a$}
\UI$X \fCenter Y \AARR (Z \AANDOR W)$
\DisplayProof
 \\

\end{tabular}

 &

\begin{tabular}{rl}
\mc{2}{c}{multiplicative} \\
 & \\
\AX$X \fCenter Y$
\doubleLine
\LeftLabel{\fns $\MTOPBOT$}
\UI$\MTOPBOT \MANDOR X \fCenter Y$
\DisplayProof
 &
\AX$X \fCenter Y$
\doubleLine
\RightLabel{\fns $\MTOPBOT$}
\UI$X \fCenter Y \MANDOR \MTOPBOT$
\DisplayProof
 \\

 & \\

\AX$X \MANDOR Y \fCenter Z$
\LeftLabel{\fns $E_m$}
\UI$Y \MANDOR X \fCenter Z $
\DisplayProof
 &
\AX$X \fCenter Y \MANDOR Z$
\RightLabel{\fns $E_m$}
\UI$X \fCenter Z \MANDOR Y$
\DisplayProof \\

 & \\

\AX$X \MANDOR (Y \MANDOR Z) \fCenter W$
\doubleLine
\LeftLabel{\fns $A_m$}
\UI$(X \MANDOR Y) \MANDOR Z \fCenter W $
\DisplayProof
 &
\AX$X \fCenter (Y \MANDOR Z) \MANDOR W$
\RightLabel{\fns $A_m$}
\doubleLine
\UI$X \fCenter Y \MANDOR (Z \MANDOR W)$
\DisplayProof
 \\

 & \\
\phantom{
\AX$X \fCenter Y$
\LeftLabel{\fns $W_m$}
\UI$X \MANDOR \WCIRCW \Gamma \fCenter Y$
\DisplayProof}
 & 
\phantom{
\AX$X \fCenter Y$
\RightLabel{\fns $W_m$}
\UI$X \fCenter Y \MANDOR \WCIRCW \Gamma$
\DisplayProof}
 \\

 & \\

\phantom{
\AX$\WCIRCW \Gamma \MANDOR \WCIRCW \Gamma \fCenter Y$
\LeftLabel{\fns $C_m$}
\UI$\WCIRCW \Gamma \fCenter Y$
\DisplayProof}
 & 
\phantom{
\AX$X \fCenter \WCIRCW Y \MANDOR \WCIRCW Y$
\RightLabel{\fns $C_m$}
\UI$X \fCenter \WCIRCW Y$
\DisplayProof}
 \\

 & \\

\AX$(X \MARR Y) \MANDOR Z \fCenter W$
\LeftLabel{\fns Gri$_m$}
\UI$X \MARR (Y \MANDOR Z) \fCenter W$
\DisplayProof
 &
\AX$X \fCenter (Y \MARR Z) \MANDOR W$
\RightLabel{\fns Gri$_m$}
\UI$X \fCenter Y \MARR (Z \MANDOR W)$
\DisplayProof
 \\

\end{tabular}

 \\
\end{tabular}
\end{center}

\begin{itemize}
\item Pure $\mathsf{Kernel}$-type structural rules
\end{itemize}
\begin{center}
\begin{tabular}{rl}

\mc{1}{c}{$\mathsf{\oc\textrm{-}Kernel}$} & \mc{1}{c}{\!\!\!\!$\mathsf{\wn\textrm{-}Kernel}$} \\
 & \\
\!\!\!\!\!
\begin{tabular}{rl}
\AX$\Gamma \fCenter \Delta$
\doubleLine
\LeftLabel{\fns $\KTOPBOT$}
\UI$\KTOPBOTK \KANDORK \Gamma \fCenter \Delta$
\DisplayProof
 &
\AX$\Gamma \fCenter \Delta$
\doubleLine
\RightLabel{\fns $\KTOPBOT$}
\UI$\Gamma \fCenter \Delta \KANDORK \KTOPBOTK$
\DisplayProof
 \\

 & \\

\AX$\Gamma \KANDORK \Delta \fCenter \Lambda$
\LeftLabel{\fns $E_\oc$}
\UI$\Delta \KANDORK \Gamma \fCenter \Lambda$
\DisplayProof
 &
\AX$\Gamma \fCenter \Delta \KANDORK \Sigma$
\RightLabel{\fns $E_\oc$}
\UI$\Gamma \fCenter \Sigma \KANDORK \Delta$
\DisplayProof
 \\

 & \\

\AX$\Gamma \KANDORK (\Delta \KANDORK \Theta) \fCenter \Lambda$
\doubleLine
\LeftLabel{\fns $A_\oc$}
\UI$(\Gamma \KANDORK \Delta) \KANDORK \Theta \fCenter \Lambda$
\DisplayProof
 &
\AX$\Gamma \fCenter (\Delta \KANDORK \Sigma) \KANDORK \Lambda$
\RightLabel{\fns $A_\oc$}
\doubleLine
\UI$\Gamma \fCenter \Delta \KANDORK (\Sigma \KANDORK \Lambda)$
\DisplayProof
 \\

 & \\

\AX$\Gamma \fCenter \Delta$
\LeftLabel{\fns $W_\oc$}
\UI$\Gamma \KANDORK \Lambda \fCenter \Delta$
\DisplayProof
 &
\AX$\Gamma \fCenter \Delta$
\RightLabel{\fns $W_\oc$}
\UI$\Gamma \fCenter \Delta \KANDORK \Sigma$
\DisplayProof
 \\

 & \\

\AX$\Gamma \KANDORK \Gamma \fCenter \Delta$
\LeftLabel{\fns $C_\oc$}
\UI$\Gamma \fCenter \Delta$
\DisplayProof
 &
\AX$\Gamma \fCenter \Delta \KANDORK \Delta$
\RightLabel{\fns $C_\oc$}
\UI$\Gamma \fCenter \Delta$
\DisplayProof
 \\
\end{tabular}

 &

\!\!\!\!\!\!\!\!\!\!\!

\begin{tabular}{rl}
\AX$\Pi \fCenter \Sigma$
\doubleLine
\LeftLabel{\fns $\TOPBOTK$}
\UI$\KTOPBOTK \KANDORK \Pi \fCenter \Sigma$
\DisplayProof
 &
\AX$\Pi \fCenter \Sigma$
\doubleLine
\RightLabel{\fns $\TOPBOTK$}
\UI$\Pi \fCenter \Sigma \KANDORK \KTOPBOTK$
\DisplayProof
 \\

 & \\

\AX$\Pi \KANDORK \Sigma \fCenter \Psi$
\LeftLabel{\fns $E_\wn$}
\UI$\Sigma \KANDORK \Pi \fCenter \Psi$
\DisplayProof
 &
\AX$\Pi \fCenter \Sigma \KANDORK \Psi$
\RightLabel{\fns $E_\wn$}
\UI$\Pi \fCenter \Psi \KANDORK \Sigma$
\DisplayProof
 \\

 & \\

\AX$\Pi \KANDORK (\Sigma \KANDORK \Psi) \fCenter \Omega$
\doubleLine
\LeftLabel{\fns $A_\wn$}
\UI$(\Pi \KANDORK \Sigma) \KANDORK \Psi \fCenter \Omega$
\DisplayProof
 &
\AX$\Pi \fCenter (\Sigma \KANDORK \Psi) \KANDORK \Omega$
\RightLabel{\fns $A_\wn$}
\doubleLine
\UI$\Pi \fCenter \Sigma \KANDORK (\Psi \KANDORK \Omega)$
\DisplayProof
 \\

 & \\

\AX$\Pi \fCenter \Sigma$
\LeftLabel{\fns $W_\wn$}
\UI$\Pi \KANDORK \Psi \fCenter \Sigma$
\DisplayProof
 &
\AX$\Pi \fCenter \Sigma$
\RightLabel{\fns $W_\wn$}
\UI$\Pi \fCenter \Sigma \KANDORK \Psi$
\DisplayProof
 \\

 & \\

\AX$\Pi \KANDORK \Pi \fCenter \Sigma$
\LeftLabel{\fns $C_\wn$}
\UI$\Pi \fCenter \Sigma$
\DisplayProof
 &
\AX$\Pi \fCenter \Sigma \KANDORK \Sigma$
\RightLabel{\fns $C_\wn$}
\UI$\Pi \fCenter \Sigma$
\DisplayProof
 \\
\end{tabular}

 \\
\end{tabular}
\end{center}

\begin{itemize}
\item Multi-type structural rules
\end{itemize}
\begin{center}
\begin{tabular}{rl}
\AX$\WCIRCW \Gamma \MANDOR \WCIRCW \Delta \fCenter X$
\doubleLine
\LeftLabel{\fns reg / coreg$_{\oc \mathsf{L}}$}
\UI$\WCIRCW (\Gamma \KANDORK \Delta) \fCenter X$
\DisplayProof
 &
\AX$X \fCenter \WCIRCW \Pi \MANDOR \WCIRCW \Sigma$
\doubleLine
\RightLabel{\fns reg / coreg$_{\wn \mathsf{L}}$}
\UI$X \fCenter \WCIRCW (\Pi \KANDORK \Sigma)$
\DisplayProof
 \\
\end{tabular}
\end{center}

\begin{center}
\begin{tabular}{rl}
\AX$\MTOPBOT \fCenter X$
\doubleLine
\LeftLabel{\fns nec / conec$_{\oc \mathsf{L}}$}
\UI$\WCIRCW \KTOPBOTK \fCenter X$
\DisplayProof
 &
\AX$X \fCenter \MTOPBOT$
\doubleLine
\RightLabel{\fns nec / conec$_{\wn \mathsf{L}}$}
\UI$X \fCenter \WCIRCW \KTOPBOTK$
\DisplayProof
 \\
\end{tabular}
\end{center}

\begin{itemize}
\item Pure $\mathsf{Linear}$-type operational rules
\end{itemize}
\begin{center}
\begin{tabular}{@{}c@{}c@{}}

\!\!\!\!\!\!\!\!\!\!\!\!\!\!\!\!\!
\begin{tabular}{@{}rl@{}}

\mc{2}{c}{additive} \\

 & \\

\AXC{\phantom{$\bot \fCenter $}}
\LeftLabel{\fns $\abot$}
\UI$\abot \fCenter \ATOPBOT$
\DisplayProof
 &
\AX$X \fCenter \ATOPBOT$
\RightLabel{\fns $\abot$}
\UI$X \fCenter \abot$
\DisplayProof
 \\

 & \\

\AX$\ATOPBOT \fCenter X$
\LeftLabel{\fns $\aatop$}
\UI$\aatop \fCenter X$
\DisplayProof
 &
\AXC{\phantom{$ \fCenter \top$}}
\RightLabel{\fns $\aatop$}
\UI$\ATOPBOT \fCenter \aatop$
\DisplayProof
 \\

 & \\

\AX$A \AANDOR B \fCenter X$
\LeftLabel{\fns $\aand$}
\UI$A \aand B \fCenter X$
\DisplayProof
 &
\AX$X \fCenter A$
\AX$Y \fCenter B$
\RightLabel{\fns $\aand$}
\BI$X \AANDOR Y \fCenter A \aand B$
\DisplayProof
\ \, \\

 & \\
\AX$A \fCenter X$
\AX$B \fCenter Y$
\LeftLabel{\fns $\aor$}
\BI$A \aor B \fCenter X \AANDOR Y$
\DisplayProof
 &
\AX$X \fCenter A \AANDOR B$
\RightLabel{\fns $\aor$}
\UI$X \fCenter A \aor B $
\DisplayProof
 \\

 & \\
\ \,
\phantom{
\AX$A \AARR B\fCenter Z$
\LeftLabel{\fns $(\adrarr)$}
\UI$A \adrarr B \fCenter Z$
\DisplayProof}
 &
\phantom{
\AX$A \fCenter X$
\AX$ Y \fCenter B$
\RightLabel{\fns $(\adrarr)$}
\BI$X \AARR Y \fCenter A \adrarr B $
\DisplayProof}
 \\
\end{tabular}

 &
\!\!\!\!\!
\begin{tabular}{@{}rl@{}}
\mc{2}{c}{multiplicative} \\

 & \\

\AXC{\phantom{$\bot \fCenter $}}
\LeftLabel{\fns $\mbot$}
\UI$\mbot \fCenter \MTOPBOT$
\DisplayProof
 &
\AX$X \fCenter \MTOPBOT$
\RightLabel{\fns $\mbot$}
\UI$X \fCenter \mbot$
\DisplayProof
 \\

 & \\

\AX$\ATOPBOT \fCenter X$
\LeftLabel{\fns $\mtop$}
\UI$\mtop \fCenter X$
\DisplayProof
 &
\AXC{\phantom{$ \fCenter \top$}}
\RightLabel{\fns $\mtop$}
\UI$\MTOPBOT \fCenter \mtop$
\DisplayProof
 \\

 & \\

\AX$A \MANDOR B \fCenter X$
\LeftLabel{\fns $\mand$}
\UI$A \mand B \fCenter X$
\DisplayProof
 &
\AX$X \fCenter A$
\AX$Y \fCenter B$
\RightLabel{\fns $\mand$}
\BI$X \MANDOR Y \fCenter A \mand B$
\DisplayProof
\ \, \\

 & \\
\AX$A \fCenter X$
\AX$B \fCenter Y$
\LeftLabel{\fns $\mor$}
\BI$A \mor B \fCenter X \MANDOR Y$
\DisplayProof
 &
\AX$X \fCenter A \MANDOR B$
\RightLabel{\fns $\mor$}
\UI$X \fCenter A \mor B $
\DisplayProof
 \\

 & \\
\ \,
\AX$X \fCenter A$
\AX$B \fCenter Y$
\LeftLabel{\fns $\mrarr$}
\BI$A \mrarr B \fCenter X \MARR Y$
\DisplayProof
 &
\AX$Z \fCenter A \MARR B$
\RightLabel{\fns $\mrarr$}
\UI$Z \fCenter A \mrarr B $
\DisplayProof
 \\
\end{tabular}
 \\
\end{tabular}
\end{center}

\begin{itemize}
\item Operational rules for multi-type unary operators
\end{itemize}
\begin{center}
\begin{tabular}{rlrl}


\mc{2}{c}{$\mathsf{K_\oc} \to \mathsf{L}$} & \mc{2}{c}{$\mathsf{L} \to \mathsf{K_\oc}$}\\

 & \\
\AX$\WCIRCW \alpha \fCenter X$
\LeftLabel{\fns $\wdia$}
\UI$\wdiaw \alpha \fCenter X$
\DisplayProof
 &
\AX$\Gamma \fCenter \alpha$
\RightLabel{\fns $\wdia$}
\UI$\WCIRCW \Gamma \fCenter \wdiaw \alpha$
\DisplayProof
 &
\AX$\Gamma \fCenter \BCIRCB A$
\LeftLabel{\fns $\bbox$}
\UI$\Gamma \fCenter \bboxb A$
\DisplayProof
 &
\AX$A \fCenter X$
\RightLabel{\fns $\bbox$}
\UI$\bboxb A \fCenter \BCIRCB X$
\DisplayProof
 \\

 & & & \\

\mc{2}{c}{$\mathsf{L} \to \mathsf{K_\wn}$} & \mc{2}{c}{$\mathsf{K_\wn} \to \mathsf{L}$} \\

 & \\
\AX$\BCIRCB A \fCenter \Pi$
\LeftLabel{\fns $\diab$}
\UI$\bdiab A \fCenter \Pi$
\DisplayProof
 &
\AX$X \fCenter A$
\RightLabel{\fns $\diab$}
\UI$\BCIRCB X \fCenter \bdiab A$
\DisplayProof
 &
\AX$X \fCenter \WCIRCW \xi$
\LeftLabel{\fns $\boxw$}
\UI$X \fCenter \wboxw \xi$
\DisplayProof
 &
\AX$\xi \fCenter \Pi$
\RightLabel{\fns $\boxw$}
\UI$\wboxw \xi \fCenter \WCIRCW \Pi$
\DisplayProof
 \\

\end{tabular}
\end{center}


\subsubsection*{Co-intuitionistic and bi-intuitionistic variants}
The calculus for the bi-intuitionistic  (resp.~co-intuitionistic) variant of linear logic with exponentials is defined by adding (resp.~replacing the introduction rules for $\mrarr$ with) the following introduction rules in  the calculus given above:
\begin{center}
\begin{tabular}{rl}
\AX$A \MARR B \fCenter Z$
\LeftLabel{\fns $\mdrarr$}
\UI$A \mdrarr B \fCenter Z$
\DisplayProof
 &
\AX$A \fCenter X$
\AX$ Y \fCenter B$
\RightLabel{\fns $\mdrarr$}
\BI$X \MARR Y \fCenter A \mdrarr B $
\DisplayProof
 \\
\end{tabular}
\end{center}

\subsubsection*{Paired variants}
Paired variants of each calculus given above (i.e.~intuitionistic, co-intuitionistic, bi-intuitionistic) are defined by adding one, the other or both rows of display postulates below (depending on whether one, the other or both binary maps $\KKrarr$ and $\KKBrarr$ are part of the definition of the heterogeneous algebras associated with the given linear logic), and, accordingly, one, the other or both pairs of FS/co-FS rules, corresponding to the defining properties of the maps $\KKBrarr$ and $\KKrarr$, respectively.

\begin{itemize}
\item Display postulates for  multi-type binary operators
\end{itemize}
\begin{center}
\begin{tabular}{rl}
  \AX$\Gamma \KKBAND \Pi \fCenter \Sigma$
\doubleLine
\LeftLabel{\fns res$_{\oc\wn}$}
\UI$\Pi \fCenter \Gamma \KKRARR \Sigma$
\DisplayProof
 &
\AX$\Gamma \KKBAND \Pi \fCenter \Sigma$
\doubleLine
\RightLabel{\fns res$_{\oc\wn}$}
\UI$\Gamma  \fCenter  \Sigma \KKLARR \Pi$
\DisplayProof
 \\
 & \\
 \AX$\Gamma \fCenter \Pi \KKOR \Delta$
\doubleLine
\LeftLabel{\fns res$_{\wn\oc}$}
\UI$\Pi \KKBARR \Gamma\fCenter \Delta$
\DisplayProof
 &
\AX$\Gamma \fCenter \Pi \KKOR \Delta$
\doubleLine
\RightLabel{\fns res$_{\wn\oc}$}
\UI$\Gamma \KKBLARR \Delta \fCenter  \Pi$
\DisplayProof
 \\
\end{tabular}
\end{center}

\begin{itemize}
\item Structural rules corresponding to the pairing axioms
\end{itemize}
\begin{center}
\begin{tabular}{rl}
\AX$\WCIRCW \Pi \MARR \WCIRCW \Gamma \fCenter X$
\doubleLine
\LeftLabel{\fns FS / coFS$_{\wn\oc \mathsf{L}}$}
\UI$\WCIRCW (\Pi \KKBARR \Gamma) \fCenter X$
\DisplayProof
 &
\AX$X \fCenter \WCIRCW \Gamma \MARR \WCIRCW \Pi$
\doubleLine
\RightLabel{\fns FS / coFS$_{\oc\wn \mathsf{L}}$}
\UI$X \fCenter \WCIRCW (\Gamma \KKRARR \Pi)$
\DisplayProof
 \\
\end{tabular}
\end{center}

\subsubsection*{Classical linear variants}
The propositional linear base of each calculus introduced above turns classical by adding (one or the other of) the following rules:

\begin{center}
\begin{tabular}{rl}
\AX$X \MARR (Y \MANDOR Z) \fCenter W$
\LeftLabel{\fns coGri$_m$}
\UI$(X \MARR Y) \MANDOR Z \fCenter W$
\DisplayProof
 &
\AX$X \fCenter Y \MARR (Z \MANDOR W)$
\RightLabel{\fns coGri$_m$}
\UI$X \fCenter (Y \MARR Z) \MANDOR W$
\DisplayProof
 \\
\end{tabular}
\end{center}
adding which, not only the sequents $(A\mrarr \mbot)\mrarr \mbot\dashv\vdash A$ become derivable, but also $(A\mdrarr \mbot)\mdrarr \mbot\dashv\vdash A$.

\subsubsection*{Relevant and affine variants}
Relevant and affine variants of each calculus given above are defined by adding one or the other row of structural rules below:

\begin{center}
\begin{tabular}{rl}
\AX$X \MANDOR X \fCenter Y$
\LeftLabel{\fns C$_m$}
\UI$X \fCenter Y$
\DisplayProof
 &
\AX$X \fCenter Y \MANDOR Y$
\RightLabel{\fns C$_m$}
\UI$X \fCenter Y$
\DisplayProof
 \\

 & \\

\AX$X \fCenter Y$
\LeftLabel{\fns W$_m$}
\UI$X \MANDOR Z \fCenter Y$
\DisplayProof
 &
\AX$X \fCenter Y$
\RightLabel{\fns W$_m$}
\UI$X \fCenter Y \MANDOR Z$
\DisplayProof
 \\
\end{tabular}
\end{center}
Adding both rows would erase the distinction between the multiplicative and the additive behaviour.

\subsection{Linear negations as primitive connectives}
In the present paper, we have taken the intuitionistic linear setting as basic and, as is usual in this setting, linear negation (and dual linear negation) are defined connectives. Namely, linear negation $A^\mbot$ is defined as $A \mrarr \mbot$ and dual linear negation $A^\mtop$ as $A \mdrarr \mtop$ (cf.~\cite{Gore1998}). However, one can alternatively stipulate that negation(s) are primitive.
In this case, the following pure $\mathsf{Linear}$-type structural and operational connectives need to be added to the language of D.LL:
\begin{center}
\begin{tabular}{|c|c|}
\hline
\mc{2}{|c|}{$\ast$} \\
\hline
$(\ )^{\mtop}$ & $(\ )^{\mbot}$ \\
\hline
\end{tabular}
\end{center}

In the present subsection we discuss this alternative.

\subsubsection*{(Bi-)intuitionistic linear negations}

\begin{itemize}
\item Display postulates for linear negations
\end{itemize}
\begin{center}
\begin{tabular}{rl}
\AX$\ast X \fCenter Y$
\LeftLabel{\fns Gal$_m$}
\UI$\ast Y \fCenter X$
\DisplayProof
 &
\AX$X \fCenter \ast Y $
\RightLabel{\fns Gal$_m$}
\UI$Y \fCenter \ast Y$
\DisplayProof
 \\
\end{tabular}
\end{center}
\begin{itemize}
\item Operational rules  for linear negations
\end{itemize}
\begin{center}
\begin{tabular}{rlrl}
\AX$X \fCenter A$
\UI$A\mneg \fCenter \MNEG X$
\DisplayProof
 &
\AX$X \fCenter \MNEG A$
\UI$X \fCenter A\mneg$
\DisplayProof
 &
\AX$A \fCenter X$
\UI$\MNEG X \fCenter A^\mtop$
\DisplayProof
 &
\AX$\MNEG A \fCenter X$
\UI$A^\mtop \fCenter X$
\DisplayProof
 \\
\end{tabular}
\end{center}
In the calculus extended with the rules above, and the following rules
\begin{center}
\begin{tabular}{cc}
\AX$X \MARR \MTOPBOT \fCenter Y $
\doubleLine
\LeftLabel{\fns coimp - left neg}
\UI$\ast X \fCenter  Y$
\DisplayProof
&
\AX$X \fCenter Y \MARR \MTOPBOT$
\doubleLine
\RightLabel{\fns imp - right neg}
\UI$X \fCenter \ast Y$
\DisplayProof
\\
 \end{tabular}
\end{center}
 the sequents $A \mrarr \mbot \dashv\vdash A\mneg$ and $A \mdrarr \mtop \dashv\vdash A^\mtop$ are then derivable as follows

\begin{center}
\begin{tabular}{cc}
\AX$A \fCenter A$
\AX$\mbot \fCenter \MTOPBOT$
\BI$A \mrarr \mbot \fCenter A \MARR \MTOPBOT$
\RightLabel{\fns imp - right neg}
\UI$A \mrarr \mbot \fCenter \ast A$
\UI$A \mrarr \mbot \fCenter A^\mbot$
\DisplayProof
 &
\AX$A \fCenter A$
\UI$A^\mbot \fCenter \ast A$
\RightLabel{\fns right neg - imp}
\UI$A^\mbot \fCenter A \MARR \MTOPBOT$
\UI$A \MANDOR A^\mbot \fCenter \MTOPBOT$
\UI$A \MANDOR A^\mbot \fCenter \mbot$
\UI$A^\mbot \fCenter A \MARR \mbot$
\UI$A^\mbot \fCenter A \mrarr \mbot$
\DisplayProof
 \\
\end{tabular}
\end{center}

\begin{center}
\begin{tabular}{ccc}
\AX$A \fCenter A$
\UI$\ast A \fCenter A^\mtop$
\UI$\ast A \MANDOR \MTOPBOT \fCenter A^\mtop$
\LeftLabel{\fns left neg - coimp}
\UI$A \MARR \MTOPBOT \fCenter A^\mtop$
\UI$\MTOPBOT \fCenter A \MANDOR A^\mtop$
\UI$\mtop \fCenter A \MANDOR A^\mtop$
\UI$A \MARR \mtop \fCenter A^\mtop$
\UI$A \mdrarr \mtop \fCenter A^\mtop$
\DisplayProof
 &
\AX$A \fCenter A$
\AX$\MTOPBOT \fCenter \mtop$
\BI$A \MARR \MTOPBOT \fCenter A \mdrarr \mtop$
\LeftLabel{\fns coimp - left neg}
\UI$\ast A \fCenter A \mdrarr \mtop$
\UI$A^\mtop \fCenter A \mdrarr \mtop$
\DisplayProof
 \\
\end{tabular}
\end{center}

\subsubsection*{Paired variants}

In paired linear logics, either one or both heterogeneous negations $\alpha^{\botk} \in \mathsf{K}_\wn$ and $\xi^{\ktop} \in \mathsf{K}_\oc$ can be defined as $\alpha \KKrarr \botk$ and $\xi\KKBrarr \ktop$ respectively (and, at the structural level, also the `symmetric' negations $^{\ktop}\alpha \in \mathsf{K}_\wn$ and $^{\botk}\xi \in \mathsf{K}_\oc$ defined as $\ktop \KKBlarr \alpha $ and $\botk\KKlarr \xi$ respectively). When these negations are taken as primitive, the following heterogeneous structural and operational connectives need to be added to the language of D.LL:

\begin{center}
\begin{tabular}{|c|c|c|c|c|c|c|c|}
\hline
\mc{2}{|c|}{\fns $\mathsf{K_\oc} \to \mathsf{K_\wn}$} & \mc{2}{c|}{\fns $\mathsf{K_\wn} \to \mathsf{K_\oc}$} & \mc{2}{c|}{\fns $\mathsf{K_\oc} \to \mathsf{K_\wn}$} & \mc{2}{c|}{\fns $\mathsf{K_\wn} \to \mathsf{K_\oc}$} \\
\hline
\mc{2}{|c|}{$ _\wn\circledast$} & \mc{2}{c|}{$\circledast_\oc$} & \mc{2}{c|}{$\circledast_\wn$} & \mc{2}{c|}{$ _\oc\circledast$}\\
\hline
$^{\ktop}(\ )$ & $\phantom{(\ )^{\ktop}}$ & $(\ )^{\ktop}$ & $\phantom{(\ )^{\ktop}}$ & $\phantom{(\ )^{\botk}}$ & $(\ )^{\botk}$ & $\phantom{(\ )^{\botk}}$ & $^{\botk}(\ )$ \\
\hline
\end{tabular}
\end{center}

\begin{itemize}
\item Display postulates for heterogeneous negations
\end{itemize}
\begin{center}
\begin{tabular}{rl}
\AX$\circledast_\oc \Sigma   \fCenter \Gamma$
\LeftLabel{\fns Gal$_\oc\wn$}
\doubleLine
\UI$ _\wn\circledast \Gamma \fCenter \Sigma$
\DisplayProof
 &
\AX$\Sigma \fCenter \circledast_\wn \Gamma $
\RightLabel{\fns Gal$_\wn\oc$}
\doubleLine
\UI$\Gamma \fCenter _\oc\circledast \Sigma$
\DisplayProof
 \\
\end{tabular}
\end{center}
As usual, we will drop the subscripts, since the reading is unambiguous.

\begin{itemize}
\item Operational rules  for heterogeneous negations
\end{itemize}

\begin{center}
\begin{tabular}{rlrl}
\AX$\Gamma \fCenter \alpha$
\UI$\alpha^{\botk} \fCenter \circledast \Gamma$
\DisplayProof
 &
\AX$\Gamma \fCenter \circledast \alpha$
\UI$\Gamma \fCenter \alpha^{\botk}$
\DisplayProof
 &
\AX$\xi \fCenter \Pi$
\UI$\circledast \Pi \fCenter \xi^{\ktop}$
\DisplayProof
 &
\AX$\circledast \xi \fCenter \Pi$
\UI$\xi^{\ktop} \fCenter \Pi$
\DisplayProof
 \\
\end{tabular}
\end{center}

In the calculus extended with the rules above, and the following rules
\begin{center}
\begin{tabular}{cc}
\AX$\Pi \KKBARR \KTOPBOTK \fCenter \Gamma $
\doubleLine
\LeftLabel{\fns dual het neg}
\UI$\circledast \Pi \fCenter  \Gamma$
\DisplayProof
&
\AX$\Pi \fCenter \Gamma \KKRARR \KTOPBOTK$
\doubleLine
\RightLabel{\fns het neg}
\UI$\Pi \fCenter \circledast \Gamma$
\DisplayProof
\\
 \end{tabular}
\end{center}
  the sequents $\alpha\KKrarr\botk \dashv\vdash \alpha^{\botk}$ and $\xi\KKBrarr\ktop \dashv\vdash \xi^{\ktop}$ become derivable (we omit the corresponding derivations). In this language, it becomes possible to formulate the following alternative versions of the FS/coFS rules (which are equivalent to them if the linear implications are defined connectives):
\begin{center}
\begin{tabular}{rl}
\AX$\ast \WCIRCW \Pi \MANDOR \WCIRCW \Gamma \fCenter X$
\doubleLine
\LeftLabel{\fns ${\textrm{P}}$ / co${\textrm{P}}_{\wn\oc \mathsf{L}}$}
\UI$\WCIRCW (\circledast \Pi \KANDORK \Gamma) \fCenter X$
\DisplayProof
 &
\AX$X \fCenter \ast \WCIRCW \Gamma \MANDOR \WCIRCW \Pi$
\doubleLine
\RightLabel{\fns ${\textrm{P}}$ / co${\textrm{P}}_{\oc\wn \mathsf{L}}$}
\UI$X \fCenter \WCIRCW (\circledast \Gamma \KANDORK \Pi)$
\DisplayProof
 \\
\end{tabular}
\end{center}

\subsubsection*{Interdefinable exponentials}
A connection which is at least as strong as the one captured by the P/coP rules is encoded in the following rules:


\begin{center}
\begin{tabular}{rl}
\AX$\ast \WCIRCW \Gamma \fCenter X$
\LeftLabel{\fns swap-in / -out}
\doubleLine
\UI$\WCIRCW \circledast \Gamma \fCenter X$
\DisplayProof
 &
\AX$X \fCenter \ast \WCIRCW \Pi$
\RightLabel{\fns swap-in / -out}
\doubleLine
\UI$X \fCenter \WCIRCW \circledast \Pi$
\DisplayProof
 \\

\end{tabular}
\end{center}

Indeed, in the presence of the rules above, P and coP are derivable as follows:

\begin{center}
\begin{tabular}{cc}
\AX$\ast \WCIRCW \Pi \MANDOR \WCIRCW \Gamma \fCenter X$
\UI$\WCIRCW \Gamma \MANDOR \ast \WCIRCW \Pi \fCenter X$
\UI$\ast \WCIRCW \Pi \fCenter \WCIRCW \Gamma \MARR X$
\LeftLabel{\fns swap-in}
\UI$\WCIRCW \circledast \Pi \fCenter \WCIRCW \Gamma \MARR X$
\UI$\WCIRCW \Gamma \MANDOR \WCIRCW \circledast \Pi \fCenter X$
\UI$\WCIRCW \circledast \Pi \MANDOR \WCIRCW \Gamma \fCenter X$
\LeftLabel{\fns reg$_{\oc \mathsf{L}}$}
\UI$\WCIRCW (\circledast \Pi \KANDORK \Gamma) \fCenter X$
\DisplayProof
 &
\AX$\WCIRCW (\circledast \Pi \KANDORK \Gamma) \fCenter X$
\LeftLabel{\fns coreg$_{\oc \mathsf{L}}$}
\UI$\WCIRCW \circledast \Pi \MANDOR \WCIRCW \Gamma \fCenter X$
\UI$\WCIRCW \Gamma \MANDOR \WCIRCW \circledast \Pi \fCenter X$
\UI$\WCIRCW \circledast \Pi \fCenter \WCIRCW \Gamma \MARR X$
\LeftLabel{\fns swap-out}
\UI$\ast \WCIRCW \Pi \fCenter \WCIRCW \Gamma \MARR X$
\UI$\WCIRCW \Gamma \MANDOR \ast \WCIRCW \Pi \fCenter X$
\UI$\ast \WCIRCW \Pi \MANDOR \WCIRCW \Gamma \fCenter X$
\DisplayProof
 \\
\end{tabular}
\end{center}
Using swap-out and swap-in one can prove that the following sequents are derivable:

\begin{center}
\begin{tabular}{cc}
\AX$A \fCenter A$
\UI$\ast A \fCenter A^{\mtop}$
\UI$\BCIRCB \ast A \fCenter \bdiab A^{\mtop}$
\UI$\ast A \fCenter \WCIRCW \bdiab A^{\mtop}$
\UI$\ast \WCIRCW \bdiab A^{\mtop} \fCenter A$
\LeftLabel{\fns swap-in}
\UI$\WCIRCW \circledast \bdiab A^{\mtop} \fCenter A$
\UI$\circledast \bdiab A^{\mtop} \fCenter \BCIRCB A$
\UI$\circledast \bdiab A^{\mtop} \fCenter \bboxb A$
\UI$\WCIRCW \circledast \bdiab A^{\mtop} \fCenter \wdiaw \bboxb A$
\LeftLabel{\fns def}
\UI$\WCIRCW \circledast \bdiab A^{\mtop} \fCenter \oc A$
\LeftLabel{\fns swap-out}
\UI$\ast \WCIRCW \bdiab A^{\mtop} \fCenter \oc A$
\UI$\ast \oc A \fCenter \WCIRCW \bdiab A^{\mtop}$
\UI$\ast \oc A \fCenter \wboxw \bdiab A^{\mtop}$
\RightLabel{\fns def}
\UI$\ast \oc A \fCenter \wn A^{\mtop}$
\UI$\ast \wn A^{\mtop} \fCenter \oc A$
\UI$(\wn A^{\mtop})^{\mtop} \fCenter \oc A$
\DisplayProof
 &
\AX$A \fCenter A$
\UI$A^{\mbot} \fCenter \ast A$
\UI$\bboxb A^{\mbot} \fCenter \BCIRCB \ast A$
\UI$\WCIRCW \bboxb A^{\mbot} \fCenter \ast A$
\UI$A \fCenter \ast \WCIRCW \bboxb A^{\mbot}$
\RightLabel{\fns swap-in}
\UI$A \fCenter \WCIRCW \circledast \bboxb A^{\mbot}$
\UI$\BCIRCB A \fCenter \circledast \bboxb A^{\mbot}$
\UI$\bdiab A \fCenter \circledast \bboxb A^{\mbot}$
\UI$\wboxw \bdiab A \fCenter \WCIRCW \circledast \bboxb A^{\mbot}$
\LeftLabel{\fns def}
\UI$\wn A \fCenter \WCIRCW \circledast \bboxb A^{\mbot}$
\RightLabel{\fns swap-out}
\UI$\wn A \fCenter \ast \WCIRCW \bboxb A^{\mbot}$
\UI$\WCIRCW \bboxb A^{\mbot} \fCenter \ast \wn A$
\UI$\wdiaw \bboxb A^{\mbot} \fCenter \ast \wn A$
\LeftLabel{\fns def}
\UI$\oc A^{\mbot} \fCenter \ast \wn A$
\UI$\wn A \fCenter \ast \oc A^{\mbot}$
\UI$\wn A \fCenter (\oc A^{\mbot})^{\mbot}$
\DisplayProof
\end{tabular}
\end{center}

\subsubsection*{Classical linear negations}
When negations are primitive, the classical linear propositional base can be captured by adding the following structural rules:
\begin{center}
\begin{tabular}{rl}
\mc{2}{c}{
\!\!\!\!\!\!\!\!\!\!\!\!\!\!\!\!\!\!\!\!\!\!\!\!\!\!\!\!\!\!
\AX$X \fCenter Y$
\LeftLabel{\fns pseudo contr}
\doubleLine
\UI$\ast Y \fCenter \ast X$
\DisplayProof}
 \\
 & \\
\AX$X \fCenter Y \MANDOR Z$
\LeftLabel{\fns left neg}
\doubleLine
\UI$\ast Y \MANDOR X \fCenter Z$
\DisplayProof
 &
\AX$X \MANDOR Y \fCenter Z$
\RightLabel{\fns right neg}
\doubleLine
\UI$Y \fCenter \ast X \MANDOR Z$
\DisplayProof
 \\
\end{tabular}
\end{center}
 The  rules above are the counterparts of the classical Grishin rules Gri$_m$ and coGri$_m$ in the language in which negation is primitive and the two linear implications are defined. In the calculus extended accordingly, the sequents $A \mrarr B \dashv\vdash A^{\mbot} \mor B$ and $A \mdrarr B \dashv\vdash A^{\mtop} \mand B$ are then derivable as follows:

\begin{center}
\begin{tabular}{cc}
\AX$A \fCenter A$
\AX$B \fCenter B$
\BI$A \mrarr B \fCenter A \MARR B$
\UI$A \MANDOR A \mrarr B \fCenter B$
\UI$A \mrarr B \fCenter \ast A \MANDOR B$
\UI$A \mrarr B \fCenter B \MANDOR \ast A$
\UI$B \MARR A \mrarr B \fCenter \ast A$
\UI$B \MARR A \mrarr B \fCenter A^{\mbot}$
\UI$A \mrarr B \fCenter B \MANDOR A^{\mbot}$
\UI$A \mrarr B \fCenter A^{\mbot} \MANDOR B$
\UI$A \mrarr B \fCenter A^{\mbot} \mor B$
\DisplayProof
 &
\AX$A \fCenter A$
\UI$A^\mbot \fCenter \ast A$
\AX$B \fCenter B$
\BI$A^\mbot \mor B \fCenter \ast A \MANDOR B$
\UI$A \MANDOR A^\mbot \mor B \fCenter B$
\UI$A^\mbot \mor B \fCenter A \MARR B$
\UI$A^\mbot \mor B \fCenter A \mrarr B$
\DisplayProof
 \\
\end{tabular}
\end{center}

\begin{center}
\begin{tabular}{ccc}
\AX$A \fCenter A$
\UI$\ast A \fCenter A^{\mtop}$
\AX$B \fCenter B$
\BI$\ast A \MANDOR B \fCenter A^{\mtop} \mand B$
\UI$B \fCenter A \MANDOR A^{\mtop} \mand B$
\UI$A \MARR B \fCenter A^{\mtop} \mand B$
\UI$A \mdrarr B \fCenter A^{\mtop} \mand B$
\DisplayProof
 &
\AX$A \fCenter A$
\AX$B \fCenter B$
\BI$A \MARR B \fCenter A \mdrarr B$
\UI$B \fCenter A \MANDOR A \mdrarr B$
\UI$\ast A \MANDOR B \fCenter A \mdrarr B$
\UI$B \MANDOR \ast A \fCenter A \mdrarr B$
\UI$\ast A \fCenter B \MARR A \mdrarr B$
\UI$A^{\mtop} \fCenter B \MARR A \mdrarr B$
\UI$B \MANDOR A^{\mtop} \fCenter A \mdrarr B$
\UI$A^{\mtop} \MANDOR B \fCenter A \mdrarr B$
\UI$A^{\mtop} \mand B \fCenter A \mdrarr B$
\DisplayProof
 \\
\end{tabular}
\end{center}

Moreover, left and right negation are interderivable:

\begin{center}
\begin{tabular}{cc}
\AX$A \fCenter A$
\UI$\ast A \fCenter \ast A$
\UI$A^\mtop \fCenter \ast A$
\UI$A^\mtop \fCenter A^\mbot$
\DisplayProof
 &
\AX$A \fCenter A$
\UI$\ast A \fCenter A^\mtop$
\UI$\ast A^\mtop \fCenter A$
\UI$A^\mbot \fCenter \ast \ast A^\mtop$
\UI$\ast A^\mtop \fCenter \ast A^\mbot$
\UI$A^\mbot \fCenter A^\mtop$
\DisplayProof
 \\
\end{tabular}
\end{center}

Augmenting the paired setting with the following rules:
 \begin{center}
\begin{tabular}{rl}
\AX$ \Gamma \fCenter \Delta$
\LeftLabel{\fns pseudo contr$_\oc$}
\UI$\circledast \Delta \fCenter \circledast \Gamma$
\DisplayProof
 &
\AX$ \Pi \fCenter \Sigma$
\RightLabel{\fns pseudo contr$_\wn$}
\UI$\circledast \Sigma \fCenter \circledast \Pi$
\DisplayProof
 \\
 & \\
\AX$\circledast \BCIRCB X \fCenter \Sigma$
\LeftLabel{\fns swap-in / -out}
\doubleLine
\UI$\BCIRCB \ast X \fCenter \Sigma$
\DisplayProof
 &
\AX$\Gamma \fCenter \circledast \BCIRCB X$
\RightLabel{\fns swap-in / -out}
\doubleLine
\UI$\Gamma \fCenter \BCIRCB \ast X$
\DisplayProof
 \\
\end{tabular}
\end{center}
the following sequents become derivable as well.
\begin{center}
\begin{tabular}{cc}
\AX$A \fCenter A$
\UI$\bboxb A \fCenter \BCIRCB A$
\UI$\circledast \BCIRCB A \fCenter \circledast \bboxb A$
\LeftLabel{\fns swap-in}
\UI$\BCIRCB \ast A \fCenter \ast \bboxb A$
\UI$\ast A \fCenter \WCIRCW \ast \bboxb A$
\UI$A^{\mtop} \fCenter \WCIRCW \ast \bboxb A$
\UI$\BCIRCB A^{\mtop} \fCenter \ast \bboxb A$
\UI$\bdiab A^{\mtop} \fCenter \ast \bboxb A$
\UI$\wboxw \bdiab A^{\mtop} \fCenter \WCIRCW \ast \bboxb A$
\LeftLabel{\fns def}
\UI$\wn A^{\mtop} \fCenter \WCIRCW \ast \bboxb A$
\UI$\ast \WCIRCW \ast \bboxb A \fCenter (\wn A^{\mtop})^{\mtop}$
\UI$\ast (\wn A^{\mtop})^{\mtop} \fCenter \WCIRCW \ast \bboxb A$
\RightLabel{\fns swap-out}
\UI$\ast (\wn A^{\mtop})^{\mtop} \fCenter \ast \WCIRCW \bboxb A$
\UI$\WCIRCW \bboxb A \fCenter (\wn A^{\mtop})^{\mtop}$
\UI$\wdiaw \bboxb A \fCenter (\wn A^{\mtop})^{\mtop}$
\RightLabel{\fns def}
\UI$\oc A \fCenter (\wn A^{\mtop})^{\mtop}$
\DisplayProof
 &
\AX$A \fCenter A$
\UI$\BCIRCB A \fCenter \bdiab A$
\UI$\circledast \bdiab A \fCenter \circledast \BCIRCB A$
\RightLabel{\fns swap-in}
\UI$\ast \bdiab A \fCenter \BCIRCB \ast A$
\UI$\WCIRCW \ast \bdiab A \fCenter \ast A$
\UI$\WCIRCW \ast' \bdiab A \fCenter A^{\mbot}$
\UI$\ast \bdiab A \fCenter \BCIRCB A^{\mbot}$
\UI$\ast \bdiab A \fCenter \bboxb A^{\mbot}$
\UI$\WCIRCW \ast \bdiab A \fCenter \wdiaw \bboxb A^{\mbot}$
\RightLabel{\fns def}
\UI$\WCIRCW \ast \bdiab A \fCenter \oc A^{\mbot}$
\UI$(\oc A^{\mbot})^{\mbot} \fCenter \ast \WCIRCW \ast \bdiab A$
\UI$\WCIRCW \ast \bdiab A \fCenter \ast (\oc A^{\mbot})^{\mbot}$
\LeftLabel{\fns swap-out}
\UI$\ast \WCIRCW \bdiab A \fCenter \ast (\oc A^{\mbot})^{\mbot}$
\UI$(\oc A^{\mbot})^{\mbot} \fCenter \WCIRCW \bdiab A$
\UI$(\oc A^{\mbot})^{\mbot} \fCenter \wboxw \bdiab A$
\RightLabel{\fns def}
\UI$(\oc A^{\mbot})^{\mbot} \fCenter \wn A$
\DisplayProof
\end{tabular}
\end{center}

\section{Properties}\label{sec:properties}
\subsection{Soundness}
\label{ssec:soundness}

In the present subsection, we outline the  verification of the soundness of the rules of D.LL w.r.t.~the semantics of {\em perfect} heterogeneous algebras  (cf.~Definition \ref{def:heterogeneous algebras}). The first step consists in interpreting structural symbols as logical symbols according to their (precedent or succedent) position,\footnote{\label{footnote:def precedent succedent pos}For any  sequent $x\vdash y$, we define the signed generation trees $+x$ and $-y$ by labelling the root of the generation tree of $x$ (resp.\ $y$) with the sign $+$ (resp.\ $-$), and then propagating the sign to all nodes according to the polarity of the coordinate of the connective assigned to each node. Positive (resp.\ negative) coordinates propagate the same (resp.\ opposite) sign to the corresponding child node.  Then, a substructure $z$ in $x\vdash y$ is in {\em precedent} (resp.\ {\em succedent}) {\em position} if the sign of its root node as a subtree of $+x$ or $-y$ is  $+$ (resp.\ $-$).}
as indicated in the synoptic tables of Section \ref{ssec:language of DLL}. This makes it possible to interpret sequents as inequalities, and rules as quasi-inequalities. For example, the rules on the left-hand side below are interpreted as the quasi-inequalities on the right-hand side:

\begin{center}
\begin{tabular}{rcl}
\AX$X \fCenter \WCIRCW \Pi \MANDOR \WCIRCW \Sigma$
\UI$X \fCenter \WCIRCW (\Pi \KANDORK \Sigma)$
\DisplayProof
&$\quad\rightsquigarrow\quad$&
$\forall a\forall \xi\forall \chi[a\leq  \wboxw \xi \mor\wboxw \chi \Rightarrow a\leq \wboxw(\xi \kork \chi)]$
\\
&&\\
\AX$\WCIRCW \Pi \MARR \WCIRCW \Gamma \fCenter X$
\UI$\WCIRCW (\Pi \KKBARR \Gamma) \fCenter X$
\DisplayProof
&$\quad\rightsquigarrow\quad$& $\forall \xi\forall \alpha\forall c[\wboxw\xi \mdrarr  \wdiaw\alpha\leq c\Rightarrow  \wdiaw(\xi \KKBrarr\alpha)  \leq c]$
\\
\end{tabular}
\end{center}
The verification of the soundness of the rules of D.LL  then consists in verifying the validity of their corresponding quasi-inequalities in perfect heterogeneous algebras. The verification of the soundness of pure-type rules and of the introduction rules following this procedure is routine, and is omitted.
The validity of the quasi-inequalities corresponding to multi-type structural rules follows straightforwardly from two observations. First, the quasi-inequality corresponding to each rule is obtained by running the algorithm ALBA (cf.~\cite{GMPTZ}) on one of the inequalities in the right column of the statement of Proposition \ref{prop: from non analytic to analytic}. Below we perform an ALBA reduction on  the last inequality in that list:
\begin{center}
\begin{tabular}{c l}
& $\forall a\forall b[\bdiab(a\mdrarr b)\KKBrarr \bboxb b\leq  \bboxb a]$\\
iff & $\forall \xi\forall \alpha\forall c\forall a\forall b[(\bdiab(a\mdrarr b)\leq \xi\ \&\ \alpha\leq \bboxb b\ \&\ a\leq c)\Rightarrow \xi\KKBrarr \alpha\leq  \bboxb c]$\\
iff & $\forall \xi\forall \alpha\forall c\forall b[(\bdiab(c\mdrarr b)\leq \xi\ \&\ \alpha\leq \bboxb b)\Rightarrow \xi\KKBrarr \alpha\leq  \bboxb c]$\\
iff & $\forall \xi\forall \alpha\forall c\forall b[(\bdiab(c\mdrarr b)\leq \xi\ \&\ \wdiaw\alpha\leq  b)\Rightarrow \xi\KKBrarr \alpha\leq  \bboxb c]$\\
iff & $\forall \xi\forall \alpha\forall c[\bdiab(c\mdrarr \wdiaw\alpha)\leq \xi\Rightarrow \xi\KKBrarr \alpha\leq  \bboxb c]$\\
iff & $\forall \xi\forall \alpha\forall c[c\mdrarr \wdiaw\alpha\leq \wboxw\xi\Rightarrow \wdiaw(\xi\KKBrarr \alpha)\leq   c]$\\
iff & $\forall \xi\forall \alpha\forall c[\wdiaw\alpha\leq c\kork \wboxw\xi\Rightarrow \wdiaw(\xi\KKBrarr \alpha)\leq   c]$\\
iff & $\forall \xi\forall \alpha\forall c[\wboxw\xi\mdrarr \wdiaw\alpha\leq c\Rightarrow \wdiaw(\xi\KKBrarr \alpha)\leq   c].$\\

\end{tabular}
\end{center}
It can be readily checked that the ALBA rewriting rules  applied in the computation above (adjunction rules and Ackermann rules) are sound on perfect heterogeneous algebras. As discussed in \cite{GMPTZ}, the soundness of these rewriting rules only depends on the order-theoretic properties of the interpretation of the logical connectives and their adjoints and residuals. The fact that some of these maps are not internal operations but have different domains and codomains does not make any substantial difference.  The second observation is that the axioms in the right column of the statement of Proposition \ref{prop: from non analytic to analytic} are valid by construction on heterogeneous algebras of suitable similarity type (this readily follows from the results in Sections \ref{ssec: linear algebras}, \ref{ssec:from ono axioms to eterogeneous arrows} and \ref{ssec: reverse engineering}), and hence in particular on {\em perfect} heterogeneous algebras. We finish this subsection by showing a one-step run of ALBA on another inequality in the statement of Proposition \ref{prop: from non analytic to analytic}.
\begin{center}
\begin{tabular}{c l}
& $\forall \xi\forall \chi[\wboxw \xi \mor\wboxw \chi \leq \wboxw(\xi \kork \chi)]$\\
iff & $\forall a\forall \xi\forall \chi[a\leq  \wboxw \xi \mor\wboxw \chi \Rightarrow a\leq \wboxw(\xi \kork \chi)].$
\end{tabular}
\end{center}

\subsection{Completeness}
In the present subsection, we show that the translations of the axioms and rules of Girard's calculus for linear logic (cf.\ \cite{Girard1987}) are derivable in D.LL. Since Girard's calculus is complete w.r.t.~the appropriate class of perfect linear algebras, and hence w.r.t~their associated perfect heterogeneous algebras, this is enough to show the completeness of  the version of D.LL corresponding to Girard's calculus.\footnote{In Section \ref{sec:some derivations}, we derive the axioms of the Hilbert-style presentations of intuitionistic linear logic. In a similar fashion, it is possible to transfer the completeness of the Hilbert-style presentation of each variant of linear logic w.r.t.~its associated class of perfect heterogeneous algebras to the associated proper display calculus.}

The derivations of axioms and rules not involving exponentials are standard and we omit them.
In the remainder of the present subsection, we  focus on the rules involving exponentials, namely,

\begin{itemize}
\item left (resp.~right) dereliction and right (resp.~left) promotion rules:
\end{itemize}
\begin{center}
\begin{tabular}{cccc}
\AX$X \MANDOR A \fCenter Y$
\UI$X \MANDOR \oc A \fCenter Y$
\DisplayProof
 &
\AX$X \fCenter A \MANDOR Y$
\UI$X \fCenter \wn A \MANDOR Y$
\DisplayProof
 &
\AX$\oc X \fCenter A \MANDOR \wn Y$
\UI$\oc X \fCenter \oc A \MANDOR \wn Y$
\DisplayProof
 &
\AX$\oc X \MANDOR A \fCenter \wn Y$
\UI$\oc X \MANDOR \wn A \fCenter \wn Y$
\DisplayProof
 \\
\end{tabular}
\end{center}

\begin{itemize}
\item left (resp.~right) weakening and left (resp.~right) contraction rules:
\end{itemize}
\begin{center}
\begin{tabular}{cccc}
\AX$X \fCenter Y$
\UI$X \MANDOR \oc A \fCenter Y$
\DisplayProof
 &
\AX$X \fCenter Y$
\UI$X \fCenter \wn A \MANDOR Y$
\DisplayProof
 &
\AX$X \MANDOR \oc A \MANDOR \oc A \fCenter Y$
\UI$X \MANDOR \oc A \fCenter Y$
\DisplayProof
 &
\AX$X \fCenter \wn A \MANDOR \wn A \MANDOR Y$
\UI$X \fCenter \wn A \MANDOR Y$
\DisplayProof
 \\
\end{tabular}
\end{center}
where, abusing notation, $\oc X$ and $\wn Y$ denote structures which are built from formulas of the form $\oc A$ and $\wn B$ respectively, using only the structural counterpart of $\mand$ and $\mor$.

Translating these rules in the language of D.LL we obtain:

\begin{center}
\begin{tabular}{cccc}
\AX$X \MANDOR A \fCenter Y$
\UI$X \MANDOR \wdiaw \bboxb A \fCenter Y$
\DisplayProof
 &
\AX$X \fCenter A \MANDOR Y$
\UI$X \fCenter \wboxw \bdiab A \MANDOR Y$
\DisplayProof
&
\AX$\wdiaw \bboxb X \fCenter A \MANDOR \wboxw \bdiab Y$
\UI$\wdiaw \bboxb X \fCenter \wdiaw \bboxb A \MANDOR \wboxw \bdiab Y$
\DisplayProof
 &
\AX$\wdiaw \bboxb X \MANDOR A \fCenter \wboxw \bdiab Y$
\UI$\wdiaw \bboxb X \MANDOR \wboxw \bdiab A \fCenter \wboxw \bdiab Y$
\DisplayProof
 \\

 & & & \\

\AX$X \fCenter Y$
\UI$X \MANDOR \wdiaw \bboxb A \fCenter Y$
\DisplayProof
 &
\AX$X \fCenter Y$
\UI$X \fCenter \wboxw \bdiab A \MANDOR Y$
\DisplayProof
 &
\AX$X \MANDOR \wdiaw \bboxb A \MANDOR \wdiaw \bboxb A \fCenter Y$
\UI$X \MANDOR \wdiaw \bboxb A \fCenter Y$
\DisplayProof
 &
\AX$X \fCenter \wboxw \bdiab A \MANDOR \wboxw \bdiab A \MANDOR Y$
\UI$X \fCenter \wboxw \bdiab A \MANDOR Y$
\DisplayProof
 \\
\end{tabular}
\end{center}
where, abusing notation, $\wdiaw \bboxb X$ and $\wboxw \bdiab Y$ denote structures which are built from formulas of the form $\wdiaw \bboxb A^t$ and $\wboxw \bdiab B^t$ respectively, using only the structural counterpart of $\mand$ and $\mor$.

\begin{itemize}
\item Derivations of left- and right-dereliction:
\end{itemize}
\begin{center}
\begin{tabular}{cc}
\AX$X \MANDOR A \fCenter Y$
\UI$A \fCenter X \MARR Y$
\UI$\bboxb A \fCenter \BCIRCB (X \MARR Y)$
\UI$\WCIRCW \bboxb A \fCenter X \MARR Y$
\UI$\wdiaw \bboxb A \fCenter X \MARR Y$
\UI$X \MANDOR \wdiaw \bboxb A \fCenter Y$
\DisplayProof
 &
\AX$X \fCenter A \MANDOR Y$
\UI$X \fCenter Y \MANDOR A$
\UI$Y \MARR X \fCenter A$
\UI$\BCIRCB (Y \MARR X) \fCenter \bdiab A$
\UI$Y \MARR X \fCenter \WCIRCW \bdiab A$
\UI$Y \MARR X \fCenter \wboxw \bdiab A$
\UI$X \fCenter Y \MANDOR \wboxw \bdiab A$
\UI$X \fCenter \wboxw \bdiab A \MANDOR Y$
\DisplayProof
 \\
\end{tabular}
\end{center}

\begin{itemize}
\item Derivations of left- and right-weakening:
\end{itemize}
\begin{center}
\begin{tabular}{cc}
\AX$X \fCenter Y$
\LeftLabel{\fns $W_m$}
\UI$X \MANDOR \WCIRCW \bboxb A \fCenter Y$
\UI$\WCIRCW \bboxb A \fCenter X \MARR Y$
\UI$\wdiaw \bboxb A \fCenter X \MARR Y$
\UI$X \MANDOR \wdiaw \bboxb A \fCenter Y$
\DisplayProof
 &
\AX$X \fCenter Y$
\RightLabel{\fns $W_m$}
\UI$X \fCenter Y \MANDOR \WCIRCW \bdiab A$
\UI$Y \MARR X \fCenter \WCIRCW \bdiab A$
\UI$Y \MARR X \fCenter \wboxw \bdiab A$
\UI$X \fCenter Y \MANDOR \wboxw \bdiab  A$
\UI$X \fCenter \wboxw \bdiab  A \MANDOR Y$
\DisplayProof
\end{tabular}
\end{center}

For the purpose of showing that the promotion and contraction rules are derivable, it is enough to show that the following rules are derivable in D.LL.

\begin{center}
\begin{tabular}{cccc}
\AX$X \MANDOR (\WCIRCW \bboxb A \MANDOR \WCIRCW \bboxb A) \fCenter Y$
\UI$X \MANDOR  \WCIRCW \bboxb A \fCenter Y$
\DisplayProof
 &
\AX$X \fCenter (\WCIRCW \bdiab A \MANDOR \WCIRCW \bdiab A) \MANDOR Y$
\UI$X \fCenter  \WCIRCW \bdiab A \MANDOR Y$
\DisplayProof
 &
\AX$\WCIRCW \Gamma \fCenter A \MANDOR \WCIRCW \Pi$
\UI$\WCIRCW \Gamma \fCenter \wdiaw \bboxb A \MANDOR \WCIRCW \Pi$
\DisplayProof
 &
\AX$\WCIRCW \Gamma \MANDOR A \fCenter \WCIRCW \Pi$
\UI$\WCIRCW \Gamma \MANDOR \wboxw \bdiab A \fCenter \WCIRCW \Pi$
\DisplayProof
 \\
\end{tabular}
\end{center}
Indeed, as discussed in Section \ref{sec: inversion lemmas}, left-introduction rules for $\mathcal{F}$-connectives (such as diamonds) and right-introduction rules for $\mathcal{G}$-connectives (such as boxes) are invertible. Hence, if e.g.~$\wdiaw \bboxb X \fCenter A \MANDOR \wboxw \bdiab Y$ is derivable, using the inversion lemmas, associativity, exchange, regrouping, and display, one can show that $\WCIRCW \Gamma \fCenter A \MANDOR \WCIRCW \Pi$ is derivable  for some $\Gamma$ and $\Pi$. Also, if $\WCIRCW \Gamma \fCenter \wdiaw \bboxb A \MANDOR \WCIRCW \Pi$ is derivable, then using associativity, exchange, coregrouping, display, and box- and diamond-introduction rules, one can show that $\wdiaw \bboxb X \fCenter \wdiaw \bboxb A \MANDOR \wboxw \bdiab Y$ is derivable. In what follows, we show that the rules displayed above are derivable, which completes the proof that the promotion rules are derivable.

\begin{itemize}
\item Derivations of left- and right-promotion
\end{itemize}
\begin{center}
\begin{tabular}{cc}
\AX$\wdiaw \bboxb X \MANDOR A \fCenter \wboxw \bdiab Y$
\dashedLine
\UI$\WCIRCW \Gamma \MANDOR A \fCenter \WCIRCW \Pi$
\UI$A \fCenter \WCIRCW \Gamma \MARR \WCIRCW \Pi$
\RightLabel{\fns FS}
\UI$A \fCenter \WCIRCW (\Gamma \KKRARR \Pi)$
\UI$\BCIRCB A \fCenter \Gamma \KKRARR \Pi$
\UI$\bdiab A \fCenter \Gamma \KKRARR \Pi$
\UI$\wboxw \bdiab A \fCenter \WCIRCW (\Gamma \KKRARR \Pi)$
\RightLabel{\fns co-FS}
\UI$\wboxw \bdiab A \fCenter \WCIRCW \Gamma \MARR \WCIRCW \Pi$
\UI$\WCIRCW \Gamma \MANDOR \wboxw \bdiab A \fCenter \WCIRCW \Pi$
\dashedLine
\UI$\wdiaw \bboxb X \MANDOR \wboxw \bdiab A \fCenter \wboxw \bdiab Y$
\DisplayProof
 &
\AX$\wdiaw \bboxb X \fCenter A \MANDOR \wboxw \bdiab Y$
\dashedLine
\UI$\WCIRCW \Gamma \fCenter A \MANDOR \WCIRCW \Pi$
\UI$\WCIRCW \Gamma \fCenter \WCIRCW \Pi \MANDOR A$
\UI$\WCIRCW \Pi \MARR \WCIRCW \Gamma \fCenter A$
\LeftLabel{\fns FS}
\UI$\WCIRCW (\Pi \KKBARR \Gamma) \fCenter A$
\UI$\Pi \KKBARR \Gamma \fCenter \BCIRCB A$
\UI$\Pi \KKBARR \Gamma \fCenter \bboxb A$
\UI$\WCIRCW (\Pi \KKBARR \Gamma) \fCenter \wdiaw \bboxb A$
\LeftLabel{\fns co-FS}
\UI$\WCIRCW \Pi \MARR \WCIRCW \Gamma \fCenter \wdiaw \bboxb A$
\UI$\WCIRCW \Gamma \fCenter \WCIRCW \Pi \MANDOR \wdiaw \bboxb A$
\UI$\WCIRCW \Gamma \fCenter \wdiaw \bboxb A \MANDOR \WCIRCW \Pi$
\dashedLine
\UI$\wdiaw \bboxb X \fCenter \wdiaw \bboxb A \MANDOR \wboxw \bdiab Y$
\DisplayProof
 \\
\end{tabular}
\end{center}

\begin{itemize}
\item Derivations of left- and right-contraction:
\end{itemize}
\begin{center}
\begin{tabular}{cc}
\AX$X \MANDOR (\wdiaw \bboxb A \MANDOR \wdiaw \bboxb A) \fCenter Y$
\dashedLine
\UI$X \MANDOR (\WCIRCW \bboxb A \MANDOR \WCIRCW \bboxb A) \fCenter Y$
\UI$\WCIRCW \bboxb A \MANDOR \WCIRCW \bboxb A \fCenter X \MARR Y$
\LeftLabel{\fns $C_m$}
\UI$\WCIRCW \bboxb A \fCenter X \MARR Y$
\UI$\wdiaw \bboxb A \fCenter X \MARR Y$
\UI$X \MANDOR \wdiaw \bboxb A \fCenter Y$
\DisplayProof
 &
 \AX$X \fCenter (\wboxw \bdiab A \MANDOR \wboxw \bdiab A) \MANDOR Y$
 \dashedLine
\UI$X \fCenter (\WCIRCW \bdiab A \MANDOR \WCIRCW \bdiab A) \MANDOR Y$
\UI$X \fCenter Y \MANDOR (\WCIRCW \bdiab A \MANDOR \WCIRCW \bdiab A)$
\UI$Y \MARR X \fCenter \WCIRCW \bdiab A \MANDOR \WCIRCW \bdiab A$
\RightLabel{\fns $C_m$}
\UI$Y \MARR X \fCenter \WCIRCW \bdiab A$
\UI$Y \MARR X \fCenter \wboxw \bdiab A$
\UI$X \fCenter Y \MANDOR \wboxw \bdiab A$
\UI$X \fCenter \wboxw \bdiab A \MANDOR Y$
\DisplayProof
 \\
\end{tabular}
\end{center}

\subsection{Conservativity}
\label{ssec:conservativity}
To argue that (each) calculus introduced in Section \ref{sec:ProperDisplayCalculiForLinearLogics} adequately captures its associated linear logic, we follow the standard proof strategy discussed in \cite{GMPTZ,GKPLori}. Let $\mathrm{LL}$ denote a Hilbert-style presentation of (one of the variants of) linear logic  (viz.~those given in \cite{Ono1993,Troelstra1992}); let $\vdash_{\mathrm{LL}}$ denote the syntactic consequence relation arising from $\mathrm{LL}$, and let $\models_{\mathrm{HA}_{\mathrm{LL}}}$ denote the semantic consequence relation arising from (perfect) heterogeneous $\mathrm{LL}$-algebras. We need to show that, for all formulas $A$ and $B$ of the original language of linear logic, if $A^t\vdash B^t$ is a D.LL-derivable sequent, then  $A\vdash_{\mathrm{LL}} B$. This claim can be proved using  the following  facts: (a) the rules of D.LL are sound w.r.t.~perfect heterogeneous $\mathrm{LL}$-algebras  (cf.~Section \ref{ssec:soundness}),  (b) $\mathrm{LL}$ is strongly complete w.r.t.~perfect  $\mathrm{LL}$-algebras, and (c) perfect  $\mathrm{LL}$-algebras are equivalently presented as  perfect heterogeneous $\mathrm{LL}$-algebras (cf.~Section \ref{ssec: reverse engineering}), so that the sematic consequence relations arising from each type of structures preserve and reflect the translation (cf.~Proposition \ref{prop:consequence preserved and reflected}). Then, let $A, B$ be formulas of  the original $\mathrm{LL}$-language. If  $A^t\vdash B^t$ is a D.LL-derivable sequent, then, by (a),  $A^t\models_{\mathrm{HA}_{\mathrm{LL}}} B^t$. By (c), this implies that $A\models_{\mathrm{LL}} B$, where $\models_{\mathrm{LL}}$ denotes the semantic consequence relation arising from (perfect) $\mathrm{LL}$-algebras. By (b), this implies that $A\vdash_{\mathrm{LL}} B$, as required.

\subsection{Cut elimination and subformula property}
\label{ssec:CutEliminationAndSubformulaProperty}

In the present section, we outline the proof of cut elimination and subformula property for the calculi introduced in Section \ref{sec:ProperDisplayCalculiForLinearLogics}. As discussed earlier on, these calculi have been designed so that the cut elimination and subformula property  do not need to be proved via the original argument by Gentzen, but can rather be inferred from a meta-theorem, following the strategy introduced by Belnap for display calculi. The meta-theorem to which we will appeal for the calculi of Section \ref{sec:ProperDisplayCalculiForLinearLogics} was proved in \cite{TrendsXIII}, and  in Section \ref{appendix: cut elim} we report on a restricted version of it (cf.~Theorem \ref{theor:cut elim}) which is the one specifically applying to {\em proper multi-type display calculi} (cf.~Definition \ref{def:proper display calculi}).

By Theorem \ref{theor:cut elim}, it is enough to verify that the calculi of Section \ref{sec:ProperDisplayCalculiForLinearLogics} meet the conditions  listed in Definition \ref{def:proper display calculi}. All conditions except C$_8$ are readily satisfied by inspecting the rules. In what follows we verify C$_8$. This requires to check that reduction steps are available for every application of the cut rule in which both cut-formulas are principal, which either remove the original cut altogether or replace it by one or more cuts on formulas of strictly lower complexity.

\paragraph*{Atomic propositions:}

\begin{center}
\begin{tabular}{ccc}
\bottomAlignProof
\AX$p \fCenter p$
\AX$p \fCenter p$
\BI$p \fCenter p$
\DisplayProof

 & $\rightsquigarrow$ &

\bottomAlignProof
\AX$p \fCenter p$
\DisplayProof
 \\
\end{tabular}
\end{center}

\paragraph*{Constants:}

\begin{center}
\begin{tabular}{ccc}
\bottomAlignProof
\AX$\MTOPBOT \fCenter \mtop$
\AXC{\ \ \ $\vdots$ \raisebox{1mm}{$\pi_1$}}
\noLine
\UI$\MTOPBOT \fCenter X$
\UI$\mtop \fCenter X$
\BI$\MTOPBOT \fCenter X$
\DisplayProof

 & $\rightsquigarrow$ &

\bottomAlignProof
\AXC{\ \ \ $\vdots$ \raisebox{1mm}{$\pi_1$}}
\noLine
\UI$\MTOPBOT \fCenter X$
\DisplayProof
 \\
\end{tabular}
\end{center}

\noindent The cases for $\mbot$, $\aatop$, $\abot$ are standard and similar to the one above.

\paragraph*{Binary connectives monotone in each coordinate:}

\begin{center}
\begin{tabular}{ccc}
\!\!\!\!\!
\bottomAlignProof
\AXC{\ \ \ $\vdots$ \raisebox{1mm}{$\pi_1$}}
\noLine
\UI$X \fCenter A$
\AXC{\ \ \ $\vdots$ \raisebox{1mm}{$\pi_2$}}
\noLine
\UI$Y \fCenter B$
\BI$X \MANDOR Y \fCenter A \mand B$
\AXC{\ \ \ $\vdots$ \raisebox{1mm}{$\pi_3$}}
\noLine
\UI$A \MANDOR B \fCenter Z$
\UI$A \mand B \fCenter Z$
\BI$X \MANDOR Y \fCenter Z$
\DisplayProof

 & $\rightsquigarrow$ &

\!\!\!\!\!\!\!\!\!\!\!\!\!\!\!\!\!\!\!\!
\bottomAlignProof
\AXC{\ \ \ $\vdots$ \raisebox{1mm}{$\pi_1$}}
\noLine
\UI$X \fCenter A$
\AXC{\ \ \ $\vdots$ \raisebox{1mm}{$\pi_2$}}
\noLine
\UI$Y \fCenter B$
\AXC{\ \ \ $\vdots$ \raisebox{1mm}{$\pi_3$}}
\noLine
\UI$A \MANDOR B \fCenter Z$
\UI$B \fCenter A \MARR Z$
\BI$Y \fCenter A \MARR Z$
\UI$A \MANDOR Y \fCenter Z$
\UI$Y \MANDOR A \fCenter Z$
\UI$A \fCenter Y \MARR Z$
\BI$X \fCenter Y \MARR Z$
\UI$Y \MANDOR X \fCenter Z$
\UI$X \MANDOR Y \fCenter Z$
\DisplayProof
 \\
\end{tabular}
\end{center}

\noindent The cases for $A \mand B$, $A \aand B$, $A \mor B$, $A \aor B$ are standard and similar to the one above.

\paragraph*{Binary connectives with some antitone coordinate:}

\begin{center}
\begin{tabular}{ccc}
\!\!\!\!\!
\bottomAlignProof
\AXC{\ \ \ $\vdots$ \raisebox{1mm}{$\pi_1$}}
\noLine
\UI$X \fCenter A \MARR B$
\UI$X \fCenter A \mrarr B$

\AXC{\ \ \ $\vdots$ \raisebox{1mm}{$\pi_2$}}
\noLine
\UI$Y \fCenter A$
\AXC{\ \ \ $\vdots$ \raisebox{1mm}{$\pi_3$}}
\noLine
\UI$B \fCenter Z$
\BI$A \mrarr B \fCenter Y \MARR Z$
\BI$X \fCenter Y \MARR Z$
\DisplayProof

 & $\rightsquigarrow$ &

\!\!\!\!\!\!\!\!\!\!\!\!\!\!\!\!\!\!\!\!
\bottomAlignProof
\AXC{\ \ \ $\vdots$ \raisebox{1mm}{$\pi_2$}}
\noLine
\UI$Y \fCenter A$
\AXC{\ \ \ $\vdots$ \raisebox{1mm}{$\pi_1$}}
\noLine
\UI$X \fCenter A \MARR B$
\UI$A \MANDOR X \fCenter B$
\UI$X \MANDOR A \fCenter B$
\UI$A \fCenter X \MARR B$
\BI$Y \fCenter X \MARR B$
\UI$X \MANDOR Y \fCenter B$
\UI$Y \MANDOR X \fCenter B$

\AXC{\ \ \ $\vdots$ \raisebox{1mm}{$\pi_3$}}
\noLine
\UI$B \fCenter Z$
\BI$Y \MANDOR X \fCenter Z$
\UI$X \fCenter Y \MARR Z$
\DisplayProof
 \\
\end{tabular}
\end{center}

\noindent The case for $A \mdrarr B$ is standard and similar to the one above.

\paragraph*{Unary multi-type connectives:}

\begin{center}
\begin{tabular}{ccc}
\bottomAlignProof
\AXC{\ \ \ $\vdots$ \raisebox{1mm}{$\pi_1$}}
\noLine
\UI$\Gamma \fCenter \alpha$
\UI$\WCIRCW \Gamma \fCenter \wdiaw \alpha$
\AXC{\ \ \ $\vdots$ \raisebox{1mm}{$\pi_2$}}
\noLine
\UI$\WCIRCW \alpha \fCenter X$
\UI$\wdiaw \alpha \fCenter X$
\BI$\WCIRCW \Gamma \fCenter X$
\DisplayProof

 & $\rightsquigarrow$ &

\!\!\!\!\!\!\!
\bottomAlignProof
\AXC{\ \ \ $\vdots$ \raisebox{1mm}{$\pi_1$}}
\noLine
\UI$\Gamma \fCenter \alpha$
\AXC{\ \ \ $\vdots$ \raisebox{1mm}{$\pi_2$}}
\noLine
\UI$\WCIRCW \alpha \fCenter X$
\UI$\alpha \fCenter \BCIRCB X$
\BI$\Gamma \fCenter \BCIRCB X$
\UI$\WCIRCW \Gamma \fCenter X$
\DisplayProof
 \\
\end{tabular}
\end{center}

\begin{center}
\begin{tabular}{ccc}
\bottomAlignProof
\AXC{\ \ \ $\vdots$ \raisebox{1mm}{$\pi_1$}}
\noLine
\UI$\Gamma \fCenter \BCIRCB A$
\UI$\Gamma \fCenter \bboxb A$
\AXC{\ \ \ $\vdots$ \raisebox{1mm}{$\pi_2$}}
\noLine
\UI$A \fCenter X$
\UI$\bboxb A \fCenter \BCIRCB X$
\BI$\Gamma \fCenter \BCIRCB X$
\DisplayProof

 & $\rightsquigarrow$ &

\!\!\!\!\!\!\!
\bottomAlignProof
\AXC{\ \ \ $\vdots$ \raisebox{1mm}{$\pi_1$}}
\noLine
\UI$\Gamma \fCenter \BCIRCB A$
\UI$\WCIRCW \Gamma \fCenter A$
\AXC{\ \ \ $\vdots$ \raisebox{1mm}{$\pi_2$}}
\noLine
\UI$A \fCenter X$
\BI$\WCIRCW \Gamma \fCenter X$
\UI$\Gamma \fCenter \BCIRCB X$
\DisplayProof
 \\
\end{tabular}
\end{center}

\noindent The cases for $\wboxw \xi$ and $\bdiab A$ are standard and similar to the ones above.


\section{Structural control}\label{sec:structural control}
We have argued that, because modularity is in-built in proper display calculi, embedding the proof theory of linear logic in the framework of proper display calculi helps to make the connections between linear logic and other neighbouring logics more systematic. Specifically, by adding the appropriate analytic rules to the version of D.LL corresponding to each linear logic (intuitionistic, bi-intuitionistic and classical) considered in this paper, we can  capture e.g.~the affine and relevant counterparts of each linear logic, while preserving all properties of the basic systems. Moreover, fragments and expansions of linear logic can be captured in the same way, thus creating a framework which accounts for substructural logics. Finally, each of these calculi can be further embedded in richer calculi, for instance to obtain proper display calculi for dynamic epistemic logics on substructural propositional bases.\footnote{In the present paper we have treated the environment of {\em distributive} linear logic with exponentials. However, general linear logic with exponentials can be accounted for by treating the additive connectives as done in \cite{GrecoPalmigianoLatticeLogic}, where a proper display calculus is introduced for the logic of general lattices. Again, we wish to stress that it is the modularity of proper display calculi which makes it possible to obtain straightforwardly proper display calculi for general (i.e.~non-distributive) linear logic by conjoining the treatment of multiplicative connectives and exponentials developed in the present paper to the treatment of additive connectives  developed in \cite{GrecoPalmigianoLatticeLogic}.}

Establishing these connections is very useful for transferring techniques, insights and results from one logical setting to another. In the present section, we give one example of this transfer of techniques from linear logic to categorial grammar in linguistics.

\medskip

In categorial grammar, the proof-theoretic framework of Lambek calculus and some of its extensions are used for generating grammatically well-formed sentences in natural language by means of logical derivations. One crucial problem in this area is accounting for the fact that   grammar rules often admit {\em exceptions}, understood as rules which yield grammatically non well-formed constructions if applied unrestrictedly, but grammatically well-formed sentences if applied in a  controlled way. In \cite{KurtoninaMoortgat1997}, Kurtonina and Moortgat propose a proof-theoretic framework which accounts for the controlled  application of  associativity, commutativity and their combination (among others). Their proposal is conceptually akin to the exponentials in linear logic. Indeed, the basic language of their proposal is an expansion of the basic Lambek calculus with two modal operators adjoint to one another,  inspired to the modal operators into which $\oc$ decomposes. 
In fact, the requirements on the modal operators of \cite{KurtoninaMoortgat1997} would perfectly match the multi-type modal operators introduced in the present paper, were it not for the fact that they are captured algebraically  as operations {\em internal} to an FL-algebra, rather than having different algebras as domain and codomain. That is, Kurtonina and Moortgat adopt a single-type environment. In what follows, we recast (a fragment of) Kurtonina and Moortgat's framework for structural control in a multi-type setting.

\medskip
 We consider three different multi-type environments, each of which includes two types: a $\mathsf{General}$ type (corresponding to Lambek calculus), and a $\mathsf{Special}$ type, corresponding to associative, commutative, and associative+commutative Lambek calculus, respectively. The three environments have the same language, specified as follows:
 \begin{align*}
 \mathsf{General}\ni  A ::= & \,p \mid \, \wdiaw\alpha   \mid A \mand A \mid A \mapsto A \mid A {\,\ensuremath{\rotatebox[origin=c]{180}{$\mapsto$}}\,} A\\
\mathsf{Special}\ni  \alpha ::= &\, \bboxb A \mid \alpha \odot \alpha \mid  \alpha  \rightarrowtail \alpha \mid \alpha \leftarrowtail \alpha
\end{align*}

The language above is interpreted into algebraic structures $(\mathbb{L}, \mathbb{A}, \bboxb, \wdiaw)$  such that:
\begin{enumerate}
\item[FL1.] $\mathbb{L} = (L, \leq, \mand, \mapsto, {\,\ensuremath{\rotatebox[origin=c]{180}{$\mapsto$}}\,})$ is a partially ordered algebra;
\item[FL2.] $a\mand b\leq c$ iff $a\leq c{\,\ensuremath{\rotatebox[origin=c]{180}{$\mapsto$}}\,} b$ iff $b\leq a\mapsto c$ for all $a, b, c\in L$;
\item[FL3.] $\mathbb{A} = (A, \leq, \odot, \rightarrowtail, \leftarrowtail)$ is a partially ordered algebra;
\item[FL4.] $\alpha\odot \beta\leq \gamma$ iff $\alpha\leq \gamma\rightarrowtail \beta$ iff $\beta\leq \alpha \leftarrowtail \gamma$ for all $\alpha, \beta, \gamma\in A$;
\item[FL5.] $\bboxb: \mathbb{L}\to \mathbb{A}$ and $\wdiaw: \mathbb{A}\to\mathbb{L}$ are such that $\wdiaw\dashv \bboxb$ and $\bboxb\wdiaw = Id_{\mathbb{A}}$;
\item[FL6.] $\wdiaw \alpha\mand \wdiaw \beta = \wdiaw (\alpha\odot\beta)$ for for all $\alpha, \beta \in A$.
\end{enumerate}
Structures $(\mathbb{L}, \mathbb{A}, \bboxb, \wdiaw)$ satisfying FL1-FL6 will be referred to as {\em heterogeneous FL-algebras}. Any such  structure is {\em associative} if in addition
\begin{enumerate}
\item[FL7.] $\alpha\odot (\beta\odot \gamma) = (\alpha\odot\beta)\odot \gamma$ for all $\alpha, \beta, \gamma\in A$,
\end{enumerate}
and is {\em commutative} if
\begin{enumerate}
\item[FL8.] $\alpha\odot \beta = \beta\odot \alpha$ for all $\alpha, \beta\in A$.
\end{enumerate}

With an argument similar to the one given in Proposition \ref{prop:reverse engineering}, one shows that any heterogeneous FL-algebra gives rise to an algebra $\mathbb{L} = (L, \leq, \mand, \mapsto, {\,\ensuremath{\rotatebox[origin=c]{180}{$\mapsto$}}\,}, \oc)$ such that $\mathbb{L} = (L, \leq, \mand, \mapsto, {\,\ensuremath{\rotatebox[origin=c]{180}{$\mapsto$}}\,})$ satisfies FL1 and FL2, and $\oc: L\to L$ defined as $\oc a: = \wdiaw\bboxb a$ satisfies S2 and S3. With the help of $\oc$ (or, equivalently, of $\wdiaw\bboxb$), the controlled commutativity and associativity in $\mathbb{L}$ can be expressed as follows:
\[\oc A\mand(\oc B\mand \oc C) = (\oc A\mand\oc B)\mand \oc C \quad \mbox{ and }\quad \oc A\mand\oc B = \oc B \mand \oc A,\]
which corresponds to the full internal labelling discussed in \cite[Section 3.2]{KurtoninaMoortgat1997}.
A basic multi-type display calculus D.FL can be straightforwardly introduced along the lines of D.LL in the following language:

\begin{center}
\begin{tabular}{|c|c|c|c|c|c|c|c|c|c|c|c|c|c|c|c|}
\hline
\mc{6}{|c|}{\fns Pure General-type} & \mc{6}{|c|}{\fns Pure Special-type} & \mc{4}{|c|}{\fns Multi-type} \\
\hline
\mc{2}{|c|}{$,$} & \mc{2}{c|}{$\gg$} & \mc{2}{c|}{$\ll$} & \mc{2}{|c|}{$;$} & \mc{2}{c|}{$>$} & \mc{2}{c|}{$<$} & \mc{2}{|c|}{$\BCIRCB$} & \mc{2}{c|}{$\WCIRCW$} \\
\hline
$\mand$ & $\phantom{\times}$ & $\phantom{\mapsto}$ & $\mapsto$ & $\phantom{\ensuremath{\rotatebox[origin=c]{180}{$\mapsto$}}}$ & \mc{1}{c|}{\ensuremath{\rotatebox[origin=c]{180}{$\mapsto$}}} & \mc{1}{|c|}{$\odot$} & $\phantom{\odot}$ & $\phantom{\rightarrowtail}$ & $\rightarrowtail$ & $\phantom{\leftarrowtail}$ & \mc{1}{c|}{$\leftarrowtail$} & \mc{1}{|c|}{$\phantom{\bboxb}$} & $\bboxb$ & $\wdiaw$ & $\phantom{\wdiaw}$ \\
\hline
\end{tabular}
\end{center}
 The rules of D.FL include  identity axioms for atomic formulas (of $\mathsf{General}$ type), cut rules for both types, display postulates for the pure-type connectives modelled on conditions FL2 and FL4, display postulates for the multi-type connectives modelled on FL5, standard introduction rules for all connectives,\footnote{I.e.,  the left-introduction rules for $\mand, \odot$ and $\wdiaw$ are invertible and right-introduction rules of the remaining connectives are invertible.} and the following {\em regrouping/co-regrouping} rule, which captures FL6:
\[\AX$\WCIRCW\Gamma\MANDOR \WCIRCW\Delta\fCenter X$
\LeftLabel{\fns reg / coreg}
\doubleLine
\UI$\WCIRCW(\Gamma\KANDORK \Delta)\fCenter X$
\DisplayProof\]

The associative,  commutative, and associative + commutative extensions of the basic calculus D.FL are respectively defined by adding one, the other or both of the following rules, which hold unrestricted in the appropriate $\mathsf{Special}$ type:
\begin{center}
\begin{tabular}{c c c}
\AX$\Gamma \KANDORK(\Delta\KANDORK \Theta)\fCenter \Lambda$
\LeftLabel{\fns $A_s$}
\doubleLine
\UI$(\Gamma\KANDORK \Delta)\KANDORK \Theta\fCenter \Lambda$
\DisplayProof
&$\quad$&
\AX$\Gamma \KANDORK\Delta\fCenter X$
\LeftLabel{\fns $E_s$}
\doubleLine
\UI$\Delta\KANDORK \Gamma\fCenter X$
\DisplayProof
\end{tabular}
\end{center}
Then,  the appropriate extension of D.FL derives the restricted associativity and commutativity holding in the $\mathsf{General}$ type using the interaction between regrouping and co-regrouping, display rules, and the {\em unrestricted}  associativity and commutativity holding in the $\mathsf{Special}$ type:
\begin{center}
\begin{tabular}{cc}
Exchange & Associativity \\
 & \\
\!\!\!\!\!\!\!\!\!\!\!\!\!\!\!\!\!\!\!\!\!\!
\AX$A \fCenter A$
\UI$\bboxb A \fCenter \BCIRCB A$
\UI$\bboxb A \fCenter \bboxb A$
\UI$\WCIRCW \bboxb A \fCenter \wdiaw \bboxb A$
\AX$B \fCenter B$
\UI$\bboxb B \fCenter \BCIRCB B$
\UI$\bboxb B \fCenter \bboxb B$
\UI$\WCIRCW \bboxb B \fCenter \wdiaw \bboxb B$
\BI$\WCIRCW \bboxb A \MANDOR \WCIRCW \bboxb B \fCenter \wdiaw \bboxb A \mand \wdiaw \bboxb B$
\LeftLabel{\fns reg}
\UI$\WCIRCW (\bboxb A \KANDORK \bboxb B) \fCenter \wdiaw \bboxb A \mand \wdiaw \bboxb B$
\UI$\bboxb A \KANDORK \bboxb B \fCenter \BCIRCB \wdiaw \bboxb A \mand \wdiaw \bboxb B$
\LeftLabel{\fns $E_s$}
\UI$\bboxb B \KANDORK \bboxb A \fCenter \BCIRCB \wdiaw \bboxb A \mand \wdiaw \bboxb B$
\UI$\WCIRCW (\bboxb B \KANDORK \bboxb A) \fCenter \wdiaw \bboxb A \mand \wdiaw \bboxb B$
\LeftLabel{\fns coreg}
\UI$\WCIRCW \bboxb B \MANDOR \WCIRCW \bboxb A \fCenter \wdiaw \bboxb A \mand \wdiaw \bboxb B$
\dashedLine
\UI$\wdiaw \bboxb B \mand \wdiaw \bboxb A \fCenter \wdiaw \bboxb A \mand \wdiaw \bboxb B$
\DisplayProof

 &
\!\!\!\!\!\!\!\!\!\!\!\!\!\!\!\!\!\!\!\!\!\!\!\!\!\!\!\!\!\!\!\!\!\!\!\!
\AX$A \fCenter A$
\UI$\bboxb A \fCenter \BCIRCB A$
\UI$\bboxb A \fCenter \bboxb A$
\UI$\WCIRCW \bboxb A \fCenter \wdiaw \bboxb A$
\AX$B \fCenter B$
\UI$\bboxb B \fCenter \BCIRCB B$
\UI$\bboxb B \fCenter \bboxb B$
\UI$\WCIRCW \bboxb B \fCenter \wdiaw \bboxb B$
\BI$\WCIRCW \bboxb A \MANDOR \WCIRCW \bboxb B \fCenter \wdiaw \bboxb A \mand \wdiaw \bboxb B$
\LeftLabel{\fns reg}
\UI$\WCIRCW (\bboxb A \KANDORK \bboxb B) \fCenter \wdiaw \bboxb A \mand \wdiaw \bboxb B$

\AX$C \fCenter C$
\UI$\bboxb C \fCenter \BCIRCB C$
\UI$\bboxb C \fCenter \bboxb C$
\UI$\WCIRCW \bboxb C \fCenter \wdiaw \bboxb C$
\BI$\WCIRCW (\bboxb A \KANDORK \bboxb B) \MANDOR \WCIRCW \bboxb C \fCenter (\wdiaw \bboxb A \mand \wdiaw \bboxb B) \mand \wdiaw \bboxb C$
\LeftLabel{\fns reg}
\UI$\WCIRCW ((\bboxb A \KANDORK \bboxb B) \KANDORK \bboxb C) \fCenter (\wdiaw \bboxb A \mand \wdiaw \bboxb B) \mand \wdiaw \bboxb C$
\UI$(\bboxb A \KANDORK \bboxb B) \KANDORK \bboxb C \fCenter \BCIRCB (\wdiaw \bboxb A \mand \wdiaw \bboxb B) \mand \wdiaw \bboxb C$
\LeftLabel{\fns $A_s$}
\UI$\bboxb A \KANDORK (\bboxb B \KANDORK \bboxb C) \fCenter \BCIRCB (\wdiaw \bboxb A \mand \wdiaw \bboxb B) \mand \wdiaw \bboxb C$

\UI$\WCIRCW (\bboxb A \KANDORK (\bboxb B \KANDORK \bboxb C)) \fCenter  (\wdiaw \bboxb A \mand \wdiaw \bboxb B) \mand \wdiaw \bboxb C$
\LeftLabel{\fns coreg}
\UI$\WCIRCW \bboxb A \MANDOR \WCIRCW (\bboxb B \KANDORK \bboxb C) \fCenter  (\wdiaw \bboxb A \mand \wdiaw \bboxb B) \mand \wdiaw \bboxb C$
\UI$\WCIRCW (\bboxb B \KANDORK \bboxb C) \fCenter \WCIRCW \bboxb A \MARR (\wdiaw \bboxb A \mand \wdiaw \bboxb B) \mand \wdiaw \bboxb C$
\LeftLabel{\fns coreg}
\UI$\WCIRCW \bboxb B \MANDOR \WCIRCW \bboxb C \fCenter \WCIRCW \bboxb A \MARR (\wdiaw \bboxb A \mand \wdiaw \bboxb B) \mand \wdiaw \bboxb C$

\UI$\WCIRCW \bboxb A \MANDOR (\WCIRCW \bboxb B \MANDOR \WCIRCW \bboxb C) \fCenter (\wdiaw \bboxb A \mand \wdiaw \bboxb B) \mand \wdiaw \bboxb C$
\dashedLine
\UI$\wdiaw \bboxb A \mand (\wdiaw \bboxb B \mand \wdiaw \bboxb C) \fCenter (\wdiaw \bboxb A \mand \wdiaw \bboxb B) \mand \wdiaw \bboxb C$
\DisplayProof
 \\
\end{tabular}
\end{center}

\section{Conclusions}\label{sec:conclusions}
\paragraph{Results.} In the present paper, we have introduced  proper display calculi for several variants of classical, intuitionistic and bi-intuitionistic linear logic, and proved soundness, syntactic completeness, conservativity, cut elimination and subformula property for each. These results are key instances of results in the wider research program of multi-type algebraic proof theory \cite{GrecoPalmigianoMultiTypeAlgebraicProofTheory}, which, generalizing \cite{GMPTZ,MZ16}, integrates algebraic canonicity and correspondence techniques \cite{CoPa12,CFPS15,PaSoZh15,PaSoZh16,FrPaSa16,CoPa:non-dist,CP:constructive,CoCr14,CCPZ,CFPPTW,ConRob,LeRoux:MThesis:2016} into structural proof theory, and is aimed at defining calculi with the same excellent properties of those introduced in the present paper for each member of a  class of logics encompassing both dynamic logics (such as PAL \cite{Plaza}, DEL \cite{BMS,Multitype}, PDL \cite{Kozen,PDL}, game logic \cite{pauly2003game}, coalition logic \cite{pauly2002coalition}) and (non-distributive) substructural logics \cite{galatos2007residuated}. This theory guarantees, among other things, that each of the aforementioned properties transfers to calculi corresponding to {\em fragments} of the linear languages considered in the present paper, and to their analytic {\em axiomatic extensions} and {\em expansions}. As an application of the results and insights developed in the present paper, we have given a multi-type reformulation of the mathematically akin but independently motivated formal framework of  \cite{KurtoninaMoortgat1997},  aimed at extending the use of exponential-type connectives, decomposed into pairs of adjoint modal operators, for the controlled application of structural rules. In particular, we have outlined a multi-type framework in which the applications of commutativity and associativity are controlled in the same way in which applications of weakening and contraction are controlled in linear logic.

\paragraph{From distributive to general linear logic. } In the various  linear logics treated in the present paper,  additive connectives  are {\em distributive}. However, general linear logics  can be also accounted for in the framework of proper multi-type display calculi by treating the additive connectives as done in \cite{GrecoPalmigianoLatticeLogic}, where a proper multi-type display calculus is introduced for the logic of general lattices. Again, this can be achieved straightforwardly thanks to the modularity in-built in the design of proper multi-type display calculi.

\paragraph{Alternative symmetrization of intuitionistic linear logic. } Closely connected to the modularity of (proper) display calculi is their {\em symmetry}. Indeed, while, in the original Gentzen calculi, the difference between e.g.~the classical and the intuitionistic behaviour is captured by a restriction on the shape of sequents which entails that succedent parts are managed differently from precedent ones,   display calculi  sequents have the same unrestricted shape in every setting with no difference in the management of  precedent and succedent parts.  This is why, in the environment introduced in the present paper,  it has been  natural to consider {\em linear subtraction} $\mdrarr$ along with linear implication. This connective makes it possible to realize a symmetrization of intuitionistic linear logic  alternative to the one realized by switching to classical linear logic, and  rather analogous to the one effected by switching from intuitionistic to bi-intuitionistic logic. As discussed in Section \ref{sec: background can ext},  linear subtraction can always be interpreted in {\em perfect} intuitionistic linear algebras as the left residual of $\mor$, and moreover, intuitionistic linear logic is complete w.r.t.~perfect intuitionistic linear algebras. These facts imply that bi-intuitionistic linear logic, defined as the logic of bi-intuitionistic linear algebras (cf.~Definition \ref{def:linear algebras}), conservatively extends intuitionistic linear logic. When it comes to the treatment of exponentials in the bi-intuitionistic setting, we have considered both Ono's interaction axiom  (cf.~P1 in Definition \ref{def:linear algebras}) and its symmetric version expressed in terms of $\mdrarr$ (cf.~BLP2 in Definition \ref{def:linear algebras}). In Section \ref{ssec:from ono axioms to eterogeneous arrows}, we have showed that, while Ono's interaction axiom corresponds to the left-promotion rule (of $\wn$), its symmetric version corresponds to the right-promotion rule (of $\oc$). Bi-intuitionistic linear logic provides an environment in which all the original rules involving exponentials are derivable (i.e.~restricted weakening and contraction, promotion and dereliction rules). We conjecture   that $\oc$ and $\wn$ are not necessarily interdefinable in the bi-intuitionistic linear setting as in classical linear logic. These features are interesting and deserving further investigation. In particular, the residuation between ${\mdrarr}\dashv{\mor}$ can perhaps provide a handle towards an improvement in the understanding of the computational meaning of both connectives.

\paragraph{Extending the boundaries of properness. } Multi-type calculi are a natural generalization of the framework of display calculi, which was introduced by Belnap \cite{Belnap1982} and refined by Wansing \cite{Wansing1998} with the aim of providing an environment in which  ``an indefinite number of logics all mixed together'' can coexist. The main technical motivation for the introduction of proper multi-type display calculi is that they can provide a way for circumventing the characterization theorems of \cite{GMPTZ,CiabattoniRamanayake2016}, which set hard boundaries to the scope of Belnap and Wansing's proper display calculi. The results of the present paper are a case in point: indeed,  linear logic with exponentials cannot be captured by a proper {\em single type} display calculus (cf.~Section \ref{sec: multi-type language}), but is properly displayable in a suitable multi-type setting. Conceptually, multi-type calculi realize a natural prosecution of Belnap's program, and one which is particularly suited to account both for frameworks (such as those of dynamic logics)  characterized by the coexistence of formulas and other heterogeneous elements (such as agents, actions, strategies, coalitions, resources), and for frameworks (such as those of linear and substructural logics) characterized by the coexistence of different ``structural regimes''. Interestingly, both  dynamic logics and  substructural logics have proven difficult to treat with standard proof-theoretic techniques, and  the hurdles in each of these settings can often be understood in terms of the interaction between the different components  of a given logical framework (such as heterogeneous elements in various dynamic logics, and different structural regimes in various substructural logics). The multi-type methodology provides a uniform strategy to address these hurdles (and has been successful in the case of various logics) by recognizing each component explicitly as a {\em type}, introducing enough structural machinery so as to cater for each type on a par with any other type, and then capturing the interaction between different types purely at the structural level, by means of analytic rules involving multi-type structural connectives.

\paragraph{Further directions. } Multi-type calculi form an environment in which it has been possible to settle the question concerning the analitycity of linear logic. However, and perhaps even more interestingly, this  environment also helps to clearly formulate a broad range of questions at various levels of generality, spanning from the one concerning the alternative symmetrization of linear logic discussed above, to the concrete implementation of Girard's research program on `the unity of logic' \cite{girard:unity}. Another such question concerns the systematic exploration of the different versions of the analytic rules which encode the pairing axioms P1 and BLP2 (cf.~Definition \ref{def:linear algebras}) and make it possible to derive the non-analytic promotion/demotion rules. Indeed, these different versions are not equivalent in every setting, which opens up the possibility of making finer-grained distinctions. 
We conjecture that these relations can be expressed 
also in fragments of the languages considered in the present paper, such as the purely positive setting, along the lines of Dunn's positive modal logic \cite{Dunn95}. 


\bibliography{BIB}
\bibliographystyle{plain}

\appendix

\section{Proper multi-type display calculi and their metatheorem}
\label{appendix: cut elim}
For the sake of self-containedness, in the present section we report on (an adaptation of)  the definitions and results of \cite{TrendsXIII}, from which the cut-elimination and subformula property can be straightforwardly inferred for the calculi defined in Section \ref{sec:ProperDisplayCalculiForLinearLogics}.


The calculi defined in Section \ref{sec:ProperDisplayCalculiForLinearLogics} satisfy stronger requirements than those for which the cut elimination meta-theorem \cite[Theorem 4.1]{TrendsXIII} holds. Hence,  below we provide the corresponding restriction of the definition of
quasi-proper multi-type calculus given in \cite{TrendsXIII}, which applies specifically to the calculi of Section \ref{sec:ProperDisplayCalculiForLinearLogics}. 
The resulting definition, given below, is the exact counterpart in the multi-type setting of the definition of proper display calculi introduced in \cite{Wansing1998} and generalized in \cite{GMPTZ}.

A sequent $x\vdash y$ is {\em type-uniform} if $x$ and $y$ are of the same type.
 \begin{definition}
 \label{def:proper display calculi}
 {\em Proper multi-type display calculi} are  those satisfying the following list of conditions:

\paragraph*{C$_1$: Preservation of operational terms.\;} Each operational term occurring in a premise of an inference rule {\em inf} is a subterm of some operational term in the conclusion of {\em inf}.

\paragraph*{C$_2$: Shape-alikeness and type-alikeness of parameters.\;} Congruent parameters\footnote{The congruence relation between non active-parts in rule-applications is understood as derived from the specification of each rule; that is, we assume that each schematic rule of the system comes with an explicit specification of which elements are congruent to which (and then the congruence relation is defined as the reflexive and transitive closure of the resulting relation). Our convention throughout the paper is that congruent parameters are denoted by the same structural variables. For instance, in the rule $$\frac{X;Y\vdash Z}{Y;X\vdash Z}$$ the structures $X,Y$ and $Z$ are parametric and the occurrences of $X$ (resp.\ $Y$, $Z$) in the premise and the conclusion are congruent.} are occurrences of the same structure, and are of the same type.


\paragraph*{C$_3$: Non-proliferation of parameters.\;} Each parameter in an inference rule {\em inf} is congruent to at most one constituent in the conclusion of {\em inf}.

\paragraph*{C$_4$: Position-alikeness of parameters.\;} Congruent parameters are either all in  precedent position or all in succedent position (cf.~Footnote \ref{footnote:def precedent succedent pos}).

\paragraph*{C$_5$: Display of principal constituents.\;} If an operational term $a$ is principal in the conclusion sequent $s$ of a derivation $\pi$, then $a$ is in display. 


\paragraph*{C$_6$: Closure under substitution for succedent parts within each type.\;} Each rule is closed under simultaneous substitution of arbitrary structures for congruent operational terms occurring in succedent position, {\em within each type}.

\paragraph*{C$_7$: Closure under substitution for precedent parts within each type.\;} Each rule is closed under simultaneous substitution of arbitrary structures for congruent operational terms occurring in precedent position, {\em within each type}.

\paragraph*{C$_8$: Eliminability of matching principal constituents.\;}
This condition requests a standard Gentzen-style checking, which is now limited to the case in which  both cut formulas are {\em principal}, i.e.~each of them has been introduced with the last rule application of each corresponding subdeduction. In this case, analogously to the proof Gentzen-style, condition C$_8$ requires being able to transform the given deduction into a deduction with the same conclusion in which either the cut is eliminated altogether, or is transformed in one or more applications of the cut rule, involving proper subterms of the original operational cut-term. 


\paragraph*{C$_9$: Type-uniformity of derivable sequents.} Each derivable sequent is type-uniform.

\paragraph*{C$_{10}$: Preservation of type-uniformity of cut rules.} All cut rules preserve type-uniformity.
\end{definition}

Since proper multi-type display calculi are quasi-proper, the following theorem is an immediate consequence of \cite[Theorem 4.1]{TrendsXIII}:
\begin{theorem}
\label{theor:cut elim} Every proper multi-type display calculus enjoys cut elimination and subformula property.
\end{theorem}

\section{Analytic inductive inequalities}
\label{sec:analytic inductive ineq}
In the present section, we  specialize the definition of {\em analytic inductive inequalities} (cf.\ \cite{GMPTZ,GrecoPalmigianoMultiTypeAlgebraicProofTheory}) to the  multi-type language $\mathcal{L}_\mathrm{MT}$, in the types $\mathsf{Linear}$, $\oc\text{-}\mathsf{Kernel}$ and $\wn\text{-}\mathsf{Kernel}$ (respectively abbreviated as $\mathsf{L}$, $\mathsf{K}_{\oc}$ and $\mathsf{K}_{\wn}$),  defined in Section \ref{sec: multi-type language}  and reported below for the reader's convenience:\footnote{The definition given in the present appendix is applicable to the setting of distributive linear logic only. For the definition in the general setting, the reader is referred to \cite{GrecoPalmigianoMultiTypeAlgebraicProofTheory}.}
\begin{align*}
\mathsf{K}_{\oc}\ni \alpha ::=&\, \iota(A) \mid \ktop \mid \kbot \mid \alpha \kor \alpha \mid  \alpha \kand \alpha \mid \alpha \krarr \alpha \mid \xi\KKBrarr \alpha \\
\mathsf{K}_{\wn}\ni \xi ::=&\, \gamma(A) \mid \topk \mid \botk \mid \xi \andk \xi \mid \xi \ork \xi \mid \xi \drarrk \xi \mid \alpha\KKrarr \xi \\
\mathsf{L}\ni  A ::= & \,p \mid \, e_{\oc}(\alpha) \mid e_{\wn}(\xi) \mid \mtop \mid \mbot \mid A\mneg \mid A \mand A \mid A \mor A \mid A \mrarr A \mid A \mdrarr A\mid \aatop \mid \abot \mid A \aand A \mid A \aor A
\end{align*}
We will make use of the following auxiliary definition: an {\em order-type} over $n\in \mathbb{N}$  is an $n$-tuple $\epsilon\in \{1, \partial\}^n$. For every order type $\epsilon$, we denote its {\em opposite} order type by $\epsilon^\partial$, that is, $\epsilon^\partial(i) = 1$ iff $\epsilon(i)=\partial$ for every $1 \leq i \leq n$.
The connectives of the language above are grouped together  into the  families $\mathcal{F}: = \mathcal{F}_{\mathsf{K}_{\oc}}\cup \mathcal{F}_{\mathsf{K}_{\wn}}\cup \mathcal{F}_{\mathsf{L}}\cup \mathcal{F}_{\mathrm{MT}}$ and $\mathcal{G}: = \mathcal{G}_{\mathsf{K}_{\oc}}\cup \mathcal{G}_{\mathsf{K}_{\wn}}\cup \mathcal{G}_{\mathsf{L}}\cup \mathcal{G}_{\mathrm{MT}}$ defined as follows:
\begin{center}
\begin{tabular}{lcl}
$\mathcal{F}_{\mathsf{K}_{\oc}}: = \{\ktop, \kand\}$ &$\quad$& $\mathcal{G}_{\mathsf{K}_{\oc}}: = \{\kbot, \kor, \krarr\}$\\
$\mathcal{F}_{\mathsf{K}_{\wn}}: = \{\topk, \andk, \drarrk\}$ && $\mathcal{G}_{\mathsf{K}_{\wn}}: = \{\botk, \ork\}$\\
$\mathcal{F}_{\mathsf{L}}: = \{\mtop, \mand, \mdrarr, \aatop, \aand\}$ && $\mathcal{G}_{\mathsf{L}}: = \{\mbot, \mor, \mrarr, \abot, \aor\}$\\
$\mathcal{F}_{\mathrm{MT}}: = \{e_{\oc}, \gamma, \KKBrarr\}$ && $\mathcal{G}_{\mathrm{MT}}: = \{e_{\wn},\iota, \KKrarr \}$.\\
\end{tabular}
\end{center}
For any $f\in \mathcal{F}$  (resp.\ $g\in \mathcal{G}$), we let $n_f\in \mathbb{N}$ (resp.~$n_g\in \mathbb{N}$) denote the arity of $f$ (resp.~$g$), and the order-type $\epsilon_f$ (resp.~$\epsilon_g$) on $n_f$ (resp.~$n_g$)  indicate whether the $i$th coordinate of $f$ (resp.\ $g$) is positive ($\epsilon_f(i) = 1$,  $\epsilon_g(i) = 1$) or  negative ($\epsilon_f(i) = \partial$,  $\epsilon_g(i) = \partial$). The order-theoretic motivation for this partition is that the algebraic interpretations of $\mathcal{F}$-connectives (resp.\ $\mathcal{G}$-connectives), preserve finite joins (resp.\ meets) in each positive coordinate and reverse finite meets (resp.\ joins) in each negative coordinate.

				
For any term $s(p_1,\ldots p_n)$, any order type $\epsilon$ over $n$, and any $1 \leq i \leq n$, an \emph{$\epsilon$-critical node} in a signed generation tree of $s$ is a leaf node $+p_i$ with $\epsilon(i) = 1$ or $-p_i$ with $\epsilon(i) = \partial$. An $\epsilon$-{\em critical branch} in the tree is a branch ending in an $\epsilon$-critical node. For any term $s(p_1,\ldots p_n)$ and any order type $\epsilon$ over $n$, we say that $+s$ (resp.\ $-s$) {\em agrees with} $\epsilon$, and write $\epsilon(+s)$ (resp.\ $\epsilon(-s)$), if every leaf in the signed generation tree of $+s$ (resp.\ $-s$) is $\epsilon$-critical.
				 We will also write $+s'\prec \ast s$ (resp.\ $-s'\prec \ast s$) to indicate that the subterm $s'$ inherits the positive (resp.\ negative) sign from the signed generation tree $\ast s$. Finally, we will write $\epsilon(s') \prec \ast s$ (resp.\ $\epsilon^\partial(s') \prec \ast s$) to indicate that the signed subtree $s'$, with the sign inherited from $\ast s$, agrees with $\epsilon$ (resp.\ with $\epsilon^\partial$).
				\begin{definition}[\textbf{Signed Generation Tree}]
					\label{def: signed gen tree}
					The \emph{positive} (resp.\ \emph{negative}) {\em generation tree} of any $\mathcal{L}_\mathrm{MT}$-term $s$ is defined by labelling the root node of the generation tree of $s$ with the sign $+$ (resp.\ $-$), and then propagating the labelling on each remaining node as follows:
					\begin{itemize}
						\item[] For any node labelled with $\ell\in \mathcal{F}\cup \mathcal{G}$ of arity $n_\ell\geq 1$, and for any $1\leq i\leq n_\ell$, assign the same (resp.\ the opposite) sign to its $i$th child node if $\epsilon_\ell(i) = 1$ (resp.\ if $\epsilon_\ell(i) = \partial$).
					\end{itemize}
					Nodes in signed generation trees are \emph{positive} (resp.\ \emph{negative}) if are signed $+$ (resp.\ $-$).
				\end{definition}
				
		\begin{definition}[\textbf{Good branch}]
					\label{def:good:branch}
					Nodes in signed generation trees will be called \emph{$\Delta$-adjoints}, \emph{syntactically left residual (SLR)}, \emph{syntactically right residual (SRR)}, and \emph{syntactically right adjoint (SRA)}, according to the specification given in Table \ref{Join:and:Meet:Friendly:Table}.
					A branch in a signed generation tree $\ast s$, with $\ast \in \{+, - \}$, is called a \emph{good branch} if it is the concatenation of two paths $P_1$ and $P_2$, one of which may possibly be of length $0$, such that $P_1$ is a path from the leaf consisting (apart from variable nodes) only of PIA-nodes\footnote{For explanations of our choice of terminologies here, we refer to \cite[Remark 3.24]{PaSoZh16}.}, and $P_2$ consists (apart from variable nodes) only of Skeleton-nodes. 
					\begin{table}[\here]
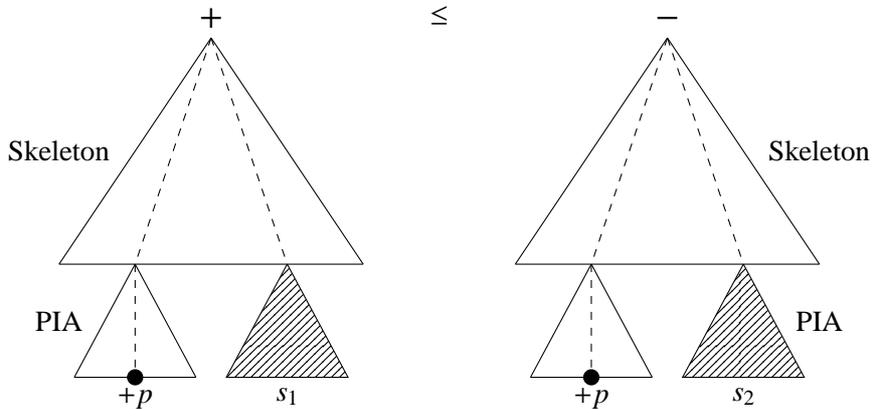
						\begin{center}
							\begin{tabular}{| c | c |}
								\hline
								Skeleton &PIA\\
								\hline
								$\Delta$-adjoints & SRA \\
								\begin{tabular}{ c c c c c c c }
									$+$ &$\kor$ &$\ork$ &$\aor$ &$\kand$ &$\andk$ &$\aand$\\
									$-$ & $\kand$ &$\andk$ &$\aand$ &$\kor$ &$\ork$ &$\aor$\\
									\hline
								\end{tabular}
								&
								\begin{tabular}{c c c c c c c c c}
									$+$ & $\ktop$ & $\topk$ & $\aatop$ &$\kand$ &$\andk$ &$\aand$ &$e_{\wn}$ &$\iota$ \\
									$-$ & $\kbot$ & $\botk$ & $\abot$ & $\kor$ &$\ork$ &$\aor$ &$e_{\oc}$ &$\gamma$  \\
									\hline
								\end{tabular}
								\\
								SLR &SRR\\
								\begin{tabular}{c c c c c c c}
									$+$ & $\phantom{\kbot}$ & $\phantom{\kbot}$ & $\kbot$ & $\botk$ & $\abot$ &$f$  \\
									$-$ & && $\ktop$ & $\topk$ & $\aatop$ & $g$   \\
								\end{tabular}
								&\begin{tabular}{c c c c c c c}
									$+$ &$\phantom{\vee}$&$\phantom{\vee}$&$\phantom{\vee}$ &$g$  &$\phantom{\vee}$ & with $n_g  = 2$\\
									$-$ &$\phantom{\vee}$&$\phantom{\vee}$&$\phantom{\wedge}$ &$f$  &$\phantom{\vee}$& with $n_f  = 2$\\
								\end{tabular}
								\\
								\hline
							\end{tabular}
						\end{center}
						\caption{Skeleton and PIA nodes.}\label{Join:and:Meet:Friendly:Table}
						\vspace{-1em}
					\end{table}
				\end{definition}
	
				\begin{center}
		\begin{tikzpicture}
		\draw (-5,-1.5) -- (-3,1.5) node[above]{\Large$+$} ;
		\draw (-5,-1.5) -- (-1,-1.5) ;
		\draw (-3,1.5) -- (-1,-1.5);
		\draw (-5,0) node{Skeleton} ;
		\draw[dashed] (-3,1.5) -- (-4,-1.5);
		\draw[dashed] (-3,1.5) -- (-2,-1.5);
		\draw (-4,-1.5) --(-4.8,-3);
		\draw (-4.8,-3) --(-3.2,-3);
		\draw (-3.2,-3) --(-4,-1.5);
		\draw[dashed] (-4,-1.5) -- (-4,-3);
		\draw[fill] (-4,-3) circle[radius=.1] node[below]{$+p$};
		\draw
		(-2,-1.5) -- (-2.8,-3) -- (-1.2,-3) -- (-2,-1.5);
		\fill[pattern=north east lines]
		(-2,-1.5) -- (-2.8,-3) -- (-1.2,-3);
		\draw (-2,-3.25)node{$s_1$};
		\draw (-5,-2.25) node{PIA} ;
		\draw (0,1.8) node{$\leq$};
		\draw (5,-1.5) -- (3,1.5) node[above]{\Large$-$} ;
		\draw (5,-1.5) -- (1,-1.5) ;
		\draw (3,1.5) -- (1,-1.5);
		\draw (5,0) node{Skeleton} ;
		\draw[dashed] (3,1.5) -- (4,-1.5);
		\draw[dashed] (3,1.5) -- (2,-1.5);
		\draw (2,-1.5) --(2.8,-3);
		\draw (2.8,-3) --(1.2,-3);
		\draw (1.2,-3) --(2,-1.5);
		\draw[dashed] (2,-1.5) -- (2,-3);
		\draw[fill] (2,-3) circle[radius=.1] node[below]{$+p$};
		\draw
		(4,-1.5) -- (4.8,-3) -- (3.2,-3) -- (4, -1.5);
		\fill[pattern=north east lines]
		(4,-1.5) -- (4.8,-3) -- (3.2,-3) -- (4, -1.5);
		\draw (4,-3.25)node{$s_2$};
		\draw (5,-2.25) node{PIA} ;
		\end{tikzpicture}
	\end{center}

				\begin{definition}[\textbf{Analytic inductive inequalities}]
	\label{def:analytic inductive ineq}
					For any order type $\epsilon$ and any irreflexive and transitive relation $<_\Omega$ on $p_1,\ldots p_n$, the signed generation tree $*s$ $(* \in \{-, + \})$ of a term $s(p_1,\ldots p_n)$ is \emph{analytic $(\Omega, \epsilon)$-inductive} if
					\begin{enumerate}
						\item  every branch of $*s$ is good (cf.\ Definition \ref{def:good:branch});
						\item for all $1 \leq i \leq n$, every $m$-ary SRR-node occurring in  any $\epsilon$-critical branch with leaf $p_i$ is of the form $ \circledast(s_1,\dots,s_{j-1},\beta,s_{j+1}\ldots,s_m)$, where for any $h\in\{1,\ldots,m\}\setminus j$: 
						\begin{enumerate}
							\item $\epsilon^\partial(s_h) \prec \ast s$ (cf.\ discussion before Definition \ref{def:good:branch}), and
							%
							\item $p_k <_{\Omega} p_i$ for every $p_k$ occurring in $s_h$ and for every $1\leq k\leq n$.
						\end{enumerate}

					\end{enumerate}
					
					We will refer to $<_{\Omega}$ as the \emph{dependency order} on the variables. An inequality $s \leq t$ is \emph{analytic $(\Omega, \epsilon)$-inductive} if the signed generation trees $+s$ and $-t$ are analytic $(\Omega, \epsilon)$-inductive. An inequality $s \leq t$ is \emph{analytic inductive} if is analytic $(\Omega, \epsilon)$-inductive for some $\Omega$ and $\epsilon$.
				\end{definition}

In each setting in which they are defined, analytic inductive inequalities are 	a subclass of inductive inequalities (cf.~\cite{GMPTZ}). In their turn, inductive inequalities are {\em canonical} (that is, preserved under canonical extensions, as defined in each setting). Hence, the following is an immediate consequence of  general results on the canonicity of inductive inequalities.
\begin{theorem}
\label{theor: analytic inductive is canonical}
Analytic inductive $\mathcal{L}_{\mathrm{MT}}$-inequalities are canonical.
\end{theorem}

\section{Background on canonical extensions}
\label{sec: background can ext}
In the present section, we report on basic notions and facts of canonical extensions of bounded lattices which are used in the present paper. Our presentation is based on \cite{CoPa:non-dist}. The proofs of many basic properties can be found in \cite{gehrke2001bounded,DGP05}.
\begin{definition}
\label{def:can:ext}
				The \emph{canonical extension} of a bounded  lattice $L$ is a complete  lattice $L^\delta$ containing $L$ as a sublattice, such that:
				\begin{enumerate}
					\item \emph{(denseness)} every element of $L^\delta$ can be expressed both as a join of meets and as a meet of joins of elements from $L$;
					\item \emph{(compactness)} for all $S,T \subseteq L$, if $\bigwedge S \leq \bigvee T$ in $L^\delta$, then $\bigwedge F \leq \bigvee G$ for some finite sets $F \subseteq S$ and $G\subseteq T$.
				\end{enumerate}
				\end{definition}
The canonical extension $L^\delta$ of a (distributive) lattice $L$ is a \emph{perfect} (distributive) lattice, i.e.\ a complete (and completely distributive) lattice which is completely join-generated by its completely join-irreducible elements and completely meet-generated by its completely meet-irreducible elements.\footnote{An element $j\in L$ is {\em completely join-irreducible} if $j\neq \bot$ and  for every $S\subseteq L$, if $j\leq \bigvee S$ then $j\in S$. We let $J^{\infty}(L)$ denote the set of the completely join-irreducible elements of $L$. Completely meet-irreducible elements are defined order-dually, and their collection is denoted by $M^{\infty}(L)$.} Moreover, canonical extensions are unique up to isomorphisms fixing the original algebra. An element $k \in L^\delta$ (resp.~$o\in L^\delta$) is \emph{closed} (resp.\ \emph{open}) if is the meet (resp.\ join) of some subset of $L$. Let $K(L^\delta)$ (resp.\ $O(L^\delta)$) be the set of closed (resp.\ open) elements of $L^\delta$. It is easy to see that the denseness condition in Definition \ref{def:can:ext} implies that $J^{\infty}(L^\delta)\subseteq K (L^\delta)$ and $M^{\infty}(L^\delta)\subseteq O (L^\delta)$.

\begin{definition}
\label{def:f sigma and f pi}
For every unary, order-preserving map $f : L \to M$ between bounded lattices, the $\sigma$-{\em extension} of $f$ is defined firstly by declaring, for every $k\in K(L^\delta)$,
$$f^\sigma(k):= \bigwedge\{ f(a)\mid a\in L\mbox{ and } k\leq a\},$$ and then, for every $u\in L^\delta$,
$$f^\sigma(u):= \bigvee\{ f^\sigma(k)\mid k\in K(L^\delta)\mbox{ and } k\leq u\}.$$
The $\pi$-{\em extension} of $f$ is defined firstly by declaring, for every $o\in O(L^\delta)$,
$$f^\pi(o):= \bigvee\{ f(a)\mid a\in L\mbox{ and } a\leq o\},$$ and then, for every $u\in L^\delta$,
$$f^\pi(u):= \bigwedge\{ f^\pi(o)\mid o\in O(L^\delta)\mbox{ and } u\leq o\}.$$
\end{definition}
It is easy to see that the $\sigma$- and $\pi$-extensions of monotone maps  are monotone. Moreover,
the $\sigma$-extension of a map which preserves finite joins (resp.~reverses  finite  meets) will preserve {\em arbitrary} joins (resp.~reverse  {\em arbitrary}  meets). Because canonical extensions are complete lattices, this implies (cf.~\cite[Proposition 7.34]{DaveyPriestley2002}) that the $\sigma$-extension of any such map is a left (Galois) adjoint, that is, its right (resp.~Galois) adjoint exists.

Dually, the $\pi$-extension of a map which preserves finite meets (resp.~reverses  finite  joins) will preserve {\em arbitrary} meets (resp.~reverse  {\em arbitrary}  joins), and hence the $\pi$-extension of any such map is a right (Galois) adjoint, that is, its left (Galois) adjoint exists.

Finally, if  $f: L\to M$ and $g:M\to L$ are such that $f\dashv g$, then $f^\sigma\dashv g^\pi$.

The definitions above apply also to $n$-ary operations which are $\epsilon$-monotone\footnote{That is, are monotone (resp.~antitone) in each coordinate $i$ such that $\epsilon(i) = 1$ (resp.~$\epsilon(i) = \partial$).} for some order type $\epsilon$  over $n$ (cf.\ Section \ref{sec:analytic inductive ineq}). Indeed,
let us first observe that taking  order-duals interchanges closed and open elements:
$K({(L^\delta)}^\partial) = O(L^\delta)$ and $O({(L^\delta)}^\partial) = K(L^\delta)$;  similarly, $K({(L^n)}^\delta) =K(L^\delta)^n$, and $O({(L^n)}^\delta) =O(L^\delta)^n$. Hence,  $K({(L^\delta)}^\epsilon) =\prod_i K(L^\delta)^{\epsilon(i)}$ and $O({(L^\delta)}^\epsilon) =\prod_i O(L^\delta)^{\epsilon(i)}$ for every lattice $L$ and every order-type $\epsilon$ over any $n\in \mathbb{N}$, where
\begin{center}
\begin{tabular}{cc}
$K(L^\delta)^{\epsilon(i)}: =\begin{cases}
K(L^\delta) & \mbox{if } \epsilon(i) = 1\\
O(L^\delta) & \mbox{if } \epsilon(i) = \partial\\
\end{cases}
$ &
$O(L^\delta)^{\epsilon(i)}: =\begin{cases}
O(L^\delta) & \mbox{if } \epsilon(i) = 1\\
K(L^\delta) & \mbox{if } \epsilon(i) = \partial.\\
\end{cases}
$\\
\end{tabular}
\end{center}
From these observations it immediately follows that taking the canonical extension of a lattice  commutes with taking order-duals and products, namely:
${(L^\partial)}^\delta = {(L^\delta)}^\partial$ and ${(L_1\times L_2)}^\delta = L_1^\delta\times L_2^\delta$.
Hence,  ${(L^\partial)}^\delta$ can be  identified with ${(L^\delta)}^\partial$,  ${(L^n)}^\delta$ with ${(L^\delta)}^n$, and
${(L^\epsilon)}^\delta$ with ${(L^\delta)}^\epsilon$ for any order type $\epsilon$ over $n$, where $L^\epsilon: = \prod_{i = 1}^n L^{\epsilon(i)}$.
These identifications make it possible to obtain the definition of $\sigma$-and $\pi$-extensions of $\epsilon$-monotone operations of any arity $n$ and order-type $\epsilon$ over $n$ by instantiating the corresponding definitions given above for monotone and unary functions. The $\sigma$- and $\pi$-extensions of the lattice operations coincide with the lattice operations of the canonical extension.  

\begin{definition}
\label{def: normal (d)le}
For any lattice $L$, an  operation $h$ on $L$ of arity $n_h$ is {\em normal} if it is order-preserving or order-reversing in each coordinate, and moreover one of the following conditions holds: (a) $h$ preserves finite (hence possibly empty) joins in each coordinate in which it is order-preserving and reverses finite meets in  each coordinate in which it is order-reversing; (b) $h$ preserves finite (hence possibly empty) meets in each coordinate in which it is order-preserving and reverses finite joins in  each coordinate in which it is order-reversing.

A {\em normal (distributive) lattice expansion} is an algebra $\mathbb{A} = (L, \mathcal{F}, \mathcal{G})$ such that $L$ is a bounded (distributive) lattice and $\mathcal{F}$ and $\mathcal{G}$ are finite (possibly empty) and disjoint sets of operations on $L$ such that every $f\in \mathcal{F}$ is normal and verifies condition (a), and every $g\in \mathcal{G}$ is normal and verifies condition (b).
\end{definition}
Intuitionistic and bi-intuitionistic linear algebras (without exponentials), Heyting, co-Heyting and bi-Heyting algebras are all examples of normal (distributive) lattice expansions: for intuitionistic linear algebras, take $\mathcal{F}: = \{\mand, \mtop\}$ and $\mathcal{G}: = \{\mrarr, \mor, \mbot\}$; for bi-intuitionistic linear algebras, take $\mathcal{F}: = \{\mdrarr, \mand,  \mtop\}$ and $\mathcal{G}: = \{\mrarr, \mor, \mbot\}$; for Heyting algebras, take $\mathcal{F}: = \varnothing$ and $\mathcal{G}: = \{\krarrk\}$; for co-Heyting algebras, take $\mathcal{F}: = \{\kdrarrk\}$ and $\mathcal{G}: = \varnothing$; for bi-Heyting algebras, take $\mathcal{F}: = \{\krarrk\}$ and $\mathcal{G}: = \{\kdrarrk\}$.
\begin{definition}
\label{def:can ext normal le}
The canonical extension of any normal (distributive) lattice expansion $\mathbb{A} = (L, \mathcal{F}, \mathcal{G} )$ is the normal (distributive) lattice expansion $\mathbb{A}^\delta: = (L^\delta, \mathcal{F}^\sigma, \mathcal{G}^\pi)$, where $L^\delta$  is the canonical extension of  $L$ (cf.\ Definition \ref{def:can:ext}), and $\mathcal{F}^\sigma: = \{f^\sigma\mid f\in \mathcal{F}\}$ and  $\mathcal{G}^\pi: = \{g^\pi\mid g\in \mathcal{G}\}$.
 \end{definition}
The definition above applies in a uniform way to any signature, and is motivated by the fact that, as remarked above, it preserves existing residuations/Galois connections among operations in the original signature. Another noticeable, if more technical feature of this definition is its being  independent on whether the original maps are {\em smooth} (i.e.\ their $\sigma$- and $\pi$-extensions coincide). While all unary normal operations are smooth, normal operations of arity higher than 1 might  not be smooth in general.

It follows straightforwardly from the facts above that the classes of linear algebras without exponentials, Heyting, co-Heyting and bi-Heyting algebras are closed under taking canonical extensions.
It also  follows that the canonical extension of a normal LE $\mathbb{A}$ is a {\em perfect} normal LE:
\begin{definition}
\label{def:perfect LE}
A normal LE $\mathbb{A} = (L, \mathcal{F}, \mathcal{G})$ is {\em perfect} if $L$ is a perfect lattice (cf.\ discussion above), and moreover the following infinitary distribution laws are satisfied for each $f\in \mathcal{F}$, $g\in \mathcal{G}$, $1\leq i\leq n_f$ and $1\leq j\leq n_g$: for every $S\subseteq L$,
\begin{center}
\begin{tabular}{c c }
$f(x_1,\ldots, \bigvee S, \ldots, x_{n_f}) =\bigvee \{ f(x_1,\ldots, x, \ldots, x_{n_f}) \mid x\in S \}$  & if $\epsilon_f(i) = 1$\\

$f(x_1,\ldots, \bigwedge S, \ldots, x_{n_f}) =\bigvee \{ f(x_1,\ldots, x, \ldots, x_{n_f}) \mid x\in S \}$  & if $\epsilon_f(i) = \partial$\\

$g(x_1,\ldots, \bigwedge S, \ldots, x_{n_g}) =\bigwedge \{ g(x_1,\ldots, x, \ldots, x_{n_g}) \mid x\in S \}$  & if $\epsilon_g(i) = 1$\\

$g(x_1,\ldots, \bigvee S, \ldots, x_{n_g}) =\bigwedge \{ g(x_1,\ldots, x, \ldots, x_{n_g}) \mid x\in S \}$  & if $\epsilon_g(i) = \partial$.\\

\end{tabular}
\end{center}

\end{definition}
Before finishing the present subsection, let us spell out  the definitions of the extended operations.

Denoting by $\leq^\epsilon$ the product order on $(L^\delta)^\epsilon$, we have for every $f\in \mathcal{F}$, $g\in \mathcal{G}$,  $\overline{u}\in (L^\delta)^{n_f}$ and $\overline{v}\in (L^\delta)^{n_g}$,
\begin{center}
\begin{tabular}{l l}
$f^\sigma (\overline{k}):= \bigwedge\{ f( \overline{a})\mid \overline{a}\in (L^\delta)^{\epsilon_f}\mbox{ and } \overline{k}\leq^{\epsilon_f} \overline{a}\}$ & $f^\sigma (\overline{u}):= \bigvee\{ f^\sigma( \overline{k})\mid \overline{k}\in K({(L^\delta)}^{\epsilon_f})\mbox{ and } \overline{k}\leq^{\epsilon_f} \overline{u}\}$ \\
$g^\pi (\overline{o}):= \bigvee\{ g( \overline{a})\mid \overline{a}\in (L^\delta)^{\epsilon_g}\mbox{ and } \overline{a}\leq^{\epsilon_g} \overline{o}\}$ & $g^\pi (\overline{v}):= \bigwedge\{ g^\pi( \overline{o})\mid \overline{o}\in O({(L^\delta)}^{\epsilon_g})\mbox{ and } \overline{v}\leq^{\epsilon_g} \overline{o}\}$. \\
\end{tabular}
%
\end{center}

Two facts are worth being highlighted, since they follow a pattern which is key to the conservativity argument given in Section \ref{ssec:conservativity}. Firstly, the algebraic completeness of linear logic (without exponentials), intuitionistic, co-intuitionistic and bi-intuitionistic logic,  and the canonical embedding of the algebras corresponding to these logics into their respective canonical extensions immediately give completeness of each of these logics w.r.t.\ the corresponding class of perfect normal LEs, which is condition (a) in the general conservativity argument of which the one given in Section \ref{ssec:conservativity} is an instance. Secondly, the existence of the adjoints and residuals (in each coordinate) of the extensions of the original operations provides  semantic interpretation to all  structural connectives (including to those the operational counterparts of which do not belong to the original signature). For instance, the canonical extension of any Heyting algebra (resp.~co-Heyting algebra) is naturally endowed with a bi-Heyting algebra structure, since the finite distributivity of joins over meets (resp.~meets over joins) implies complete distributivity holds in the canonical extension, which guarantees the existence of the left residual $\kdrarrk$ of $\vee$ (resp.~the right residual $\krarrk$ of $\wedge$) in the canonical extension. Likewise, the finite distributivity of $\mor$ over $\aand$ in any intuitionistic linear algebra without exponentials guarantees the existence of the left residual $\mdrarr$ of $\mor^\pi$ in the canonical extension, which is then naturally endowed with a structure of bi-intuitionistic linear algebra without exponentials. The existence of all adjoints and residuals  makes it possible to interpret the structural rules of the display calculus intended to capture each logic, and verify their soundness,  which is requirement (b) in the general conservativity argument (cf.~Section Section \ref{ssec:conservativity}). 

\section{Inversion lemmas}
\label{sec: inversion lemmas}
In the present section we prove the general inversion lemmas holding in any proper multi-type display calculus (cf.~Definition \ref{def:proper display calculi})
%
Recall that  rule $R$ is invertible if every premise sequent of $R$ may be derived from the conclusion sequent of $R$.
In what follows, we fix an arbitrary multi-type signature $(\mathcal{F}, \mathcal{G})$ 
(generalizing the presentations of Appendixes \ref{sec:analytic inductive ineq} and \ref{sec: background can ext}. See \cite{GMPTZ} for an extended discussion).
The language for (the corresponding fragment of) the associated calculus is
\begin{center}
\begin{tabular}{|c|c|c|c|}
\hline
\mc{2}{|c|}{$H$} & \mc{2}{c|}{$K$} \\
\hline
$f$ & $\phantom{f}$ & $\phantom{g}$ & $g$ \\
\hline
\end{tabular}
\end{center}
In what follows, we omit reference to types, and use  variables $x, y, z$ to denote structural terms of arbitrary type, and variables $a, b, c$ to denote operational terms of arbitrary type. All sequents are understood to be type-uniform (cf.~Section \ref{appendix: cut elim}).
In any proper multi-type display calculus, the introduction rules for any $f \in \mathcal{F}$ and $g \in \mathcal{G}$ have the following shape (cf.~\cite{GMPTZ}):

\begin{center}
\begin{tabular}{c}
\bottomAlignProof
\AX$H (a_1,\ldots, a_{n_f}) \fCenter x$
\LeftLabel{\fns$f_L$}
\UI$f(a_1,\ldots, a_{n_f}) \fCenter x$
\DisplayProof
\ \ \
\bottomAlignProof
\AxiomC{$\Big(x_i \vdash a_i \quad a_j \vdash x_j \mid 1\leq i, j\leq n_f, \varepsilon_f(i) = 1\mbox{ and } \varepsilon_f(j) = \partial \Big)$}
\RightLabel{\fns$f_R$}
\UnaryInfC{$H(x_1,\ldots, x_{n_f}) \vdash f(a_1,\ldots, a_{n})$}
\DisplayProof
 \\
 \\
\bottomAlignProof
\AxiomC{$\Big(x_i \vdash a_i \quad a_j \vdash x_j \mid 1\leq i, j\leq n_g, \varepsilon_g(i) = 1\mbox{ and } \varepsilon_g(j) = \partial \Big)$}
\LeftLabel{\fns$g_L$}
\UnaryInfC{$g(a_1,\ldots, a_{n}) \vdash K(x_1,\ldots, x_{n_g})$}
\DisplayProof
\ \ \
\bottomAlignProof
\AX$x \fCenter K(a_1,\ldots, a_{n_g})$
\RightLabel{\fns$g_R$}
\UI$x \fCenter g(a_1,\ldots, a_{n_g})$
\DisplayProof
 \\
\end{tabular}
\end{center}
In particular, if $n_f = 0 = n_g$, the rules $f_R$ and $g_L$ above reduce to the axioms ($0$-ary rules) $H\vdash f$ and $g\vdash K$. Using these rules (and the standard introduction rules for lattice connectives), the following lemma can be proved by a routine induction on terms:
\begin{lemma}
\label{lemma: identity}
In any proper multi-type display calculus, all sequents $a\vdash a$ are derivable.
\end{lemma}

\begin{lemma}
\label{lemma: invertibility}
In any proper multi-type display calculus, the left-introduction rule of any $f \in \mathcal{F}$  and the right-introduction rule of any $g \in \mathcal{G}$ are invertible.
\end{lemma}

\begin{proof}
 Using Lemma \ref{lemma: identity}, the following derivation proves the claim  for $f$.

\begin{center}
\AxiomC{$\Big(a_i \vdash a_i \quad a_j \vdash a_j \mid 1\leq i, j\leq n_f, \varepsilon_f(i) = 1\mbox{ and } \varepsilon_f(j) = \partial \Big)$}
\UnaryInfC{$H(a_1,\ldots, a_{n_f})\vdash f(a_1,\ldots, a_{n})$}
\AXC{\ \ $\vdots$ \raisebox{1mm}{$\pi$}}
\noLine
\UI$f(a_1,\ldots, a_{n_f}) \fCenter x$
\RightLabel{Cut}
\BI$H (a_1,\ldots, a_{n_f}) \fCenter x$
\DisplayProof
\end{center}
The proof of the remaining part of the statement is analogous.
\end{proof}
Using Lemma \ref{lemma: invertibility}, and suitably making use of the display property of proper multi-type display calculi, one can prove the following
\begin{corollary}
\begin{enumerate}
In any proper multi-type display calculus,
\item if $(x \fCenter y)[f(a_1,\ldots, a_{n_f})]^{pre}$ is derivable,\footnote{The notation $(x \fCenter y)[f(a_1,\ldots, a_{n_f})]^{pre}$ (resp.~$(x \fCenter y)[g(a_1,\ldots, a_{n_g})]^{succ}$) indicates that $f(a_1,\ldots, a_{n_f})$ (resp.~$g(a_1,\ldots, a_{n_g})$) occurs as a substructure of $x \fCenter y$ in precedent (resp.~succedent) position.} so is $(x \fCenter y)[H(a_1,\ldots, a_{n_f})]^{pre}$;
\item if $(x \fCenter y)[g(a_1,\ldots, a_{n_g})]^{succ}$ is derivable, so is $(x \fCenter y)$ $[K(a_1,\ldots, a_{n_g})]^{succ}$.
    \end{enumerate}
\end{corollary}


\section{Deriving Hilbert-style axioms and rules for exponentials}
\label{sec: deriving axioms}
The relevant axioms and rule capturing the behaviour of $\oc$ were considered in \cite{Avron1988, Troelstra1992}. In \cite{Ono1993} also the algebraic inequalites capturing the behaviour of $\wn$ were considered. Below we reproduce the axioms and rule for $\oc$ and the axioms for $\wn$ corresponding to the algebraic inequalities:

\begin{center}
\begin{tabular}{rl | rl}
\mc{4}{c}{\ \ \ \ \ \ \ \ \ \ \ \ \ Axioms} \\
A1. & $B \mrarr (\oc A \mrarr B)$                                            &  &  \\
A2. & $(\oc A\mrarr (\oc A \mrarr B)) \mrarr (\oc A\mrarr B)$  &  &  \\
A3. & $\oc (A \mrarr B) \mrarr (\oc A \mrarr \oc B)$                & $\oc (A \mrarr B) \vdash \wn A \mrarr \wn B$ & A6. \\
A4. & $\oc A\mrarr A$                                                             & $A \mrarr \wn A$ & A7. \\
A5. & $\oc A\mrarr \oc \oc A$                                                 & $\wn \wn A\mrarr \wn A$ & A8. \\
             & Rule                                                                             & $\wn \mbot \mrarr \mbot, \mbot \mrarr \wn \mbot$ & A9. \\
$\oc$R & $\vdash A \ \Rightarrow \ \, \vdash \oc A$                & $\mbot \mrarr \wn A$ & A10. \\
\end{tabular}
\end{center}

The rule $\oc$R is derivable in D.LL as follows:

\begin{center}
{\fns 
\begin{tabular}{c}
\AX$\MTOPBOT \fCenter A$
\LeftLabel{\fns nec}
\UI$\WCIRCW \KTOPBOTK \fCenter A$
\UI$\KTOPBOTK \fCenter \BCIRCB A$
\UI$\KTOPBOTK \fCenter \bboxb A$
\UI$\WCIRCW \KTOPBOTK \fCenter \wdiaw \bboxb A$
\LeftLabel{\fns conec}
\UI$\MTOPBOT \fCenter \wdiaw \bboxb A$
\DisplayProof
\end{tabular}
 }
\end{center}

All the previous axioms are derivable in D.LL as follows.

\begin{itemize}
\item[A1.] $B \mrarr (\oc A \mrarr B)$\ \ and\ \ A2.\ \ $(\oc A\mrarr (\oc A \mrarr B)) \mrarr (\oc A\mrarr B)$
\end{itemize}
\begin{center}
{\fns 
\begin{tabular}{cc}
\AX$B \fCenter B$
\LeftLabel{\fns $W_L$}
\UI$B \MANDOR \WCIRCW \bboxb A \fCenter B$
\UI$\WCIRCW \bboxb A \fCenter B \MARR B$
\UI$\wdiaw \bboxb A \fCenter B \MARR B$
\LeftLabel{\fns def}
\UI$\oc A \fCenter B \MARR B$
\UI$B \MANDOR \oc A \fCenter B$
\UI$\oc A \MANDOR B \fCenter B$
\UI$B \fCenter \oc A \MARR B$
\UI$B \fCenter \oc A \mrarr B$
\DisplayProof
 & 
\AX$A \fCenter A$
\UI$\bboxb A \fCenter \BCIRCB A$
\UI$\bboxb A \fCenter \bboxb A$
\UI$\WCIRCW \bboxb A \fCenter \wdiaw \bboxb A$
\RightLabel{\fns def}
\UI$\WCIRCW \bboxb A \fCenter \oc A$
\AX$A \fCenter A$
\UI$\bboxb A \fCenter \BCIRCB A$
\UI$\bboxb A \fCenter \bboxb A$
\UI$\WCIRCW \bboxb A \fCenter \wdiaw \bboxb A$
\RightLabel{\fns def}
\UI$\WCIRCW \bboxb A \fCenter \oc A$
\AX$B \fCenter B$
\BI$\oc A \mrarr B \fCenter \WCIRCW \bboxb A \MARR B$
\BI$\oc A \mrarr (\oc A \mrarr B) \fCenter \WCIRCW \bboxb A \MARR (\WCIRCW \bboxb A \MARR B)$
\UI$\WCIRCW \bboxb A \MANDOR \oc A \mrarr (\oc A \mrarr B) \fCenter \oc A \MARR B$
\UI$\WCIRCW \bboxb A \MANDOR (\WCIRCW \bboxb A \MANDOR \oc A \mrarr (\oc A \mrarr B)) \fCenter B$
\UI$(\WCIRCW \bboxb A \MANDOR \WCIRCW \bboxb A) \MANDOR \oc A \mrarr (\oc A \mrarr B) \fCenter B$
\UI$\oc A \mrarr (\oc A \mrarr B) \MANDOR (\WCIRCW \bboxb A \MANDOR \WCIRCW \bboxb A) \fCenter B$
\UI$\WCIRCW \bboxb A \MANDOR \WCIRCW \bboxb A \fCenter \oc A \mrarr (\oc A \mrarr B) \MARR B$
\LeftLabel{\fns $C_L$}
\UI$\WCIRCW \bboxb A \fCenter \oc A \mrarr (\oc A \mrarr B) \MARR B$
\UI$\wdiaw \bboxb A \fCenter \oc A \mrarr (\oc A \mrarr B) \MARR B$
\LeftLabel{\fns def}
\UI$\oc A \fCenter \oc A \mrarr (\oc A \mrarr B) \MARR B$
\UI$\oc A \mrarr (\oc A \mrarr B) \MANDOR \oc A \fCenter B$
\UI$\oc A \MANDOR \oc A \mrarr (\oc A \mrarr B) \fCenter B$
\UI$\oc A \mrarr (\oc A \mrarr B) \fCenter \oc A \MARR B$
\UI$\oc A \mrarr (\oc A \mrarr B) \fCenter \oc A \mrarr B$
\DisplayProof
 \\
\end{tabular}
 }
\end{center}

\begin{itemize}
\item[A3.] $\oc (A \mrarr B) \mrarr (\oc A \mrarr \oc B)$
\end{itemize}
\begin{center}
{\fns
\begin{tabular}{c}
\AX$A \fCenter A$
\UI$\bboxb A \fCenter \BCIRCB A$
\LeftLabel{\fns $W_K$}
\UI$\bboxb A \KANDORK \bboxb (A \mrarr B) \fCenter \BCIRCB A$
\UI$\WCIRCW (\bboxb A \KANDORK \bboxb (A \mrarr B)) \fCenter A$
\AX$B \fCenter B$
\BI$A \mrarr B \fCenter \WCIRCW (\bboxb A \KANDORK \bboxb (A \mrarr B)) \MARR B$
\UI$\bboxb (A \mrarr B) \fCenter \BCIRCB (\WCIRCW (\bboxb A \KANDORK \bboxb (A \mrarr B)) \MARR B)$
\LeftLabel{\fns $W_K$}
\UI$\bboxb A \KANDORK \bboxb (A \mrarr B) \fCenter \BCIRCB (\WCIRCW (\bboxb A \KANDORK \bboxb (A \mrarr B)) \MARR B)$
\UI$\WCIRCW (\bboxb A \KANDORK \bboxb (A \mrarr B)) \fCenter \WCIRCW (\bboxb A \KANDORK \bboxb (A \mrarr B)) \MARR B$
\UIC{$\WCIRCW (\bboxb A \KANDORK \bboxb (A \mrarr B)) \MANDOR \WCIRCW (\bboxb A \KANDORK \bboxb (A \mrarr B))\fCenter B$}
\LeftLabel{\fns reg}
\UIC{$\WCIRCW ((\bboxb A \KANDORK \bboxb (A \mrarr B)) \KANDORK (\bboxb A \KANDORK \bboxb (A \mrarr B))) \fCenter B$}
\UIC{$(\bboxb A \KANDORK \bboxb (A \mrarr B)) \KANDORK (\bboxb A \KANDORK \bboxb (A \mrarr B)) \fCenter \BCIRCB B$}
\LeftLabel{\fns $C_K$}
\UI$ \bboxb A \KANDORK \bboxb (A \mrarr B) \fCenter \BCIRCB B$
\UI$\WCIRCW (\bboxb A \KANDORK \bboxb (A \mrarr B)) \fCenter \wdiaw \bboxb B$
\RightLabel{\fns def}
\UI$\WCIRCW (\bboxb A \KANDORK \bboxb (A \mrarr B) \fCenter \oc B$
\LeftLabel{\fns co-reg}
\UI$\WCIRCW \bboxb A \MANDOR  \WCIRCW \bboxb (A \mrarr B) \fCenter \oc B$
\UI$\WCIRCW \bboxb (A \mrarr B) \fCenter \WCIRCW \bboxb A \MARR \oc B$
\UI$\wdiaw \bboxb (A \mrarr B) \fCenter \WCIRCW \bboxb A \MARR \oc B$
\LeftLabel{\fns def}
\UI$\oc (A \mrarr B) \fCenter \WCIRCW \bboxb A \MARR \oc B$
\UI$\WCIRCW \bboxb A \MANDOR \oc (A \mrarr B) \fCenter \oc B$
\UI$\oc (A \mrarr B) \MANDOR \WCIRCW \bboxb A \fCenter \oc B$
\UI$\WCIRCW \bboxb A \fCenter \oc (A \mrarr B) \MARR \oc B$
\UI$\wdiaw \bboxb A \fCenter \oc (A \mrarr B) \MARR \oc B$
\LeftLabel{\fns def}
\UI$\oc A \fCenter \oc (A \mrarr B) \MARR \oc B$
\UI$\oc (A \mrarr B) \MANDOR \oc A \fCenter \oc B$
\UI$\oc A \MANDOR \oc (A \mrarr B) \fCenter \oc B$
\UI$\oc (A \mrarr B) \fCenter \oc A \MARR \oc B$
\UI$\oc (A \mrarr B) \fCenter \oc A \mrarr \oc B$
\DisplayProof
 \\
\end{tabular}
 }
\end{center}

\begin{itemize}
\item[A4.] $\oc A\mrarr A$ \ and \ A7.\ $A\mrarr \wn A$
\end{itemize}

\begin{center}
{\fns 
\begin{tabular}{cc}
\AX$A \fCenter A$
\UI$\bboxb A \fCenter \BCIRCB A$
\UI$\WCIRCW \bboxb A \fCenter A$
\UI$\wdiaw \bboxb A \fCenter A$
\LeftLabel{\fns def}
\UI$\oc A \fCenter A$
\DisplayProof
 & 
\AX$A \fCenter A$
\UI$\BCIRCB A \fCenter \bdiab A$
\UI$A \fCenter \WCIRCW \bdiab A$
\UI$A \fCenter \wboxw \bdiab A$
\RightLabel{\fns def}
\UI$A \fCenter \wn A$
\DisplayProof
 \\
\end{tabular}
 }
\end{center}

\begin{itemize}
\item[A5.] $\oc A\mrarr \oc \oc A$ \ and \ A8. \ $\wn \wn A\mrarr \wn A$
\end{itemize}
\begin{center}
{\fns 
\begin{tabular}{cc}
\AX$A \fCenter A$
\UI$\bboxb A \fCenter \BCIRCB A$
\UI$\bboxb A \fCenter \bboxb A$
\UI$\WCIRCW \bboxb A \fCenter \wdiaw \bboxb A$
\RightLabel{\fns def}
\UI$\WCIRCW \bboxb A \fCenter \oc A$
\UI$\bboxb A \fCenter \BCIRCB \oc A$
\UI$\bboxb A \fCenter \bboxb \oc A$
\UI$\WCIRCW \bboxb A \fCenter \wdiaw \bboxb \oc A$
\RightLabel{\fns def}
\UI$\WCIRCW \bboxb A \fCenter \oc \oc A$
\UI$\wdiaw \bboxb A \fCenter \oc \oc A$
\LeftLabel{\fns def}
\UI$\oc A \fCenter \oc \oc A$
\DisplayProof
 & 
\AX$A \fCenter A$
\UI$\BCIRCB A \fCenter \bdiab A$
\UI$\bdiab A \fCenter \bdiab A$
\UI$\wboxw \bdiab A \fCenter \WCIRCW \bdiab A$
\LeftLabel{\fns def}
\UI$\wn A \fCenter \WCIRCW \bdiab A$
\UI$\BCIRCB \wn A \fCenter \bdiab A$
\UI$\bdiab \wn A \fCenter \bdiab A$
\UI$\wboxw \bdiab \wn A \fCenter \WCIRCW \bdiab A$
\LeftLabel{\fns def}
\UI$\wn \wn A \fCenter \WCIRCW \bdiab A$
\UI$\wn \wn A \fCenter \wboxw \bdiab A$
\RightLabel{\fns def}
\UI$\wn \wn A \fCenter \wn A$
\DisplayProof
 \\
\end{tabular}
 }
\end{center}

\begin{itemize}
\item[A6.] $\oc (A \mrarr B) \vdash \wn A \mrarr \wn B$
\end{itemize}

\begin{center}
{\fns 
\begin{tabular}{c}
\AX$A \fCenter A$
\AX$B \fCenter B$
\UI$\BCIRCB B \fCenter \bdiab B$
\UI$B \fCenter \WCIRCW \bdiab B$
\BI$A \mrarr B \fCenter A \MARR \WCIRCW \bdiab B$
\UI$\bboxb (A \mrarr B) \fCenter \BCIRCB (A \MARR \WCIRCW \bdiab B)$
\UI$\WCIRCW \bboxb (A \mrarr B) \fCenter A \MARR \WCIRCW \bdiab B$
\UI$A \MANDOR \WCIRCW \bboxb (A \mrarr B) \fCenter \WCIRCW \bdiab B$
\UI$\WCIRCW \bboxb (A \mrarr B) \MANDOR A \fCenter \WCIRCW \bdiab B$
\UI$A \fCenter \WCIRCW \bboxb (A \mrarr B) \MARR \WCIRCW \bdiab B$
\RightLabel{\fns FS}
\UI$A \fCenter \WCIRCW (\bboxb (A \mrarr B) \KKRARR \bdiab B)$
\UI$\BCIRCB A \fCenter \bboxb (A \mrarr B) \KKRARR \bdiab B$
\UI$\bdiab A \fCenter \bboxb (A \mrarr B) \KKRARR \bdiab B$
\UI$\wboxw \bdiab A \fCenter \WCIRCW (\bboxb (A \mrarr B) \KKRARR \bdiab B)$
\RightLabel{\fns co-FS}
\UI$\wboxw \bdiab A \fCenter \WCIRCW \bboxb (A \mrarr B) \MARR \WCIRCW \bdiab B$
\UI$\WCIRCW \bboxb (A \mrarr B) \MANDOR \wboxw \bdiab A \fCenter \WCIRCW \bdiab B$
\UI$\WCIRCW \bboxb (A \mrarr B) \MANDOR \wboxw \bdiab A \fCenter \wboxw \bdiab B$
\UI$\wboxw \bdiab A \MANDOR \WCIRCW \bboxb (A \mrarr B) \fCenter \wboxw \bdiab B$
\UI$\WCIRCW \bboxb (A \mrarr B) \fCenter \wboxw \bdiab A \MARR \wboxw \bdiab B$
\UI$\wdiaw \bboxb (A \mrarr B) \fCenter \wboxw \bdiab A \MARR \wboxw \bdiab B$
\UI$\wdiaw \bboxb (A \mrarr B) \fCenter \wboxw \bdiab A \mrarr \wboxw \bdiab B$
\DisplayProof
\end{tabular}
 }
\end{center}

\begin{itemize}
\item[A9.] $\wn \mbot \dashv\vdash \mbot$
 \ and \ A10. \ $\mbot \vdash \wn A$
\end{itemize}

\begin{center}
{\fns 
\begin{tabular}{ccc}
\AX$\mbot \fCenter \MTOPBOT$
\RightLabel{\fns $W_m$}
\UI$\mbot \fCenter \MTOPBOT \MANDOR \WCIRCW \bdiab \mbot$
\UI$\mbot \fCenter \MTOPBOT \MANDOR \WCIRCW \bdiab \mbot$
\UI$\mbot \fCenter \WCIRCW \bdiab \mbot \MANDOR \MTOPBOT$
\UI$\mbot \fCenter \WCIRCW \bdiab \mbot$
\UI$\mbot \fCenter \wboxw \bdiab \mbot$
\RightLabel{\fns def}
\UI$\mbot \fCenter \wn \mbot$
\DisplayProof
 & 
\AX$\mbot \fCenter \MTOPBOT$
\RightLabel{\fns nec}
\UI$\mbot \fCenter \WCIRCW \KTOPBOTK$
\UI$\BCIRCB \mbot \fCenter \KTOPBOTK$
\UI$\bdiab \mbot \fCenter \KTOPBOTK$
\UI$\wboxw \bdiab \mbot \fCenter \WCIRCW \KTOPBOTK$
\LeftLabel{\fns def}
\UI$\oc \mbot \fCenter \WCIRCW \KTOPBOTK$
\RightLabel{\fns conec}
\UI$\oc \mbot \fCenter \MTOPBOT$
\UI$\oc \mbot \fCenter \mbot$
\DisplayProof
 & 
\AX$\mbot \fCenter \MTOPBOT$
\RightLabel{\fns $W_m$}
\UI$\mbot \fCenter \MTOPBOT \MANDOR \WCIRCW \bdiab A$
\UI$\mbot \fCenter \MTOPBOT \MANDOR \WCIRCW \bdiab A$
\UI$\mbot \fCenter \WCIRCW \bdiab A \MANDOR \MTOPBOT$
\UI$\mbot \fCenter \WCIRCW \bdiab A$
\UI$\mbot \fCenter \wboxw \bdiab A$
\RightLabel{\fns def}
\UI$\mbot \fCenter \wn A$
\DisplayProof
 \\
\end{tabular}
 }
\end{center}


\section{Comparing derivations in Girard's calculus and D.LL}
\label{sec:some derivations}
Earlier on, we have discussed that the multi-type approach makes it possible to design calculi particularly suitable as tools of analysis. In particular, the calculus D.LL was introduced in the present paper with the aim of understanding the interaction between the additive and the multiplicative connectives mediated by the exponentials, as is encoded by the following derivable sequents:

\begin{center}
\begin{tabular}{l | l}
$\oc \aatop \dashv\vdash \mtop$ & $\wn \abot \dashv\vdash \mbot$ \\
$\oc A \mand \oc B \dashv\vdash \oc (A \aand B)$ & $\wn A \mor \wn B \dashv\vdash \wn (A \aor B)$ \\
\end{tabular}
\end{center}
In what follows, we compare the derivations of these sequents in Girard's calculus and in D.LL
Notice that the introduction via {\em controlled} weakening (emphasized by the label $W$) in the  derivation of $\oc \aatop \fCenter \mtop$ in Girard's calculus corresponds to the application of the {\em unrestricted} pure $\mathsf{K}_{\oc}$-type weakening rule $W_\oc$ in the D.LL derivation. Moreover, the right promotion with empty context in the  derivation of $\mtop \fCenter \oc \aatop$ in Girard's calculus is encoded in a sequence of  standard display and introduction rules, elicited by the rules nec and conec.

\begin{center}
\begin{tabular}{c}
\mc{1}{c}{derivations in Girard's calculus} \\
  \\
\AX$ \fCenter \mtop$
\LeftLabel{\fns $W$}
\UI$\oc \aatop \fCenter \mtop$
\DisplayProof

 \ 

\AX$ \fCenter \aatop$
\UI$ \fCenter \oc \aatop$
\UI$\mtop \fCenter \oc \aatop$
\DisplayProof
 \\
\end{tabular}
\end{center}

\begin{center}
\begin{tabular}{c}
\mc{1}{c}{derivations in D.LL} \\
 \\
\!\!\!\!\!\!\!\!\!\!\!\!\!\!\!\!\!\!\!\!\!\!
\AX$\MTOPBOT \fCenter \mtop$
\LeftLabel{\fns nec}
\UI$\WCIRCW \KTOPBOTK \fCenter \mtop$
\UI$\KTOPBOTK \fCenter \BCIRCB \mtop$
\LeftLabel{\fns $W_\oc$}
\UI$\KTOPBOTK \KANDORK \bboxb \aatop \fCenter \BCIRCB \mtop$
\UI$\WCIRCW (\KTOPBOTK \KANDORK \bboxb \aatop) \fCenter \mtop$
\LeftLabel{\fns coreg$_{\oc m}$}
\UI$\WCIRCW \KTOPBOTK \MANDOR \WCIRCW \bboxb \aatop \fCenter \mtop$
\UI$\WCIRCW \bboxb \aatop \MANDOR \WCIRCW \KTOPBOTK \fCenter \mtop$
\UI$\WCIRCW \KTOPBOTK \fCenter \WCIRCW \bboxb \aatop \MARR \mtop$
\LeftLabel{\fns conec}
\UI$\MTOPBOT \fCenter \WCIRCW \bboxb \aatop \MARR \mtop$
\UI$\WCIRCW \bboxb \aatop \MANDOR \MTOPBOT \fCenter \mtop$
\UI$\MTOPBOT \MANDOR \WCIRCW \bboxb \aatop \fCenter \mtop$
\LeftLabel{\fns $\MTOPBOT$}
\UI$\WCIRCW \bboxb \aatop \fCenter \mtop$
\UI$\wdiaw \bboxb \aatop \fCenter \mtop$
\LeftLabel{\fns def}
\UI$\oc \aatop \fCenter \mtop$
\DisplayProof
 \ 
\AX$\ATOPBOT \fCenter \aatop$
\LeftLabel{\fns $W_a$}
\UI$\ATOPBOT \AANDOR \MTOPBOT \fCenter \aatop$
\UI$\MTOPBOT \fCenter \aatop$
\LeftLabel{\fns nec}
\UI$\WCIRCW \KTOPBOTK \fCenter \aatop$
\UI$\KTOPBOTK \fCenter \BCIRCB \aatop$
\UI$\KTOPBOTK \fCenter \bboxb \aatop$
\UI$\WCIRCW \KTOPBOTK \fCenter \wdiaw \bboxb \aatop$
\RightLabel{\fns def}
\UI$\WCIRCW \KTOPBOTK \fCenter \oc \aatop$
\LeftLabel{\fns conec}
\UI$\MTOPBOT \fCenter \oc \aatop$
\UI$\mtop \fCenter \oc \aatop$
\DisplayProof
 \\
\end{tabular}
\end{center}

Below, notice that the  derivation of $\oc A \mand \oc B \vdash \oc (A \aand B)$ in Girard's calculus makes use of a {\em controlled} weakening (emphasized by the label $W$). Moreover, the right-introduction rule for $\aand$ internalizes a left-contraction (we emphasize the second observation by labelling the right introduction of $\aand$ with (C)).  The derivation of $\oc (A \aand B) \vdash \oc A \mand \oc B$ is somewhat dual to the derivation previously discussed, in that it makes use of a {\em controlled} contraction (emphasized by the label $C$), and the left-introduction rule for $\aand$ internalizes a left-weakening (we emphasize the second observation labelling the left-introduction of $\aand$ with (W)). 

\begin{center}
\begin{tabular}{cc}
\mc{2}{c}{derivations in Girard's calculus} \\
 & \\
\!\!\!\!\!\!\!\!\!\!\!\!\!\!\!
\AX$A \fCenter A$
\UI$\oc A \fCenter A$
\LeftLabel{\fns $W$}
\UI$\oc A \MANDOR \oc B \fCenter A$
\AX$B \fCenter B$
\UI$\oc B \fCenter B$
\LeftLabel{\fns $W$}
\UI$\oc A \MANDOR \oc B \fCenter B$
\LeftLabel{\fns $(C)$}
\BI$\oc A \MANDOR \oc B \fCenter A \aand B$
\UI$\oc A \MANDOR \oc B \fCenter \oc (A \aand B)$
\UI$\oc A \mand \oc B \fCenter \oc (A \aand B)$
\DisplayProof
 &
\AX$A \fCenter A$
\LeftLabel{\fns $(W)$}
\UI$A \aand B \fCenter A$
\UI$\oc (A \aand B) \fCenter A$
\UI$\oc (A \aand B) \fCenter \oc A$
\AX$B \fCenter B$
\LeftLabel{\fns $(W)$}
\UI$A \aand B \fCenter B$
\UI$\oc (A \aand B) \fCenter B$
\UI$\oc (A \aand B) \fCenter \oc B$
\BI$\oc (A \aand B) \MANDOR \oc (A \aand B) \fCenter \oc A \mand \oc B$
\LeftLabel{\fns $C$}
\UI$\oc (A \aand B) \fCenter \oc A \mand \oc B$
\DisplayProof
 \\
\end{tabular}
\end{center}
The  {\em controlled} weakening and contraction rules $W$ and $C$ applied above correspond, in the D.LL derivations of the same sequents below, to  the {\em unrestricted} pure $\mathsf{K}_{\oc}$-type weakening rule $W_\oc$ and  $C_\oc$; moreover, the internalized weakening and contraction in the additive introductions of $\aand$ are explicitly performed via the {\em unrestricted} pure $\mathsf{L}$-type rules $W_a$ and $C_a$.

\begin{center}
{\fns
\begin{tabular}{cc}
\mc{2}{c}{\normalsize derivations in D.LL} \\
 & \\

\!\!\!\!\!\!\!\!\!\!\!\!\!\!\!\!\!\!\!\!\!\!\!\!\!\!\!\!\!\!\!\!\!\!\!\!\!\!\!\!\!\!\!\!\!

\AX$A \fCenter A$
\UI$\bboxb A \fCenter \BCIRCB A$
\LeftLabel{\fns $W_\oc$}
\UI$\bboxb A \KANDORK \bboxb B \fCenter \BCIRCB A$
\UI$\WCIRCW (\bboxb A \KANDORK \bboxb B) \fCenter A$
\LeftLabel{\fns coreg$_{\oc m}$}
\UI$\WCIRCW \bboxb A \MANDOR \WCIRCW \bboxb B \fCenter A$

\AX$B \fCenter B$
\UI$\bboxb B \fCenter \BCIRCB B$
\LeftLabel{\fns $W_\oc$}
\UI$\bboxb B \KANDORK \bboxb A \fCenter \BCIRCB B$
\UI$\bboxb A \KANDORK \bboxb B \fCenter \BCIRCB B$
\UI$\WCIRCW (\bboxb A \KANDORK \bboxb B) \fCenter B$
\LeftLabel{\fns coreg$_{\oc m}$}
\UI$\WCIRCW \bboxb A \MANDOR \WCIRCW \bboxb B \fCenter A$
\BI$(\WCIRCW \bboxb A \MANDOR \WCIRCW \bboxb B) \AANDOR (\WCIRCW \bboxb A \MANDOR \WCIRCW \bboxb B) \fCenter A \aand B$
\LeftLabel{\fns $C_a$}
\UI$\WCIRCW \bboxb A \MANDOR \WCIRCW \bboxb B \fCenter A \aand B$
\LeftLabel{\fns reg$_{\oc m}$}
\UI$\WCIRCW (\bboxb A \KANDORK \bboxb B) \fCenter A \aand B$
\UI$\bboxb A \KANDORK \bboxb B \fCenter \BCIRCB A \aand B$
\UI$\bboxb A \KANDORK \bboxb B \fCenter \bboxb (A \aand B)$
\UI$\WCIRCW (\bboxb A \KANDORK \bboxb B) \fCenter \wdiaw \bboxb (A \aand B)$
\RightLabel{\fns def}
\UI$\WCIRCW (\bboxb A \KANDORK \bboxb B) \fCenter \oc (A \aand B)$
\LeftLabel{\fns coreg$_{\oc m}$}
\UI$\WCIRCW \bboxb A \MANDOR \WCIRCW \bboxb B \fCenter \oc (A \aand B)$
\UI$\WCIRCW \bboxb B \fCenter \WCIRCW \bboxb A \MARR \oc (A \aand B)$
\UI$\wdiaw \bboxb B \fCenter \WCIRCW \bboxb A \MARR \oc (A \aand B)$
\LeftLabel{\fns def}
\UI$\oc B \fCenter \WCIRCW \bboxb A \MARR \oc (A \aand B)$
\UI$\WCIRCW \bboxb A \MANDOR \oc B \fCenter \oc (A \aand B)$
\UI$\oc B \MANDOR \WCIRCW \bboxb A \fCenter \oc (A \aand B)$
\UI$\WCIRCW \bboxb A \fCenter \oc B \MARR \oc (A \aand B)$
\UI$\wdiaw \bboxb A \fCenter \oc B \MARR \oc (A \aand B)$
\LeftLabel{\fns def}
\UI$\oc A \fCenter \oc B \MARR \oc (A \aand B)$
\UI$\oc B \MANDOR \oc A \fCenter \oc (A \aand B)$
\UI$\oc A \MANDOR \oc B \fCenter \oc (A \aand B)$
\UI$\oc A \mand \oc B \fCenter \oc (A \aand B)$
\DisplayProof

 &
\!\!\!\!\!\!\!\!\!\!\!\!\!\!\!\!\!\!\!\!\!\!\!\!
\AX$A \fCenter A$
\LeftLabel{\fns $W_a$}
\UI$A \AANDOR B \fCenter A$
\UI$A \aand B \fCenter A$
\UI$\bboxb (A \aand B) \fCenter \BCIRCB A$
\UI$\bboxb (A \aand B) \fCenter \bboxb A$
\UI$\WCIRCW \bboxb (A \aand B) \fCenter \wdiaw \bboxb A$
\RightLabel{\fns def}
\UI$\WCIRCW \bboxb (A \aand B) \fCenter \oc A$

\AX$B \fCenter B$
\LeftLabel{\fns $W_a$}
\UI$B \AANDOR A \fCenter B$
\UI$A \AANDOR B \fCenter B$
\UI$A \aand B \fCenter B$
\UI$\bboxb (A \aand B) \fCenter \BCIRCB B$
\UI$\bboxb (A \aand B) \fCenter \bboxb B$
\UI$\WCIRCW \bboxb (A \aand B) \fCenter \wdiaw \bboxb B$
\RightLabel{\fns def}
\UI$\WCIRCW \bboxb (A \aand B) \fCenter \oc B$
\BI$\WCIRCW \bboxb (A \aand B) \MANDOR \WCIRCW \bboxb (A \aand B) \fCenter \oc A \mand \oc B$
\LeftLabel{\fns reg$_{\oc m}$}
\UI$\WCIRCW (\bboxb (A \aand B) \KANDORK \bboxb (A \aand B)) \fCenter \oc A \mand \oc B$
\UI$\bboxb (A \aand B) \KANDORK \bboxb (A \aand B) \fCenter \BCIRCB \oc A \mand \oc B$
\LeftLabel{\fns $C_\oc$}
\UI$\bboxb (A \aand B) \fCenter \BCIRCB \oc A \mand \oc B$
\UI$\WCIRCW \bboxb (A \aand B) \fCenter \oc A \mand \oc B$
\UI$\wdiaw \bboxb (A \aand B) \fCenter \oc A \mand \oc B$
\LeftLabel{\fns def}
\UI$\oc (A \aand B) \fCenter \oc A \mand \oc B$
\DisplayProof
 \\
\end{tabular}
 }
\end{center}
\section{A sneak preview: proper display calculi for general linear logic}
\label{sec:sneak}
In the present paper, we have mainly focused on the setting of {\em distributive} linear logic, i.e.~we have assumed that $\aand$ distributes over $\aor$. However, thanks to the modularity of proper multi-type calculi, the versions of linear logic in which this assumption is dropped are properly displayable in suitable multi-type settings in which additives are treated according to what done in \cite{GrecoPalmigianoLatticeLogic}.  Linear logic formulas can be encoded in the language of such calculi via certain (positional) translations. In what follows, we do not expand on their precise definition.\footnote{The new heterogeneous connectives $\pbboxbp_r $, $\pwdiawp_r $, and  $\pbdiabp_\ell$, $\pwboxwp_\ell $ bridge the $\mathsf{Linear}$ type with two new types, algebraically interpreted as distributive lattices. We refer to \cite{GrecoPalmigianoLatticeLogic} for more details. The only properties which are relevant to the present section are the usual ones of adjoint normal modalities.}

\begin{center}
\begin{tabular}{rcl}
$\oc A \mand \oc B \vdash \oc (A \aand B)$ & $\rightsquigarrow$ & $\wdiaw \bboxb A \mand \wdiaw \bboxb B \vdash \wdiaw \bboxb \Big(\pbboxbp_r (\pwdiawp_r A \pandp \pwdiawp_r B)\Big)$ \\
$\oc (A \aand B) \vdash \oc A \mand \oc B$ & $\rightsquigarrow$ & $\wdiaw \bboxb \Big(\pbdiabp_\ell (\pwboxwp_\ell A \pandp \pwboxwp_\ell B)\Big) \vdash \wdiaw \bboxb A \mand \wdiaw \bboxb B$. \\
\end{tabular}
\end{center}

Below, we derive the sequents $\oc A \mand \oc B \dashv\vdash \oc (A \aand B)$  in a non-distributive version of D.LL.  Also in this setting, the implicit applications of weakening and contraction internalized by the additive introductions of $\aand$ are now explicitly performed via the {\em unrestricted} weakening and contraction rules $W_\ell$ and $C_r$ respectively, which pertain to the new types.

\begin{center}
{\fns
\begin{tabular}{cc}
\!\!\!\!\!\!\!\!\!\!\!\!\!\!\!\!\!\!\!\!\!\!\!\!\!\!\!\!\!\!\!\!\!\!\!\!\!\!\!\!\!\!\!\!\!\!\!\!\!\!\!\!\!\!\!\!\!\!\!\!\!\!
\AX$A \fCenter A$
\UI$\PWCIRCWP A \fCenter \pwdiawp A$
\UI$A \fCenter \PBCIRCBP \pwdiawp A$
\UI$\bboxb A \fCenter \BCIRCB \PBCIRCBP \pwdiawp A$
\LeftLabel{\fns $W_\oc$}
\UI$\bboxb A \KANDORK \bboxb B \fCenter \BCIRCB \PBCIRCBP \pwdiawp A$

\UI$\WCIRCW (\bboxb A \KANDORK \bboxb B) \fCenter \PBCIRCBP \pwdiawp A$
\LeftLabel{\fns coreg$_{\oc m}$}
\UI$\WCIRCW \bboxb A \MANDOR \WCIRCW \bboxb B \fCenter \PBCIRCBP \pwdiawp A$
\UI$\PWCIRCWP (\WCIRCW \bboxb A \MANDOR \WCIRCW \bboxb B) \fCenter \pwdiawp A$
\AX$B \fCenter B$
\UI$\PWCIRCWP B \fCenter \pwdiawp B$
\UI$B \fCenter \PBCIRCBP \pwdiawp B$
\UI$\bboxb B \fCenter \BCIRCB \PBCIRCBP \pwdiawp B$
\LeftLabel{\fns $W_\oc$}
\UI$\bboxb B \KANDORK \bboxb A \fCenter \BCIRCB \PBCIRCBP \pwdiawp B$
\UI$\bboxb A \KANDORK \bboxb B \fCenter \BCIRCB \PBCIRCBP \pwdiawp B$

\UI$\WCIRCW (\bboxb A \KANDORK \bboxb B) \fCenter \PBCIRCBP \pwdiawp B$
\LeftLabel{\fns coreg$_{\oc m}$}
\UI$\WCIRCW \bboxb A \MANDOR \WCIRCW \bboxb B \fCenter \PBCIRCBP \pwdiawp B$

\UI$\PWCIRCWP (\WCIRCW \bboxb A \MANDOR \WCIRCW \bboxb B) \fCenter \pwdiawp B$
\BI$\PWCIRCWP (\WCIRCW \bboxb A \MANDOR \WCIRCW \bboxb B) \PANDORP \PWCIRCWP (\WCIRCW \bboxb A \MANDOR \WCIRCW \bboxb B) \fCenter \pwdiawp A \pandp \pwdiawp B$
\LeftLabel{\fns $C_r$}
\UI$\PWCIRCWP (\WCIRCW \bboxb A \MANDOR \WCIRCW \bboxb B) \fCenter \pwdiawp A \pandp \pwdiawp B$
\UI$\WCIRCW \bboxb A \MANDOR \WCIRCW \bboxb B \fCenter \PBCIRCBP \pwdiawp A \pandp \pwdiawp B$
\UI$\WCIRCW \bboxb A \MANDOR \WCIRCW \bboxb B \fCenter \pbboxbp (\pwdiawp A \pandp \pwdiawp B)$
\RightLabel{\fns def}
\UI$\WCIRCW \bboxb A \MANDOR \WCIRCW \bboxb B \fCenter A \aand B$

\LeftLabel{\fns reg$_{\oc m}$}
\UI$\WCIRCW (\bboxb A \KANDORK \bboxb B) \fCenter A \aand B$
\UI$\bboxb A \KANDORK \bboxb B \fCenter \BCIRCB A \aand B$
\UI$\bboxb A \KANDORK \bboxb B \fCenter \bboxb (A \aand B)$
\UI$\WCIRCW (\bboxb A \KANDORK \bboxb B) \fCenter \wdiaw \bboxb (A \aand B)$
\RightLabel{\fns def}
\UI$\WCIRCW (\bboxb A \KANDORK \bboxb B) \fCenter \oc (A \aand B)$
\LeftLabel{\fns coreg$_{\oc m}$}
\UI$\WCIRCW \bboxb A \MANDOR \WCIRCW \bboxb B \fCenter \oc (A \aand B)$
\UI$\WCIRCW \bboxb B \fCenter \WCIRCW \bboxb A \MARR \oc (A \aand B)$
\UI$\wdiaw \bboxb B \fCenter \WCIRCW \bboxb A \MARR \oc (A \aand B)$
\LeftLabel{\fns def}
\UI$\oc B \fCenter \WCIRCW \bboxb A \MARR \oc (A \aand B)$
\UI$\WCIRCW \bboxb A \MANDOR \oc B \fCenter \oc (A \aand B)$
\UI$\oc B \MANDOR \WCIRCW \bboxb A \fCenter \oc (A \aand B)$
\UI$\WCIRCW \bboxb A \fCenter \oc B \MARR \oc (A \aand B)$
\UI$\wdiaw \bboxb A \fCenter \oc B \MARR \oc (A \aand B)$
\LeftLabel{\fns def}
\UI$\oc A \fCenter \oc B \MARR \oc (A \aand B)$
\UI$\oc B \MANDOR \oc A \fCenter \oc (A \aand B)$
\UI$\oc A \MANDOR \oc B \fCenter \oc (A \aand B)$
\UI$\oc A \mand \oc B \fCenter \oc (A \aand B)$
\DisplayProof

 & 
\!\!\!\!\!\!\!\!\!\!\!\!\!\!\!\!\!\!\!\!\!\!\!\!\!\!\!\!\!\!\!\!\!\!\!\!\!\!\!\!\!\!
\AX$A \fCenter A$
\UI$\pwbox A \fCenter \PWCIRC A$
\LeftLabel{\fns $W_\ell$}
\UI$\pwbox A \PANDOR \pwbox B \fCenter \PWCIRC A$
\UI$\pwbox A \pand \pwbox B \fCenter \PWCIRC A$
\UI$\PBCIRC \pwbox A \pand \pwbox B \fCenter A$
\UI$\pbdia (\pwbox A \pand \pwbox B) \fCenter A$
\LeftLabel{\fns def}
\UI$A \aand B \fCenter A$
\UI$\bboxb (A \aand B) \fCenter \BCIRCB A$
\UI$\bboxb (A \aand B) \fCenter \bboxb A$
\UI$\WCIRCW \bboxb (A \aand B) \fCenter \wdiaw \bboxb A$
\RightLabel{\fns def}
\UI$\WCIRCW \bboxb (A \aand B) \fCenter \oc A$

\AX$B \fCenter B$
\UI$\pwbox B \fCenter \PWCIRC B$
\LeftLabel{\fns $W_\ell$}
\UI$\pwbox B \PANDOR \pwbox A \fCenter \PWCIRC B$
\LeftLabel{\fns $E_\ell$}
\UI$\pwbox A \PANDOR \pwbox B \fCenter \PWCIRC B$
\UI$\pwbox A \pand \pwbox B \fCenter \PWCIRC B$
\UI$\PBCIRC \pwbox A \pand \pwbox B \fCenter B$
\UI$\pbdia (\pwbox A \pand \pwbox B) \fCenter B$
\LeftLabel{\fns def}
\UI$A \aand B \fCenter B$
\UI$\bboxb (A \aand B) \fCenter \BCIRCB B$
\UI$\bboxb (A \aand B) \fCenter \bboxb B$
\UI$\WCIRCW \bboxb (A \aand B) \fCenter \wdiaw \bboxb B$
\RightLabel{\fns def}
\UI$\WCIRCW \bboxb (A \aand B) \fCenter \oc B$
\BI$\WCIRCW \bboxb (A \aand B) \MANDOR \WCIRCW \bboxb (A \aand B) \fCenter \oc A \mand \oc B$
\LeftLabel{\fns reg$_{\oc m}$}
\UI$\WCIRCW (\bboxb (A \aand B) \KANDORK \bboxb (A \aand B)) \fCenter \oc A \mand \oc B$
\UI$\bboxb (A \aand B) \KANDORK \bboxb (A \aand B) \fCenter \BCIRCB \oc A \mand \oc B$
\LeftLabel{\fns $C_\oc$}
\UI$\bboxb (A \aand B) \fCenter \BCIRCB \oc A \mand \oc B$
\UI$\WCIRCW \bboxb (A \aand B) \fCenter \oc A \mand \oc B$
\UI$\wdiaw \bboxb (A \aand B) \fCenter \oc A \mand \oc B$
\LeftLabel{\fns def}
\UI$\oc (A \aand B) \fCenter \oc A \mand \oc B$
\DisplayProof
 \\
\end{tabular}
 }
\end{center}

\end{document}